\documentclass[11pt]{amsart}
\pdfoutput=1
\usepackage[OT2,T1]{fontenc}

\usepackage{lmodern}
\usepackage{ifthen}
\usepackage{amsfonts}
\usepackage{amsxtra}
\usepackage{amssymb}
\usepackage{array}
\usepackage{eucal}
\usepackage[margin=1.0in]{geometry}
\usepackage{xcolor}
\definecolor{cite}{rgb}{0.30,0.60,1.00}
\definecolor{url}{rgb}{0.00,0.00,0.80}
\definecolor{link}{rgb}{0.40,0.10,0.20}
\usepackage[colorlinks,linkcolor=link,urlcolor=url,citecolor=cite,pagebackref,breaklinks]{hyperref}
\usepackage{lscape}
\usepackage{bbm}
\usepackage{mathtools}
\usepackage{mathrsfs}
\usepackage{appendix}
\usepackage[all]{xy}
\usepackage[lite,abbrev,msc-links,alphabetic]{amsrefs}

\usepackage{graphicx}
\usepackage{multirow}
\usepackage{pstricks}
\usepackage{pst-pdf}

\usepackage{enumitem}

\numberwithin{equation}{section}

\theoremstyle{plain}
\newtheorem{proposition}{Proposition}[section]

\newtheorem{corollary}[proposition]{Corollary}
\newtheorem{lem}[proposition]{Lemma}
\newtheorem{theorem}[proposition]{Theorem}

\theoremstyle{definition}
\newtheorem{definition}[proposition]{Definition}

\newtheorem{notation}[proposition]{Notation}

\theoremstyle{remark}
\newtheorem{remark}[proposition]{Remark}

\renewcommand{\c}[1]{\mathcal{#1}}
\renewcommand{\d}[1]{\mathbb{#1}}
\newcommand{\f}[1]{\mathfrak{#1}}
\renewcommand{\r}[1]{\mathrm{#1}}
\newcommand{\s}[1]{\mathscr{#1}}
\renewcommand{\sf}[1]{\mathsf{#1}}

\newcommand{\ol}{\overline}

\newcommand{\ang}[1]{\langle{#1}\rangle}

\newcommand{\cA}{\c A}
\newcommand{\cB}{\c B}

\newcommand{\cD}{\c D}
\newcommand{\cE}{\c E}

\newcommand{\cH}{\c H}

\newcommand{\cN}{\c N}
\newcommand{\cO}{\c O}
\newcommand{\cP}{\c P}

\newcommand{\cS}{\c S}
\newcommand{\cT}{\c T}

\newcommand{\cV}{\c V}

\newcommand{\dA}{\d A}

\newcommand{\dC}{\d C}

\newcommand{\dF}{\d F}
\newcommand{\dG}{\d G}

\newcommand{\dN}{\d N}

\newcommand{\dP}{\d P}
\newcommand{\dQ}{\d Q}
\newcommand{\dR}{\d R}
\newcommand{\dS}{\d S}
\newcommand{\dT}{\d T}

\newcommand{\dZ}{\d Z}

\newcommand{\fc}{\f c}

\newcommand{\fg}{\f g}

\newcommand{\fm}{\f m}

\newcommand{\rE}{\r E}

\newcommand{\rH}{\r H}

\newcommand{\sC}{\s C}

\newcommand{\sV}{\s V}

\newcommand{\tD}{\mathtt{D}}

\newcommand{\tF}{\mathtt{F}}

\newcommand{\tV}{\mathtt{V}}

\newcommand{\xra}{\xrightarrow}

\newcommand{\GU}{\mathrm{GU}}

\newcommand{\ac}{\r{ac}}

\newcommand{\et}{\acute{\r{e}}\r{t}}

\newcommand{\Iw}{\mathrm{Iw}}

\newcommand{\St}{\mathrm{St}}
\newcommand{\BS}{\mathrm{BS}}
\newcommand{\BB}{\mathrm{BB}}
\newcommand{\IH}{\mathrm{IH}}

\newcommand{\BC}{\mathrm{BC}}

\newcommand{\Kp}{  K^p }
\newcommand{\bKp}{  (K^p) }
\newcommand{\Fpp}{ { \mathbb{F}_{p^2} }}
\newcommand{\Fpac}{ { \mathbb{F}_{p}^{\mathrm{ac}} }}
\newcommand{\spe}{\mathrm{sp}}

\DeclareMathOperator{\rank}{rank}

\DeclareMathOperator{\diag}{diag}

\DeclareMathOperator{\End}{End}

\DeclareMathOperator{\Fil}{Fil}
\DeclareMathOperator{\Fr}{Fr}

\DeclareMathOperator{\Frob}{Frob}

\DeclareMathOperator{\Gal}{Gal}

\DeclareMathOperator{\GL}{GL}
\DeclareMathOperator{\Gr}{Gr}
\DeclareMathOperator{\Hom}{Hom}

\DeclareMathOperator{\Ind}{Ind}

\DeclareMathOperator{\Lie}{Lie}
\DeclareMathOperator{\Map}{Map}

\DeclareMathOperator{\Res}{Res}

\DeclareMathOperator{\Sh}{Sh}

\DeclareMathOperator{\Spec}{Spec}

\DeclareMathOperator{\Tr}{Tr}

\DeclareMathOperator{\val}{val}

\DeclareMathOperator{\id}{id}
\DeclareMathOperator{\im}{im}

\DeclareMathOperator{\HdR}{H_1^{\r{dR} } }

\begin{document}

	\title[]
	{Level lowering for GU(1,2)}

	\author{Hao Fu}
	\address{Universit\'{e} de Strasbourg, Strasbourg, France}
	\email{hao.fu@math.unistra.fr}

	\date{\today}

	\begin{abstract}
		Mazur's principle gives a criterion under which an irreducible 
				mod $\ell$ Galois representation arising from a modular form of level $Np$ (with $p$ prime to $N$) can also arise from a modular form of level $N.$
		We prove an analogous result  showing that   
		a mod $\ell$ Galois representation arising from a stable cuspidal automorphic representation of the unitary similitude group $G=\r{GU}(1,2)$ which is Steinberg at an inert prime $p$ can also arise from an automorphic representation of $G$ that is unramified at $p$.
	\end{abstract}

	\maketitle
	
	\tableofcontents
	
	\section{Introduction} \label{sec:notation}
	
	The level lowering problem was proposed by Serre in \cite{Ser87a,Ser87b} in the name of epsilon conjecture and served as a key step in deducing Fermat last theorem from Shimura-Taniyama-Weil conjecture. 
Ribet proved the following theorem, which he called also Mazur's principle.
\begin{theorem}\cite[Theorem 1.1]{Rib90}
Let $N$ be a positive integer and let $p,\ell$ be  prime numbers such that  $\ell$ is odd and $(p,N)=1.$ 
Let $f$ be a newform of weight 2 and level $Np$ and $\ol\rho_{f,\ell}$ be the mod $\ell$ residual Galois representation attached to $f$. Suppose 
\begin{enumerate}
\item $\ol\rho_{f,\ell}$ is absolutely irreducible;
\item $\ol\rho_{f,\ell}$ is unramified at $p$;
\item  $p \not \equiv 1 \bmod \ell$.
\end{enumerate}
Then there exists a newform $g$ of weight 2 and level $N$ such that $\ol\rho_{f,\ell}\cong \ol\rho_{g,\ell}.$
  \end{theorem}
  In his original proof, Ribet embedded the given Galois representation
  into some torsion module of the Jacobian of a modular curve. A key step
  is to analyze the Frobenius action on the toric part of Jacobians. 
 The assumption   $p \not \equiv 1 \bmod \ell$ was removed by Ribet later in \cite{Rib91}, where he took another prime number
 $q$ such that $q\not\equiv 1 \bmod \ell$ and transferred the given modular form to the one attached to
 the indefinite quaternion algebra ramified at $pq$ by Jacquet-Langlands correspondence. Then the so-called $(p,q)$ switch trick allows him to lower the level at $p$
while by Mazur's principle he can further lower the level at $q$. For a more precise explanation of Ribet's method, see \cite{Wan22}.
   
     Later Jarvis (\cite{Jar99}) and Rajaei (\cite{Raj01}) proved  similar results on level lowering of Galois representations attached to Shimura curves over totally real fields 
after a major breakthrough by Carayol in \cite{Car86}. The geometry of bad reduction of Shimura curve combined with an explicit calculation of nearby cycles shows the component group of the Jacobian of the Shimura curve is Eisenstein. 
    Along the same line van Hoften (\cite{vH21}) and Wang (\cite{Wan22}) studied level lowering for Siegel modular threefold  of paramodular level under different technical assumptions. 
	 For unitary similitude group of signature(1,2), Helm proved level lowering at a place split in the quadratic imaginary extension over a totally real field in \cite{Hel06}.   Boyer treated the case for unitary Shimura varieties
	of  Kottwitz-Harris-Taylor type in   \cite{Boy19}.

In this article we deal with level lowering at a rational prime inert in a quadratic imaginary extension for the unitary similitude group of signature (1,2). 
	 
Let $F$ be a quadratic imaginary extension over $\dQ$
and $G:=\GU(1,2)$ be the corresponding quasi-split unitary similitude group of signature (1,2).
 	Fix a prime number $p$ inert in $F$  
		and an open compact subgroup $K^p$ of $G(\dA^{\infty,p})$
		where $\dA^{\infty,p}$ is the finite ad\`ele over $\dQ$ outside $p.$
		Let $K_p\subset G(\dQ_p)$ be a hyperspecial subgroup, and $\Iw_p\subset K_p$ be an Iwahoric subgroup. 
			 Let $S$ (resp. $S_0(p)$) be the integral model of Shimura variety attached to $G$ of 
			 level $K^pK_p$ (resp. $K^p\Iw_p$).
	The main theorem is 
	\begin{theorem}[Theorem \ref{thm:main}]
		Let $\pi$ be a stable automorphic cuspidal representation of $G(\dA)$  cohomological with trivial coefficient. 
							Fix a prime number $\ell\neq p.$
		Let $\fm$ be the mod $\ell$   maximal ideal of the spherical Hecke algebra attached to 
		$\pi.$
		Let $\ol\rho_{\pi,\ell}$ be the mod $\ell$ Galois representation attached to $\pi.$ 
		Suppose 
		\begin{enumerate}
		\item  $(\pi^{\infty, p})^{K^p}\neq 0;$ 
					\item $ \pi_{p}$ is the Steinberg representation of $G_p$ twisted by an unramified character;
			\item 			 if $i\neq 2$ then $\rH^i(S\otimes{F^\ac},\dF_\ell)_\fm = 0;$
     \item $\ol{\rho}_{\pi,\ell }$ is absolutely irreducible;
				\item
				$\ol\rho_{\pi,\ell}$ is unramified at $p$;
				\item $ \ell\nmid(p-1)(p^3+1).$
			\end{enumerate} 
		Then there exists a cuspidal automorphic  representation $\tilde \pi$ of $G(\dA)$  such that $(\tilde\pi^\infty)^{K^p K_p}\neq 0$ and 
		\[ \ol{\rho}_{\pi,\ell} \cong  \ol{\rho}_{\tilde\pi,\ell}.\]
	\end{theorem} 
 	  We adapt Ribet's strategy.  
As Jacobian is unavailable for Shimura surfaces, inspired by Helm we use weight-monodromy  spectral sequence to analyze analogues of the component group
 of Jacobians of $S$ and $S_0(p).$
  In order to do so, we need a detailed study on the geometry of special fibers.
	The surface $S\otimes \Fpp$ was studied by Wedhorn in \cite{Wed01} and Vollaard in \cite{Vol10}.
	They showed that the   supersingular locus consists of geometric irreducible components which are Fermat curves of degree $p+1$ intersecting transversally at superspecial points. The complement of supersingular locus is $\mu$-ordinary locus which is dense.
	
	The geometry of $S_0(p)$ is more complicated.
	The study of local models in \cite{Bel02} implies that $S_0(p)$ has semistable reduction at $p$. 
    We define three closed strata $Y_0,Y_1,Y_2$ in ${S_0(p)}\otimes {\Fpp}.$  
We show they are all smooth and their union is  ${S_0(p)}\otimes {\Fpp}.$	 
	We further study  relations between these strata and $S\otimes \Fpp.$
	In particular,
	$Y_0$ is isomorphic to the blowup of $S\otimes\Fpp$ at superspecial points;
	$Y_1$ admits a purely inseparable morphism to the latter;
	and $Y_2$ is a $\dP^1$-bundle over the normalization of the supersingular locus of $S\otimes \Fpp$
	which is geometrically a disjoint union of Fermat curves.	
		The pairwise intersections $Y_i \cap Y_j$ are transversal and
   parameterized by discrete Shimura varieties attached to  $G',$
   where $G'$ is the unique inner form of $G$ which coincides with $G$
  at all finite places and is compact modulo center at infinity.
  This can be viewed as  a geometric incarnation of Jacquet-Langlands transfer.
	 Moreover, we show the geometric points of $Y_0\cap Y_1\cap Y_2$ 
 are in bijection with the discrete Shimura variety attached to $G'$ of level $K^p\Iw_p.$ All the morphisms are equivariant under prime-to-$p$ Hecke correspondence, and defined over $\Fpp$ thus compatible with the Frobenius action
 when taking the geometric fiber. The result bears a resemblance to those of   \cite{dSG18} and \cite{Vol10}, but is tailored for arithmetic applications by preserving Hecke equivariance and schematic structure.
 
By Matsushima's formula, the given automorphic representation $\pi$ contributes  to the intersection cohomology  of   Baily-Borel compactification of  $S_0(p).$ Fortunately, we can ignore the compactification since the cohomology of the boundary of Borel-Serre compactification vanishes when localized at $\fm$ 
by the irreducibility of the residual Galois representation.
We then write down the weight-monodromy spectral sequence for the surface $S_0(p).$ 
  
We are ready to prove the main theorem by contradiction. If there were no level lowering,  the torsion-free assumption would eliminate the possiblity that $\pi$ appears in 
the localized cohomology of $S \otimes \Fpac.$
The weight-monodromy spectral sequence would degenerate at the first page and give rise to a filtration
 of $\rH^2({S_0(p)}\otimes{F^\ac},\dF_\ell)_\fm$ with the graded pieces given by the cohomology groups of $Y_0\cap Y_1\cap Y_2.$
  The unramified condition on the residual Galois representation would force $\ol\rho_{\pi,\ell}$ to live in the localized cohomology of $ (Y_0\cap Y_1\cap Y_2)\otimes \Fpac.$ 
  We then find a contradiction by studying the generalized eigenvalues of the Frobenius action.
  
  The article is organized as follows:
  after introducing the relevant Shimura varieties  in Section \ref{sec:Shimura}, 
we study the geometry of special fiber of Shimura varieties in Section \ref{sec:geometry}.
Finally we carry out the proof of the main theorem 
   in Section \ref{sec:proof} .
   
	\subsection{Notation and conventions}
	\label{ss:notation}
	
	The following list contains basic notation and conventions we fix throughout the article, we will use them without further comments.
	\begin{itemize}
		\item We denote by $\dA$ the ring of ad\`eles over $\dQ$. For a set $\Box$ of places of $\dQ$, we denote by $\dA^\Box$ the ring of ad\`eles away from $\Box$. For a number field $F$, we put $\dA^\Box_F\coloneqq\dA^\Box\otimes_\dQ F$. If $\Box=\{v_1,\dots,v_n\}$ is a finite set, we will also write $\dA^{v_1,\dots,v_n}$ for $\dA^\Box$.
		
		\item For a field $K$, denote by $K^\ac$ the algebraic closure of $K$ and put $ G_K\coloneqq\Gal(K^\ac/K)$. Denote by $\dQ^\ac$ the algebraic closure of $\dQ$ in $\dC$. When $K$ is a subfield of $\dQ^\ac$, we take $G_K$ to be $\Gal(\dQ^\ac/K)$ hence a subgroup of $G_\dQ$.
		\item
		For every rational prime $p$, we fix an algebraic closure $\mathbb{Q}_{p}^\ac$ of $\mathbb{Q}_{p}$with the residue field $\mathbb{F}_{p}^\ac,$ and an isomorphism 
		$\iota_p:\dQ_{p}^\ac \cong \dC.$
			 \item
	 	For an algebraic group $G$ over $\dQ$, set $G_p:=G(\dQ_p)$ 
		for a rational prime $p$ and $G_\infty:=G(\dR).$
		\item Let $X$ be a scheme. The cohomology group $\rH^\bullet(X,-)$ will always be computed on the small \'{e}tale site of $X$. If $X$ is of finite type over a subfield of $\dC$, then $\rH^\bullet(X(\dC),-)$ will be understood as the Betti cohomology of the associated complex analytic space $X(\dC)$.
		
		\item
		Let $R$ be a ring. 
		Given two $R$-modules $M_1 \subset M_2$, and $s \in \dN$ an integer.
		denote by $M_1 \stackrel{s} \subset M_2$ if the length of the $R$-module  $M_2/M_1$ is $s$ (hence finite).
		\item
		Let $R$ be a ring and $M$ be a set. Denote by $R[M]$ the set of functions on $M$ with compact support with values in $R$. 
		\item If a base ring is not specified in the tensor operation $\otimes$, then it is $\mathbb{Z}$.
		\item
		For a scheme $S$ (resp. Noetherian scheme $S$ ), we denote by  
		$\mathsf{Sch}_{/ S}$ (resp.   $\mathsf{Sch}_{/ S}^{\prime}$ ) 
		the category of $S$-schemes (resp. locally Noetherian $S$-schemes). 
		If $S=\operatorname{Spec} R$ is affine, we also write
		$\mathsf{Sch}_{/ R}$ ( resp. $\mathsf{Sch}_{/ R}^{\prime} $ )
		for
		$\mathsf{Sch}_{/ S}$ ( resp. $\mathsf{Sch}_{/ S}^{\prime}$).
		\item
		The structure sheaf of a scheme $X$ is denoted by $\mathcal{O}_{X}$.
		\item
		For a scheme $X$ over an affine scheme Spec $R$ and an $R$-algebra $S$, we write $X \otimes_{R} S$ or even $X_{S}$ for $X \times_{\operatorname{Spec} R} \operatorname{Spec} S$.
		\item
		For a scheme $S$ in characteristic $p$ for some rational prime $p$, we denote by $\sigma: S \to S$ the absolute $p$-power Frobenius morphism. 
		For a perfect field $\kappa$ of characteristic $p$, we denote by $W(\kappa)$ its Witt ring, and by abuse of notation, $\sigma: W(\kappa) \to W(\kappa)$ the canonical lifting of the $p$-power Frobenius map.
		\item
		Denote by $\mathbb{P}^{1}$ the projective line scheme over $\mathbb{Z}$, and 
		$\mathbb{G}_{m,R}=\operatorname{Spec} R [T, T^{-1}]$ the multiplicative group scheme over a ring $R$. Let $\dS=\Res_{\dC/\dR} \dG_{m,\dC}$ be the Weil restriction of $\dG_{m,\dC}$ to $\dR$.
		
	\end{itemize}
\subsection{Acknowledgements}
    The author would like to express his deep gratitude to his Ph.D advisor Prof. Yichao Tian
    for suggesting this topic and the enlightening discussions, not to mention correcting the paper with great patience. 
    He also thanks Prof. Liang Xiao and Prof. Henri Carayol for the advice in finishing this paper.
    Finally, he would like to thank Lambert A'Campo, Ruiqi Bai, 
    Jiahao Niu,
    Matteo Tamiozzo, Zhiyu Zhang and Ruishen Zhao
     for the discussion. The paper is written at Institut de recherche math\'ematique avanc\'ee, Strasbourg and Morningside Center of 
    Mathematics, Beijing.
    
	\section{Shimura varieties, integral models and moduli interpretations}\label{sec:Shimura}
	In this section we introduce some Shimura varieties associated with the group of unitary similitudes.
	
	Let $F=\dQ(\sqrt{\Delta})$ be a quadratic imaginary extension of $\dQ$ with $\Delta \in \dZ$ a negative square-free element.
	Let $\fc$ be the nontrivial element in $\Gal(F/\dQ),$ and write $a^\fc$ or $\fc (a)$ for the action of $\fc$ on $a$ for $a\in F.$
	Fix an embedding $ \tau_0: F \to \mathbb{C} $
	such that 
	$ \tau_0 (\sqrt{\Delta}) \in \mathbb{R}_{>0} \cdot \sqrt{-1}.$
	Then $\Sigma_\infty:=\{\tau_0, \tau_1=\tau_0 \circ \fc \}$ is the set of all complex embeddings of $F.$ 
	Let $O_F$ be the ring of integers of $F$, $F^{ac}$ be an algebraic closure of $F$. 
	Let $(\Lambda=O_F^3,\psi)$ be the free $O_F$-module of rank 3 equipped with the hermitian form 
	$$
	\psi(u, v)={}^tuJ\bar{v}
	$$
	where
	$$
	\Phi=\begin{pmatrix} 
		& & 1 \\
		& -1 & \\
		1 & &
	\end{pmatrix}.
	$$
	Then $\psi$ is of signature (1,2) over $\dR$. Denote by $e_0 ,e_1 ,e_2 \in \Lambda$
	the standard basis vectors. We put also 
	\[
	\langle u, v\rangle_{\psi}:=\Tr_{F/\dQ}(\frac{1}{\sqrt{\Delta}}\psi(u,v))
	\]
	which is a non-degenerate alternating form $V\times V\to \dQ$. 
	Let $G= \operatorname{GU}(\Lambda,\psi)$ be the group of unitary similitudes defined over $\dZ$ by  
	$$
	G(R)=\{(g, \nu(g)) \in \mathrm{GL}_{O_F \otimes_\dZ R} (\Lambda \otimes_\dZ R) \times R^{\times}:  \psi(g x, g y)=\nu (g) \psi(x, y),\forall x, y \in \Lambda \otimes_{\dZ } R\} 
	$$
	for any $\dZ$-algebra $R$. Note that $G$ can be also defined as the similitude group of $\langle\_,\_ \rangle_{\psi}$.
	
 Let $p$ be a prime number inert in $F$.
	\subsection{Bruhat-Tits tree and open compact subgroups of $G_p$}
	\subsubsection{Bruhat-Tits tree}\cite[3.1]{BG06}
	Let $ \mathcal {T} $ be the Bruhat-Tits building of $ G_p.$  According to \cite{Tit79} or \cite[1.4]{Cho94}, 
	it is a tree, 
	and its vertices decompose into two parts $ \mathcal{V} \coprod  \tilde{\mathcal{V}}.$
	Every vertex of $   \mathcal{V}$ (resp. of $ \tilde{\mathcal{V}} ) $ has $ p ^ {3} + 1 $ (resp. $ p + 1 ) $ neighbours which are all in $ \tilde{\mathcal{V}} ) $ (resp. in $ \mathcal{V}$).
	The points of $ \cV $ are \emph{hyperspecial} points in the sense of \cite{Tit79}, those of $ \tilde{\cV} $ are special points which are not hyperspecial. We denote by $ \cE $ the set of (non-oriented) edges of $  {\cT}$.
	
	The tree $ \mathcal {T} $ is endowed with an action of $ G_p.$ The center $ Z_p \subset G_p$ acts on $\cT$ trivially. The action of $ G_p $ on $ \cV $ (resp. $   \tilde{\cV} ) $ is transitive, and the stabilizer of a vertex $ v $  acts transitively on the set of vertices of $ \mathcal {V} $ with distance $ n $ from $ v $ \cite[1.4, 1.5]{Cho94}, and therefore on the set of elements of 
	$ \cE $ of origin $ v $.
	
	\subsubsection{Maximal compact subgroup}\cite[3.2]{BG06}
	According to \cite{BT72}, a maximal compact subgroup of $ G_p $ fixes one and only one vertex of $ \mathcal {T} $, which defines a bijection between the set of maximal compact subgroups of $ G_p $ and $ \cV \coprod \tilde \cV. $ There are therefore two conjugacy classes of maximal compact subgroups of $ G_p$,  those who fix a vertex of $ \cV $, which we call \emph{hyperspecial}, and those who fix a vertex of $ \tilde \cV  $, which we call \emph{special}.
	
	Let $ v \in \cV$ and $v^{\prime} \in \tilde \cV  $. We denote by $ K_v $ and $ K_{v^{\prime}} $ the maximal compact subgroup which fixes $v$ and $v'$.
	Then $K_v$ is conjugate to $K_p:=G(\dZ_p),$ which is the stabilizer of the standard self-dual lattice
	$$
	\Lambda_{0}=\Lambda \otimes \mathbb{Z}_{p}=\langle e_{0}, e_{1}, e_{2}\rangle_{O_{F_p}}.
	$$
	In the meanwhile, $K_{v'}$ is conjugate to $\tilde{K}_p$ which is the stablizer of the lattice
	$$
	\Lambda_{1}=\langle p e_{0}, e_{1}, e_{2}\rangle_{O_{F_p}}.
	$$
	Assume that  $v$ and $v'$ are neighbors. The stabilizer  $K_v \cap K_{v'}$ of the edge $(v,v')$ is an \emph{Iwahoric subgroup} of $G_p.$
	
	\subsection{Picard modular surface over $\dC$}\label{subsec:Picard}
	Define the  bounded symmetric domain associated with $G$ as
\[
\cB=\{(z_{0}: z_{1}: z_{2}) \in \dP^2(\dC) \mid \bar{z}_{0} z_{2}+\bar{z}_{1} z_{1}+\bar{z}_{2} z_{0}<0\}
\] 
	 	which is biholomorphic to the unit ball in $\mathbb{C}^{2}$.
	The group 
	$ G(\dR)$ acts 
	by projective linear transformations on $\mathbb{P}^{2}(\mathbb{C})$, the action of $G(\dR)$ preserves $\cB$ and induces  a transitive action on $\cB$. 
	Denote by $K_{\infty}$ the stabilizer of the "center" $(-1: 0: 1)$. Then we have an homeomorphism
	$$
	G(\dR)/K_\infty \cong \cB.
	$$
	\subsection{Shimura varieties for unitary groups}
	Consider the Deligne homomorphism 
	\[
	\xymatrix@R=0pt{
		h_0\colon  \dS(\dR) = \dC^\times \ar[rr] && G(\dR)
		\\
		z=x+\sqrt{-1}y \ar@{|->}[rr] &&
		(\diag( \ol{z}, z, \ol{z} ), z\ol{z})}
	\]
	where  $\ol{z}$ is the complex conjugate of $z,$ and
	$G(\dR)$ acts on $\Hom_{\dR\text{-group scheme}}(\dS, G)$ by conjugation. The stabilizer of $h_0$ of $G(\dR)$ is $K_\infty,$ and
	there exists a bijection between $\cB$ and the $G(\dR)$-conjugacy class $X$ of $h_0.$  
	
For an compact open subgroup $K \subset  G(\dA^\infty ),$	the Shimura variety $\Sh(G,K)$ of level $K$ is a quasi-projective algebraic variety defined over $F$ whose complex points are identified with 
	$$
	\Sh(G, K)(\mathbb{C}) :=  {G}(\mathbb{Q}) \backslash  {G}(\mathbb{A}) /  KK_{\infty} \simeq  {G}(\mathbb{Q}) \backslash[ X \times  {G}(\mathbb{A}^{\infty}) / K].
	$$
	In this article, we will consider the Shimura varieties $\Sh(G,K)$ with $K$ of the form $K=K^pK_p$, $K^p\tilde {K}_p$ or $K^p\Iw_p$, where $K^p$ is a fixed open compact subgroup of $G(\dA^{\infty})$, as well as their canonical integral models over $O_{F,(p)}$. 
	
	\subsection{Dieudonn\'e theory on abelian schemes} 
	We first introduce some general notations on abelian schemes.
	\begin{definition}
		\cite[Definition 3.4.5]{LTX+22}
		Let $A$ and $B$ be two abelian schemes over a scheme $S \in \mathtt{Sch}_{/\dZ_{(p)} }.$ We say that a morphism of $S$-abelian schemes $\varphi: A \to B$ is a \emph{quasi-isogeny} if there is an  integer $n$ such that $n\varphi$ is an isogeny. We say that a morphism of $S$-abelian schemes $\varphi: A \to B$ is a \emph{quasi-$p$-isogeny} if there exists some $c \in \mathbb{Z}_{(p)}^{\times}$such that $c \varphi$ is a isogeny. A quasi-isogeny $\varphi$ is \emph{prime-to-$p$} if 
		there exist  two integers $n,n'$ both coprime to $p$ such that $n\varphi$ and $n' \varphi^{-1}$ are 
		both isogenies. 
		
		We denote by $A^{\vee}$ the \emph{dual abelian scheme} of $A$ over $S$.
		A \emph{quasi-polarization} of $A$ is a quasi-isogeny $\lambda:A \to A^\vee$ such that $n\lambda$ is a polarization of $A$ for some $n \in \dZ.$
		A quasi-polarization $\lambda : A \to A^\vee$   is called \emph{$p$-principal} if $\lambda$ is a prime-to-$p$ quasi-isogeny.
	\end{definition}

	\begin{notation}
		Let $A$ be an abelian variety over a scheme $S$. We denote by $\mathrm{H}_{1}^{ \r{dR} }(A / S)$ (resp.  $\Lie_{A / S}$, resp. $\omega_{A / S}$ )  the \emph{relative de Rham homology} (resp. \emph{Lie algebra}, resp. \emph{dual Lie algebra}) of $A / S.$
		They are all locally free $\mathcal{O}_{S}$-modules of finite rank. We have \emph{Hodge exact sequence}
		\begin{equation}\label{Hodge}
			0 \to \omega_{A^{\vee} / S} \to \mathrm{H}_{1}^{\mathrm{dR}}(A / S) \to \mathrm{Lie}_{A / S} \to 0.
		\end{equation}
		When the base $S$ is clear from the context, we sometimes suppress it from the notation.
	\end{notation}
	
	\begin{definition}
		Let $S \in \sf{Sch}_{/ \dZ_{(p)}}.$
		\begin{enumerate}
			\item
			An \emph{$O_F$-abelian scheme} over $S$ is a pair $(A, i_A)$ in which $A$ is an abelian scheme over $S$ and $i_A: O_F \to \operatorname{End}_{S}(A) \otimes  \dZ_{(p)}$ is a ring homomorphism of algebras. 
			\item
			An \emph{unitary $O_F$-abelian scheme} over $S$ is a triple $(A, i_A, \lambda_A)$ in which $(A, i_A)$ is an $O_F$-abelian scheme over $S$, and $\lambda_A: A \to A^{\vee}$ is a quasi-polarization such that $i_A(a^{\fc})^{\vee} \circ \lambda_A=\lambda_A \circ i(a)$ for every $a \in O_F.$ 
			\item
			For two $O_F$-abelian schemes $(A, i_A)$ and $(A^{\prime}, i_A^{\prime})$ over $S$, a (quasi-)homomorphism from $(A, i_A)$ to $(A^{\prime}, i_A^{\prime})$ is a (quasi-)homomorphism $\varphi: A \to A^{\prime}$ such that 
			$\varphi \circ i_A(a)=i_A^{\prime}(a) \circ \varphi$ for every $a \in O_F$. We will usually refer to such $\varphi$ as an $O_F$-linear (quasi-)homomorphism.
		\end{enumerate}
		Moreover, we will usually suppress the notion $i_A$ if the argument is insensitive to it.
	\end{definition}
	\begin{definition}[Signature type]
		Let $(A, i_A)$ be an $O_F$-abelian scheme of dimension 3 over a scheme $S \in \sf{Sch}_{/ O_F\otimes \dP}.$ Let $r,s$ be two nonnegative integers with $r+s=3.$
		We say that
		$(A, i_A)$ has signature type $(r, s)$  if for every $a \in O_F,$
		the characteristic polynomial of $i_A(a)$ on $\Lie_{A/S}$ is
		given by 
		$$
		(T-\tau_0(a))^{r}(T-\tau_1(a))^{s}  \in \cO_S[T].
		$$
	\end{definition}
	
	\begin{remark}
		Let $A$ be an $O_F$-abelian scheme of dimension 3 of signature type
		$(r,s)$ over a scheme 
		$S \in \sf{Sch}_{ / k}$. 
		Consider the decomposition 
		\[
		\xymatrix@R=0pt{
			O_F \otimes_{\dZ } k  \ar[r]^\simeq &  k \times k 
			\\
			a \otimes x \ar@{|->}[r] &
			( \ol{\tau_{0}(a) } x,  \ol{\tau_{1}(a)} x )
		}
		\]
		where the bar denotes the mod $p$ quotient map.
		Then for any $O_F\otimes k$-module $N$ we have a canonical  decomposition 
		\begin{equation}
			\label{eq:dec}
			N = N_0 \oplus N_1
		\end{equation}
		where $a\in O_F$ acts on $N_i$ through $\tau_i.$
		Then \eqref{Hodge} induces two short exact sequences
		\[ 
		0 \to \omega_{A^{\vee} / S, i} \to \mathrm{H}_{1}^{\mathrm{dR}}(A / S)_i \to \mathrm{Lie}_{A / S, i} \to 0, \, i=0,1
		\]
		of locally free $\mathcal{O}_{S}$-modules of ranks $s, 3,r$ and $r,3,s.$
	\end{remark}
	\begin{notation}\label{n:pair}
		Let $(A, \lambda_A)$ be a unitary $O_F$-abelian scheme of signature type $(r,s)$ over a scheme $S \in \sf{Sch}_{/ O_{F,(p)}}.$ We denote
		\[
		\langle \, , \, \rangle_{\lambda_A,i}: \mathrm{H}_{1}^{\mathrm{dR}}(A / S)_i \times \mathrm{H}_{1}^{\mathrm{dR}}(A / S)_{i+1}\to \mathcal{O}_{S}, ~ i=0,1
		\]
		the $\mathcal{O}_{S}$-bilinear alternating  pairing induced by the quasi-polarization $\lambda_A$, which is perfect if and only if 
		$\lambda_A$ is $p$-principal.
		Moreover, for an 
		$\mathcal{O}_{S}$-submodule $\mathcal{F} \subseteq \mathrm{H}_{1}^{\mathrm{dR}}(A / S)_{i}$, 
		we denote by $\mathcal{F}^{\perp} \subseteq \mathrm{H}_{1}^{\mathrm{dR}}(A / S)_{i+1}$  (where $i \in \dZ/2\dZ$)
		its (right) orthogonal complement under the above pairing, if $\lambda$ is clear from the context.  		 
			\end{notation}
	\begin{notation}\label{n:Frob}
		In  notation  \ref{n:pair}, put
		$$
		A^{(p)}:=A \times_{S, \sigma} S,
		$$
		where $\sigma$ is the absolute Frobenius morphism of $S$. Then we have
		\begin{enumerate}
			\item a canonical isomorphism $\mathrm{H}_{1}^{\mathrm{dR}}(A^{(p)} / S) \simeq \sigma^{*} \mathrm{H}_{1}^{\mathrm{dR}}(A / S)$ of $\mathcal{O}_{S}$-modules;
			\item  the Frobenius homomorphism $\operatorname{Fr}_{A}: A \to A^{(p)}$ which induces the \emph{Verschiebung map}
			$$
			{\tV}_{A}:=(\operatorname{Fr}_{A})_{*}: \mathrm{H}_{1}^{\mathrm{dR}}(A / S) \to \mathrm{H}_{1}^{\mathrm{dR}} (A^{(p)} / S )
			$$
			of $\mathcal{O}_{S}$-modules;
			\item the Verschiebung homomorphism $\operatorname{Ver}_{A}: A^{(p)} \to A$ which induces the \emph{Frobenius map}
			$$
			\mathtt{F}_{A}:=(\operatorname{Ver}_{A})_{*}: \mathrm{H}_{1}^{\mathrm{dR}}(A^{(p)} / S) \to \mathrm{H}_{1}^{\mathrm{dR}}(A / S)
			$$
			of $\mathcal{O}_{S}$-modules.
		\end{enumerate}
		In what follows, we will suppress $A$ in the notations $\tF_{A}$ and $ \tV_{A}$ if the reference to $A$ is clear.
	\end{notation}
	
	In Notation \ref{n:Frob}, we have $\ker  \tF=\operatorname{im} \tV=\omega_{A^{(p)} / S}$ and $\operatorname{ker} \tV =\operatorname{im} \tF .$  
	\begin{notation}
		Suppose that $S=\operatorname{Spec} \kappa$ for a perfect field $\kappa$ of characteristic $p$ containing $\dF_{p^2}$. Then we have a canonical isomorphism $\mathrm{H}_{1}^{\mathrm{dR}} (A^{(p)} / \kappa ) \simeq \mathrm{H}_{1}^{\mathrm{dR}}(A / \kappa) \otimes_{\kappa, \sigma} \kappa .$
		\begin{enumerate}
			\item
			By abuse of notation, we have
			\begin{itemize}
				\item
				the $(\kappa, \sigma)$-linear Frobenius map $ \tF: \mathrm{H}_{1}^{\mathrm{dR}}(A / \kappa) \to \mathrm{H}_{1}^{\mathrm{dR}}(A / \kappa)$ and
				\item
				the $(\kappa, \sigma^{-1})$-linear Verschiebung map $ \tV: \mathrm{H}_{1}^{\mathrm{dR}}(A / \kappa) \to \mathrm{H}_{1}^{\mathrm{dR}}(A / \kappa)$.
			\end{itemize}
			\item
			We have the \emph{covariant} Dieudonn\'e module $\mathcal{D}(A)$ associated to the $p$-divisible group $A[p^{\infty}]$, which is a free $W(\kappa)$-module, such that $\mathcal{D}(A) / p \mathcal{D}(A)$ is canonically isomorphic to $\mathrm{H}_{1}^{\mathrm{dR}}(A / \kappa)$. Again by abuse of notation, we have
			\begin{itemize}
				\item
				the $(W(\kappa), \sigma)$-linear Frobenius map $ \tF: \mathcal{D}(A) \to \mathcal{D}(A)$ lifting the one above, and
				\item
				the $(W(\kappa), \sigma^{-1})$-linear Verschiebung map 
				$\tV:\mathcal{D}(A) \to \mathcal{D}(A)$ lifting the one above, respectively, satisfying $\tF \circ \tV= \tV \circ \tF=p$.
			\end{itemize}
		\end{enumerate}
	\end{notation}
	\begin{remark}\label{n:DFrob}
		Similar to  \ref{eq:dec} we also have a decomposition
		\[
		\cD(A)=\cD(A)_0 \oplus \cD(A)_1.
		\]
	\end{remark}
	Let $(A, \lambda_A)$ be a unitary $O_{F}$-abelian scheme of signature type $(r,s)$ over Spec $\kappa.$ 
	
	We have a pairing
	$$
	\langle\, ,\, \rangle_{\lambda_A}: \mathcal{D}(A) \times \mathcal{D}(A)  \to W(\kappa)
	$$
	lifting the one in Notation \ref{n:pair}. 
	We denote by $\mathcal{D}(A)^{\perp_A}$ the $W(\kappa)$-dual of $\mathcal{D}(A)$
	\[
	\cD(A)^{\perp_A}:= \{x \in  \mathcal{D}(A) [1/p]  \mid \langle x, y \rangle_{\lambda_A}
	\in W(\kappa), \forall y \in \cD(A).\}
	\]
	as a submodule of $\mathcal{D}(A)[1/p].$  
	We have the following properties:
	\begin{enumerate}
		\item The direct summands in (\ref{eq:dec}) are totally isotropic 
		with respect to $\langle , \rangle_{\lambda_A}.$
		\item  we have 
		\[
		\langle \tF x, y  \rangle_{\lambda_A}=\langle  x, \tV y  \rangle_{\lambda_A}^\sigma, \,   \langle i_A(a) x, y  \rangle_{\lambda_A}=\langle  x, i_A(a^\fc) y  \rangle_{\lambda_A}
		\]
		for $a\in O_F.$
	\end{enumerate}
	
	Next we review some facts from the Serre-Tate theory \cite{Kat81} and the Grothendieck-Messing theory \cite{Mes72}, tailored to our application.   Consider a closed immersion $S \hookrightarrow \hat{S}$ in 
	$\mathtt{Sch}_{/ \mathbb{Z}_{{p}^{2} } }$ 
	on which $p$ is locally nilpotent, with its ideal sheaf equipped with a PD structure, and a unitary
	$O_{F}$-abelian scheme $(A, \lambda)$ of signature type $(r,s)$ over $S$. We let $\mathrm{H}_{1}^{\text {cris }}(A / \hat{S})$ be the evaluation of the first relative crystalline homology of $A / S$ at the PD-thickening $S \hookrightarrow \hat{S}$, which is a locally free $\mathcal{O}_{\hat{S}} \otimes O_{F}$-module. The polarization $\lambda$ induces a pairing
	\begin{equation}\label{paircris}
		\langle,\rangle_{\lambda, i}^{\text {cris }}: \mathrm{H}_{1}^{\text {cris }}(A / \hat{S})_{i} \times \mathrm{H}_{1}^{\text {cris }}(A / \hat{S})_{i^{\fc}} \to \mathcal{O}_{\hat{S}}
		, ~  i=0,1.
	\end{equation}
	We define two groupoids
	\begin{itemize}
		\item
		$\operatorname{Def}(S, \hat{S} ; A, \lambda)$, whose objects are unitary $O_{F}$-abelian schemes $(\hat{A}, \hat{\lambda})$ of signature type $(r,s)$ over $\hat{S}$ that lift $(A, \lambda)$;
		\item
		$\operatorname{Def}^{\prime}(S, \hat{S} ; A, \lambda)$, whose objects are pairs $(\hat{\omega}_{0}, \hat{\omega}_{1})$ where  
		$  \hat{\omega}_{i} \subseteq \mathrm{H}_{1}^{\text {cris }}(A / \hat{S})_{i}$ is a subbundle that lifts $\omega_{A^{\vee} / S, i} \subseteq \mathrm{H}_{1}^{\mathrm{dR}}(A / S)_{i}$ for $i=0,1$, such that $\langle\hat{\omega}_{0}, \hat{\omega}_{ 1}\rangle_{\lambda, 1}^{\text {cris }}=0$.
	\end{itemize}
	\begin{proposition}\label{thm:STGM}
		The functor from $\operatorname{Def}(S, \hat{S} ; A, \lambda)$ to $\operatorname{Def}^{\prime}(S, \hat{S} ; A, \lambda)$ sending $(\hat{A}, \hat{\lambda})$ to $(\omega_{\hat{A}^{\vee} , 0}, \omega_{\hat{A}^{\vee}, 1 })$ is a natural equivalence.
	\end{proposition}
	\subsection{Moduli problems}
	Fix an open compact subgroup $K^p \subset G(\dA^{\infty,p}).$
	
	\begin{definition}\label{def:S}
	Let ${S}$ be the moduli problem that  associates with every $O_{F,(p)}$-algebra  $R$ the set 
		${S}(R)$
		of equivalence classes of triples
		$  (A, \lambda_A,  \eta_A ),$ where
		\begin{itemize}
			\item
			$(A,\lambda_A)$ is a unitary $O_F$-abelian scheme of signature type (1,2) over $R$ such that $\lambda_A$ is $p$-principal;
			\item
			$\eta_A$ is a $ K^{p}$-level structure, that is, for a chosen geometric point $s$ on every connected component of $\Spec R$, a $\pi_{1}(\Spec R,s)$-invariant $K^p$-orbit of isomorphisms
			\[\eta_A:   V \otimes_{\mathbb{Q}}  \mathbb{A}^{\infty, p} 
			\stackrel{\sim}{\longrightarrow}
			\mathrm{H}_{1}^{\r{et}}(A, \mathbb{A}^{\infty,p})
			\]
such that the skew hermitian pairing $\langle \_,\_\rangle_{\psi}$ on $V\otimes_{\dQ} \mathbb{A}^{\infty, p}$ corresponds to the $\lambda_A$-Weil pairing on $\mathrm{H}_{1}^{\r{et}}(A, \mathbb{A}^{\infty,p})$ up to scalar. 
		\end{itemize}
		Two triples $(A, \lambda_A, \eta_A)$ and 
		$(A^{\prime}, \lambda_{A^{\prime}},  \eta_{A^{\prime}} )$ are equivalent
		if there is a prime-to-$p$ $O_F$-linear isogeny $\varphi:A \to A'$ such that
		\begin{itemize}
			\item there exists $c \in \dZ_{(p)}^{\times}$such that 
			$\varphi^{\vee} \circ \lambda_{A^{\prime}} \circ \varphi=c \lambda_A;$ and
			\item the $ K^{p}$-orbit of maps 
			$v \mapsto \varphi_{*} \circ \eta_A(v)  $
			for $v \in V \otimes_{\mathbb{Q}} \mathbb{A}^{\infty, p}$ coincides with $\eta_{A'}.$
		\end{itemize}
		Given $g \in  {K}^{p} \backslash  G (\mathbb{A}^{\infty, p} ) /  {K}'^{p  }$ such that $g^{-1}K^pg \subset {K'^p}$, we have a map
		${S}  (  K^p )(R)$ to 
		${S}  (  {K'^p} )(R)$ by changing $\eta_A $ to $\eta_A \circ g.$ 
	\end{definition}

	\begin{definition}\label{def:Stilde}
		Let   $\widetilde{{S}}$ be the moduli problem that  associates with every $O_{F,(p)}$-algebra $R$
		the set 
		$\widetilde{{S}}(R)$
		of equivalence classes of triples
		$(\tilde{A}, \lambda_{\tilde{A}}, \eta_{\tilde{A}}),$ where
		\begin{enumerate}
			\item  
			$(\tilde{A},\lambda_{\tilde{A} })$ is a unitary $O_F$-abelian scheme of signature type (1,2) over $R$ such that
			$\ker\lambda_{\tilde A}[p^\infty] $ is contained in $ \tilde{A}[p]$ of  rank $p^2;$
			\item  $\eta_{\tilde A}$ is a $ K^{p}$-level structure.
		\end{enumerate}
		The equivalence relation and the action of $G(\dA^{\infty,p})$ are defined similarly as in Definition \ref{def:moduliS}.
	\end{definition} 
	\begin{definition}\label{def:S0p}
		The moduli problem ${S_0(p)}$ associates with every $O_{F}\otimes \dZ_{(p)}$-algebra $R$
		the set 
		$ {S_0(p)}(R)$
		of equivalence classes of sextuples
		$(A, \lambda_A , \eta_A,
		\tilde{A}, \lambda_{\tilde{A}}, \eta_{\tilde{A}},
		\alpha)
		$ where
		\begin{enumerate}
			\item $(A, \lambda_A,   \eta_A)$ is an element in ${S} (R).$
			\item $(\tilde{A}, \lambda_{\tilde{A}}, \eta_{\tilde{A}})$
			is an element in ${\tilde{S}} (R).$
			\item $\alpha: A \to \tilde{A}$ is an $O_F$-linear quasi-$p$-isogeny such that 
			\[
			p \lambda_A= \alpha^\vee \circ \lambda_{\tilde{A}} \circ  \alpha.
			\]
			\item
			$\ker{\alpha} \subset A[p]$ is a Raynaud $O_F$-subgroup scheme of rank $p^{2}$, which is isotropic for the $\lambda_A$-Weil pairing
			\[
			e_p: A[p]\times A[p] \to \mu_p.
			\]
			For the definition of Raynaud subgroup, see \cite[1.2.1]{dSG18}.
		\end{enumerate} 
		Two septuplets $(A, \lambda_A , \eta_A,
		\tilde{A}, \lambda_{\tilde{A}}, \eta_{\tilde{A}},
		\alpha)$ and 
		$(A', \lambda_A' ,  \eta_A',
		\tilde{A'}, \lambda_{\tilde{A}'}, \eta_{\tilde{A}'},
		\alpha')$ 
		are equivalent if there are $O_{F}$-linear prime-to-$p$ quasi-isogenies 
		$ \varphi: A \to A^{\prime}$ and 
		$\varphi': \tilde A \to \tilde{A}'$ such that
		\begin{itemize}
			\item there exists $c \in \dZ_{(p)}^{\times}$
			such that 
			$\varphi^{\vee} \circ \lambda_{A^{\prime}} \circ \varphi=c \lambda_A$
			and
			$\varphi^{\vee} \circ \lambda_{\tilde A^{\prime}} \circ \varphi=c \lambda_{\tilde A}.$
			\item the $ K^{p}$-orbit of maps 
			$v \mapsto \varphi_{*} \circ \eta_A(v)  $
			for $v \in V \otimes_{\mathbb{Q}} \mathbb{A}^{\infty, p}$ coincides with $\eta_{A'}.$
			\item the $ K^{p}$-orbit of maps 
			$v \mapsto \varphi'_{*} \circ \eta_{A'}(v)  $
			for $v \in V \otimes_{\mathbb{Q}} \mathbb{ A}^{\infty, p}$ coincides with $\eta_{\tilde{A}'}.$
		\end{itemize}
	\end{definition}

It is well known that, for sufficiently small $K^p$,  the  three moduli problems $S$, $\tilde{S}$ and ${S_0(p)}$  are all  representable by  quasi-projective  schemes over $O_{F,(p)}$,
still denoted by $S,\tilde S,S_0(p)$ by abuse of notation,
 and give integral models of $\Sh(G, K^pK_p)$,  $\Sh(G, K^p\tilde K_p)$ and $\Sh(G, K^p\Iw_p)$ respectively. 
 We have  natural forgetful maps $\pi:{S_0(p)} \to {S}$  sending $(A, \lambda_A,  \eta_A, \tilde{A},  \lambda_{\tilde{A}},   \eta_{\tilde{A}}, \alpha)$ to 
$(A, \lambda_A,  \eta_A) $, and 
$\tilde\pi:{S_0(p)} \to {\tilde{S}}$  sending $(A, \lambda_A, \eta_A, \tilde{A}, \lambda_{\tilde{ A}}, \eta_{\tilde{A}}, \alpha)$ to $(\tilde{A}, \lambda_{\tilde{ A}}, \eta_{\tilde{A}}).$
This gives rise to  the diagram
\[
\xymatrix{
	& {S_0(p)} \ar[dl]_\pi  \ar[dr]^{\tilde{\pi}}& \\
	S  & & {\tilde{S}} }
\]

		\begin{remark}\label{uniform}
			For the convenience of readers, we recall why $S$ is an integral model of $\Sh(G,K^pK_p)$. 
		We shall content ourselves  with describing  a canonical  bijection 
		${S}(\dC)\simeq \Sh(G, K^pK_p)(\dC)$, which determines uniquely an isomorphism $S\otimes_{O_{F,(p)}}F\cong \Sh(G,K^pK_p)$. 
		It suffices to assign to each  point
		\[
		s=( A, \lambda_A,  \eta_A)  \in S(\dC)
		\]
		a point in
		\[
		\Sh(G, K^pK_p) (\mathbb{C})=G(\dQ) 
		\backslash(X \times G(\mathbb{A}^{\infty}) /{K}^{p} {K}_{p})
		\]
		Let $H:=\r{H}_1(A, \dQ)$, which is an $F$-vector space by the action of $O_F$ on $A$. The polarization $\lambda_A$ induces a structure of   skew hermitian space on  $H$. 
		By Hodge theory, the composed map  $$H\otimes_{\dQ}\dR\cong \rH^{\mathrm{dR}}_{1}(A, \dR)\to \rH^{\mathrm{dR}}_{1}(A, \dC)\to \Lie_A$$
		is an isomorphism, which gives
		a complex structure on $H\otimes_{\dQ}\dR$. The signature condition on $A$ ensures an isomorphism of (skew) hermitian spaces $H\otimes_{\dQ}\dR\cong V\otimes_{\dQ}\dR$. Now look at the place $p$. Since $A$ is an abelian variety up to prime-to-$p$ isogeny,  the $\dZ_p$-module $\Lambda_H:= \rH^{\et}_{1}(A, \dZ_p)$  is well-defined  and gives a self-dual lattice in  $H\otimes_{\dQ}\dQ_p\cong \rH^{\et}_1(A, \dQ_p)$. Hence there exists an isomorphism of hermitian spaces $H\otimes_{\dQ}\dQ_p\cong V\otimes_{\dQ}\dQ_p$.
		 In addition to  the prime-to-$p$ level structure $\eta_A$, the Hasse principle implies  that there exists globally an isomorphism of hermitian spaces  $\xi: H\xrightarrow{\sim} V$ over $F$ up to similitude. Fix such a $\xi$.  
		First,  the complex structure on $H\otimes_{\dQ}\dR$ transfers via $\xi$ to  
		a homomorphism of $\dR$-algebras $\dC\to \End_F(V)\otimes_{\dQ}\dR$, which leads to an  element $x\in X=G(\dR)/K_{\infty}$ because of the signature condition. Secondly, post-composing with $\xi$, the level structure $\eta_A$ gives a coset $g^pK^p:=\xi\circ\eta_A\in G(\dA^{\infty,p})/K^p$. At last,   there  exists  a coset $g_pK_p \in G(\mathbb{Q}_{p})/K_p$ such that $\xi(\Lambda_H)= g_p (\Lambda_{0})$ as lattices of $V\otimes_{\dQ}\dQ_p$ for any representative $g_p$ of $g_pK_p$.  
		Note that a different choice of $\xi$ differs by the left action of  an element of $G(\dQ)$. 
		It follows that the class   
		\[[x, g^pK^p, g_pK_p]\in  G(\dQ)\backslash (X\times G(\dA^{\infty,p})/K^p\times G(\dQ_p)/K_p) \]
		does not depend on  the choice of $\xi$, and gives the point of $\Sh(G, K^pK_p)$ corresponding to $s\in {S}(\dC)$.

	\end{remark}
	
		

\subsection{An inner form of $G$}\label{sec:inner}
Let $(W,\psi_W)$ be a hermitian space over $F$ of dimension 3 such that it is isomorphic to $(V\otimes_{\dQ}\dA^{\infty}, \psi)$ as hermitian spaces over $\dA^{\infty}$, and $(W\otimes_{\dQ}\dR, \psi_W)$ has signature $(0,3)$. Such a $(W, \psi_W)$ exists and  is unique up to isomorphism. 
	Let $G'$ be the unitary similitude group over $\dQ$ attached to $(W, \psi_W)$. Then $G'$ is an inner form of $G$  such that  $G'(\dA^{\infty})\cong G(\dA^{\infty})$. In the sequel, we fix such  an isomorphism  so that $K^p$ and $K_p$ are also viewed respectively as subgroups of $G'(\dA^{\infty,p})$ and $G'(\dQ_p)$.  As $G'(\dR)$ is compact modulo center, for an open compact subgroup $K'\subseteq G'(\dA^{\infty})$, we have a finite set
	\[
	\Sh(G',K'):=G'(\dQ)\backslash G'(\dA^{\infty})/K'.
	\] 
	We will give   moduli interpretations for $\Sh(G', K^pK_p)$, $\Sh(G',K^p\tilde K_p)$ and $\Sh(G', K^p\Iw_p)$. 

	\begin{definition}\label{def:moduliS}
		The moduli problem ${T}$ is to associate with every $O_{F,(p)}$-algebra  $R$ the set 
		${T}(R)$
		of equivalence classes of triples
		$  (B, \lambda_B,  \eta_B ),$ where
		\begin{itemize}
			\item
			$(B,\lambda_B)$ is a unitary $O_F$-abelian scheme of signature type (0,3) over $R$ such that $\lambda_B$ is $p$-principal;
			\item
			$\eta_B$ is a $ K^{p}$-level structure, that is, for a chosen geometric point $s$ on every connected component of $\Spec R$,  $\eta_B$ is a $\pi_{1}(\Spec R,s)$-invariant $K^p$-orbit of isomorphisms
			\[\eta_B:   W \otimes_{\mathbb{Q}}  \mathbb{A}^{\infty, p} 
			\stackrel{\sim}{\longrightarrow}
			\mathrm{H}_{1}(B, \mathbb{A}^{\infty,p})
			\]
			of hermitian spaces over $F \otimes_{\mathbb{Q}} \mathbb{A}^{\infty, p}.$
		\end{itemize}
		Two triples $(B, \lambda_B, \eta_B)$ and 
		$(B^{\prime}, \lambda_{B^{\prime}},  \eta_{B^{\prime}} )$ are equivalent
		if there is a prime-to-$p$ $O_F$-linear isogeny $\varphi:B \to B'$ such that
		\begin{itemize}
			\item there exists $c \in \dZ_{(p)}^{\times}$such that 
			$\varphi^{\vee} \circ \lambda_{B^{\prime}} \circ \varphi=c \lambda_B;$ and
			\item the $ K^{p}$-orbit of maps 
			$v \mapsto \varphi_{*} \circ \eta_B(v)  $
			for $v \in W \otimes_{\mathbb{Q}} \mathbb{A}^{\infty, p}$ coincides with $\eta_{B'}.$
		\end{itemize}
		Given $g \in  {K}^{p} \backslash  G' (\mathbb{A}^{\infty, p} ) /  {K}'^{p  }$ such that $g^{-1}K^pg \subset {K'^p},$
		we have a map
		${T}  (  K^p )(U)$ to 
		${T}  (  {K^p}' )(U)$ by changing $\eta_A $ to $\eta_A \circ g.$
	\end{definition} 
	\begin{definition}
		The moduli problem  $\widetilde{{T}}$ is to associate with every $O_{F}\otimes \dZ_{(p)}$-algebra $R$
		the set 
		$\widetilde{{T}}(R)$
		of equivalence classes of triples
		$(\tilde B, \lambda_{\tilde B}, \eta_{\tilde B}),$ where
		\begin{enumerate}
			\item $(\tilde B,  \lambda_{\tilde B})$ is a unitary $O_F$-abelian scheme of signature type (0,3) over $R$ such that 
			$\ker\lambda_{\tilde B}[p^\infty] $ is contained in $ \tilde{B}[p]$ of  rank $p^2;$
			\item  $\eta_{\tilde B}$ is a $ K^{p}$-level structure.
		\end{enumerate}
		The equivalence relation and the action of $G(\dA^{\infty,p})$ are defined similarly as in Definition \ref{def:Stilde}.
	\end{definition}

	\begin{definition}
		The moduli problem ${T_0(p)}$ is to associate with every $O_{F}\otimes \dZ_{(p)}$-algebra $R$
		the set ${T_0(p)}(R)$of equivalence classes of sextuples
		$(B, \lambda_B , \eta_B,
		\tilde B, \lambda_{\tilde B}, \eta_{\tilde B},
		\beta)
		$ where
		\begin{enumerate}
			\item $(B, \lambda_B,   \eta_B) \in {T} (R);$
			\item $(\tilde B, \lambda_{\tilde B}, \eta_{\tilde B}) \in {\tilde{T}} (R);$
			\item $\beta: \tilde B \to   B$ is an isogeny such that 
			\[
			p \lambda_{\tilde B} 
			= \beta^\vee \circ \lambda_B \circ  \beta;
			\]
			\item
			$\ker{\beta}$ is a $O_F$-subgroup scheme of
			$   \tilde B[p]$ of rank $p^{4}$,  which is isotropic for the $\lambda_B$-Weil pairing
			\[
			e_p: B[p]\times B[p] \to \mu_p.
			\]
		\end{enumerate}
		The equivalence relation and the action of $G(\dA^{\infty,p})$ are defined similarly as in Definition \ref{def:S0p}.
	\end{definition}  
	For sufficiently small $K^{p}$, three moduli problems defined above are representable by quasi-projective schemes over $O_{F,(p)}.$ By abuse of notation we still denote them by $T, \tilde{T}, {T_0(p)}.$
   	\begin{proposition}\label{uni1}
		We have the uniformization maps
		\[
		\xymatrix@R=0pt{
			\upsilon: 
			T(\dC) \cong \Sh(G', K^pK_p) 
			\\
			\tilde{\upsilon}: 
			\tilde{T}(\dC) \cong \Sh(G', K^p\tilde{K}_p) 
			\\
			\upsilon_0:
			{T_0(p)}(\dC) \cong \Sh(G', K^p \Iw_p) 
		}
		\]
		which is equivariant under prime-to-$p$ Hecke correspondence. That is, given $g \in K^p \backslash G(\dA^{\infty,p})/ K'^p,$ we have the commutative diagram
		\[
		\xymatrix{
			T(K^p)(\dC) \ar[r]^-\upsilon \ar[d]^g& \Sh(G', K^pK_p)  \ar[d]^g \\
			T(K'^p)(\dC) \ar[r]^-\upsilon& \Sh(G', K'^pK_p) 
		}
		\]
		for $g \in  {K}^{p} \backslash  G (\mathbb{A}^{\infty, p} ) /  {K}'^{p  }$ such that $g^{-1}K^pg \subset {K'^p}.$   
		Here we use $T(K^p)$ to emphasize the dependence of $T$ on $K^p$.
		Similar diagrams hold for $\tilde T$ and $T_0(p).$
		 \begin{proof}
			Similar to Remark \ref{uniform}. It is worthwhile noting the signature type condition forces the image of  $\dC\to \End_F(W)\otimes \dR$ lies in the  center $F\otimes \dR$.
		\end{proof}
	\end{proposition}
	
	\section{The geometry of geometric special fiber}\label{sec:geometry}
	 Let $k$ be  a prefect field.
	    Denote   by $S_k$ or $S\otimes k$ the base change of $S$ to $k$. 
         If $k=\Fpp$ we denote still by $S$ the special fiber $S \otimes \Fpp.$ 
	Same notation holds for other integral models.
	\subsection{The geometry of $S$} 
We recall the Ekedahl-Oort stratification on $S$, which  has been  studied extensively in \cite{Wed01,BW06, VW11}.
	 Given $(A,i_A,\lambda_A, \eta_A) \in S(k).$
	Define two standard Dieudonn\'e modules as "building blocks" of $\cD(A[p]):$
	\begin{definition}\cite[3.2]{BW06},\cite[2.4, 3.1]{VW11}
		\begin{enumerate}
			\item
			Define a superspecial unitary Dieudonn\'e module $\cS$ over $k$ as follows. It is a free $W(k)$-module of rank 2 with a base $\{ g,h \}.$
			Set 
			\[
			\cS_0= W(k)g , \, \cS_1= W(k) h, \, \cS= \cS_0 \oplus \cS_1.
			\]
			$\cS$ is equipped by the natural $O_F\otimes W(k)$ action.
			
			Define an alternating form on $\cS$   by $\langle g,h\rangle=-1.$
			Define a $(W(k),  \sigma)$-linear map 
			$\tF$ on $\cS$ by $\tF g=p h$ and $\tF h= -g.$ 
			Define a $(W(k),  \sigma^{-1})$-linear map 
			$\tV$ by $\tV h=g$ and $\tV g = -ph.$ 
			This makes  $\cS$ is a unitary Dieudonn\'e module of signature $(0,1)$.  Write by $\ol\cS$ its reduction mod $p.$
			\item \label{DieudonneS3}
			For an integer $r \geqslant 1$ define a unitary Dieudonn\'e module $\cB(r)$ over $k$ as follows. It is a free $W(k)$-module of rank $2r$ with a base
			$(e_{1}, \ldots, e_{r}, f_{1}, \ldots, f_{r}).$
			Set                                                                                 
			$$
			\cB(r)_{0}=W(k) e_{1} \oplus \cdots \oplus W(k) e_{r}   , \,   \cB(r)_{1}=W(k) f_{1} \oplus \cdots \oplus W(k) f_{r} ,\,\cB(r)= \cB(r)_0 \oplus \cB(r)_1.
			$$
			The alternating form is defined by
			$$
			\langle e_{i}, f_{j}\rangle=(-1)^{i} \delta_{i j} .
			$$
			Finally, define a $\sigma$-linear map 
			$\tF$ and a 
			$\sigma^{-1}$-linear map 
			$\tV$ by
			$$
			\begin{aligned}
				\tV e_{i}&= p f_{i+1} , & & \text { for } i=1, \ldots, r-1, \\ 
				\tV e_{r} &=f_{1}, & & \\
				\tV f_{i} &=e_{i+1}, & & \text { for } i=1, \ldots, r-1, \\
				\tV f_{n} &=p e_{1}, & & \\
				\tF e_{1} &=(-1)^{r} f_{r}, & & \\
				\tF e_{i} &=pf_{i-1}, & & \text { for } i=2, \ldots, r, \\ 
				\tF f_{1} &=p e_{r}, & & \\
				\tF f_{i} &=e_{i-1}, & & \text { for } i=2, \ldots, r, \\ 
			\end{aligned}
			$$
			This is a unitary Dieudonn\'e module of signature $(1,r-1).$
			Write by $\ol\cB(r)$ its reduction mod $p.$
		\end{enumerate}
	\end{definition}

	\begin{proposition}\label{prop:EO}
		Let $x = (A, i_A, \lambda_A,\eta_A) \in S(k).$ Then $\cD(A[p]) \cong \cD(A)/p $ is isomorphic to
		\[
		\ol{\cB}(r) \oplus \ol{\cS}^ {\oplus{3-r}}
		\]
		for some integer $r$ with $1 \leqslant r \leqslant 3$.
	\end{proposition}
	The \emph{Ekedahl-Oort stratification} in our case is given by
	\[
	S= S_1 \bigsqcup S_2 \bigsqcup S_3,
	\] 
where each $S_i$ is a reduced locally closed subscheme, and  a geometric point $(A, i_A, \lambda_A, \eta_A ) \in S(k)$ lies in $S_i(k)$ if and only if $$	\rH_1^{\rm{dR}}(A/k)_0 \cong \ol\cB (i) \oplus \ol\cS ^{\oplus 3-i}.$$
All $S_i$ are equidimensional \cite[Section 6]{Wed01}, and we have $\dim S_{2}=2, \dim S_{1}=0$ and $\dim S_3=1.$

	The open stratum $S_{2}$ is usually called the \emph{$\mu$-ordinary} locus, and  denoted by $S_\mu$. 
	Its complement  $S_{\r{ss} } := S_1 \cup S_3=S\backslash S_2$ is  the   \emph{supersingular} locus, i.e., the associated $F$-isocrystal $(\cD(A)[1/p],\tF)$ of $A$ has Newton slope $1/2.$ 
	Furthermore, the stratum $S_1$ is exactly the locus where $\tF \cD(A)=\tV \cD(A)$ holds. It is called  the  \emph{superspecial} locus  and denoted by $S_{\r{sp}}.$ 
	The stratum $S_3$ is called \emph{general supersingular} locus, denoted by $S_{\r{gss}}.$
	We will study the irreducible components of supersingular locus $S_{\r{ss }} .$ 
	
	\subsection{Unitary Deligne-Lusztig variety}
	Let $\kappa$ be a field containing $\dF_{p^2}$ and  denote by $\ol\kappa$ one of its algebraic closure.
	Recall   $\sigma:S \to S$ denotes the absolute $p$-power Frobenius morphism for schemes $S$ in characteristic $p$.
	\begin{definition}\label{def:vectorspace}
		Consider a pair $(\mathscr{V},\{\,,\,\})$ in which $\mathscr{V}$ is a  $\kappa$-linear space of dimension 3, and $\{\,,\,\}: \mathscr{V} \times \mathscr{V} \to \kappa$, is a  non-degenerate pairing that is $\kappa$-linear in the first variable and $(\kappa,\sigma)$-linear in the second variable. For every $\kappa$-scheme $S$, put $\mathscr{V}_{S}:=\mathscr{V} \otimes_{\kappa} \mathcal{O}_{S}$. Then there is a unique pairing $\{ \, ,\,\}_{S}: \mathscr{V}_{S} \times \mathscr{V}_{S} \to \mathcal{O}_{S}$ extending $\{ \, ,\, \}$ that is 
		$ \mathcal{O}_{S} $-linear in the first variable and $(\mathcal{O}_{S}, \sigma)$-linear in the second variable. For a subbundle $H \subseteq \mathscr{V}_{S}$, we denote by $H^\perp \subseteq \mathscr{V}_{S}$ its orthogonal complement under $\{ \,, \,\}_{S}$ defined by  
		\[
		H^\perp= \{ x \in \sV_S \mid \{x, H\}_S =0 \}.
		\]   
		When the pairing is induced by a (quasi-)polarization $\lambda_A$ of an abelian variety $A,$ we write $\perp_{\bar A}$
		instead of $\perp$ to specify.
	\end{definition}

	\begin{definition}\label{def:admis}
		We say that a pair $(\mathscr{V},\{ \, , \, \} )$ is \emph{admissible} if there exists an $\mathbb{F}_{p^{2}}$-linear subspace $\mathscr{V}_{0} \subseteq \mathscr{V}_{\bar{\kappa}}$ such that the induced map $\mathscr{V}_{0} \otimes_{\mathbb{F}_{p^{2}}} \bar{\kappa} \to \mathscr{V}_{\bar{\kappa}}$ is an isomorphism, and $\{x, y\}=-\{y, x\}^{\sigma}=\{x, y\}^{\sigma^2}$ for every $x, y \in \mathscr{V}_{0}$.
	\end{definition}
	
	\begin{definition}
		Let  $ \mathrm{DL}(\mathscr{V},\{ \, , \, \})$ be the moduli problem   associating with every $\kappa$-algebra $R$
		the set $ \operatorname{DL}(\mathscr{V},\{ \, , \, \})(R)$ of   subbundles $H$ of $\mathscr{V}_{R}$ of rank $2$ such that $H^{\perp} \subseteq H .$ We call $\mathrm{DL}(\mathscr{V},\{ \, , \, \}, h)$ the \emph{(unitary) Deligne-Lusztig variety} attached to $(\mathscr{V},\{ \, , \,\})$ of rank $2$. 
	\end{definition}
	\begin{proposition}\label{prop:admissible}
		Consider an admissible pair $(\mathscr{V},\{ \,  ,  \, \}).$ 
		Then $\mathrm{DL}(\mathscr{V},\{ \, , \, \})$ is represented by a projective smooth scheme over $\kappa$ of dimension $1$
		with a canonical isomorphism for its tangent sheaf
		\[
		\mathcal{T}_{\mathrm{DL}(\mathscr{V},\{,\}) / \kappa} \simeq  {\cH o m}(\mathcal{H} / \mathcal{H}^{\perp}, \mathscr{V}_{\mathrm{DL}(\mathscr{V},\{,\})} / \mathcal{H})
		\]
		where $\mathcal{H} \subseteq \mathscr{V}_{\mathrm{DL}(\mathscr{V},\{,\})}$ is the universal subbundle.
	Moreover, 
		$  \mathrm{DL}(\mathscr{V},\{,\})\otimes_{\kappa}{\bar \kappa} $
		is isomorphic to the Fermat curve $\sC \subset \dP^2_{\bar \kappa}:$
		\[
		\sC: \{(x:y:z  )\in \dP^2_{\bar \kappa } | x^{p+1}+y^{p+1}+z^{p+1}=0 \}.
		\]
		\begin{proof}
		For the first part, see \cite[Proposition A.1.3]{LTX+22}.
				For the second part,
				by admissibility there exists an
				$\Fpp$-linear space $\sV_0$ such that 
				$\sV_0\otimes\ol\kappa\to\sV_{\ol\kappa}$ is an isomorphism.
		Fix an element $\delta\in \dF_{p^2}^\times $ such that 
		$\delta^\sigma=-\delta.$
		Then we can find a basis $\{e_1,e_2,e_3\}$ of $\sV_0$
		which can be regarded as a basis of $\sV_{\ol\kappa}$
		such that $\{e_i,e_j\}=\delta \delta_{ij}.$
 		Take a rank 2 $\ol\kappa$-subspace $H$ of $\sV_{\ol\kappa}.$
		If $e_3\not\in H$, we can assume 
		$H=\{z e_1+xe_3,ze_2+ye_3\}$ 
		where $x,y,z\in \ol\kappa$ and $z\neq0.$
		Then $H^\perp=\{-x^p e_1-y^p e_2+z^p e_3\}.$
		The condition $H^\perp\subset H$ is equivalent to 
		 $H^\perp\cap H\neq\{0\},$ i.e., 
	\[
\left|\begin{array}{ccc}
z & 0 & x \\
0 & z & y \\
-x^p & -y^p & z^p
\end{array}\right|=z(x^{p+1}+y^{p+1}+z^{p+1})=0.
\]	 
Thus $x^{p+1}+y^{p+1}+z^{p+1}=0$. It is easy to see the map
 $\{z e_1+xe_3,ze_2+ye_3\}\mapsto (x:y:z)$
extends to an isomorphism 
$  \mathrm{DL}(\mathscr{V},\{,\})\otimes_{\kappa}{\bar \kappa} \cong \sC.$
  		\end{proof}
	\end{proposition}
	\begin{notation}
		Take a point $t= (B, \lambda_B, \eta_B ) \in T(\kappa).$ Then $B[ p^{\infty}]$ is a supersingular $p$-divisible group by the signature condition and the fact that $p$ is inert in $F$. From Notation \ref{n:Frob}, we have the $(\kappa, \sigma)$-linear Frobenius map
		\[
		{\tF}: \mathrm{H}_{1}^{\mathrm{dR}}(B / \kappa)_{i} \to \mathrm{H}_{1}^{\mathrm{dR}}(B / \kappa )_{i+1} , \quad i\in \dZ/2\dZ.
		\]
		which can be lifted to 
		\[
		{\tF}: \cD(B)_i  \to \cD(B)_{i+1}.
		\]
		We define a pairing
		\[
		\{ \, , \, \}_{t}: \mathrm{H}_{1}^{\mathrm{dR}}(B /  \kappa )_{ i} \times \mathrm{H}_{1}^{\mathrm{dR}}(B /  \kappa )_{ i+1 }  \to \kappa
		\]
		by the formula 
		$\{ x, y \}_{t}:=\langle   x ,  \tF y \rangle_{\lambda_B }.$
		This pairing can also be lifted to 
		\[
		\{ \, , \, \}_{t}: \cD(B)_i \times \cD(B)_{i+1}  \to W(\kappa)
		\]
		To ease notation, we put 
		\[
		\mathscr{V}_{t}:=\mathrm{H}_{1}^{\mathrm{dR}}(B / \kappa)_{1}. 
		\]
	\end{notation}
	\begin{lem}
		The pair $(\mathscr{V}_{t },\{ \, , \, \}_{ t })$ is admissible of rank 3. In particular, the Deligne-Lusztig variety $\mathrm{DL}_{t}:=\mathrm{DL} (\mathscr{V}_{t },\{ \, , \, \}_{ t })$ is a geometrically irreducible projective smooth scheme in $\sf{Sch}_{/ \kappa}$ of dimension 1 with a canonical isomorphism for its tangent sheaf
		\[
		\mathcal{T}_{\mathrm{DL}_{ t } / \kappa} \simeq \cH om(\mathcal{H} / \mathcal{H}^{\dashv},(\mathscr{V}_{t})_{\mathrm{DL}_{t} } / \mathcal{H})
		\]
		where $\mathcal{H} \subseteq(\mathscr{V}_{t})_{D_{t}}$ is the universal subbundle.
		\begin{proof}
			It follows from the construction that $\{ \, , \, \}_t$ is 
			$\kappa$-linear in the first variable and
			$(\kappa, \sigma)$-linear in the second variable. Thus by Proposition \ref{prop:admissible} it suffices to show that $(\mathscr{V}_{t},\{,\}_{t})$ is admissible.
			
			Note that we have a canonical isomorphism
			$(\mathscr{V}_{ t})_{\bar{\kappa}}=\mathrm{H}_{1}^{\mathrm{dR}}(B / \kappa)_{i} \otimes_{\kappa} \bar{\kappa} \simeq \mathrm{H}_{1}^{\mathrm{dR}}(B_{\bar{\kappa}} / \bar{\kappa})_i$, 
			and that the $(\bar{\kappa}, \sigma)$-linear Frobenius map $\tF: \mathrm{H}_{1}^{\mathrm{dR}}(B_{\bar{\kappa}} / \bar{\kappa})_i \to \mathrm{H}_{1}^{\mathrm{dR}}(B_{\bar{\kappa}} / \bar{\kappa})_{i+1}$ 
			and the $(\bar{\kappa}, \sigma^{-1})$-linear Verschiebung map $ \tV: \mathrm{H}_{1}^{\mathrm{dR}}(B_{\bar{\kappa}} / \bar{\kappa})_i \to \mathrm{H}_{1}^{\mathrm{dR}}(B_{\bar{\kappa}} / \bar{\kappa})_{i}$ are both bijective. Thus, we obtain a $(\bar{\kappa}, \sigma^{2})$-linear bijective map $\tV^{-1} \tF: \mathrm{H}_{1}^{\mathrm{dR}}(B_{\bar{\kappa}} / \bar{\kappa})_{i} \to \mathrm{H}_{1}^{\mathrm{dR}}(B_{\bar{\kappa}} / \bar{\kappa})_{i}$. 
			Denote by $\mathscr{V}_{0}$ the invariant subspace of 
			$\mathrm{H}_{1}^{\mathrm{dR}}(B_{\bar{\kappa}} / \bar{\kappa})_{i}$
			under  $\tV^{-1} \tF$.  
			Then the canonical map 
			$\mathscr{V}_{0} \otimes_{\mathbb{F}_{p^{2}}} \bar{\kappa} \to$ $\mathrm{H}_{1}^{\mathrm{dR}}(B / \bar{\kappa})_{i}=(\mathscr{V}_{t})_{\bar{\kappa}}$ 
			is an isomorphism. For $x, y \in \mathscr{V}_{0}$, we have
			$$
			\{x, y\}_{t}=\langle x,\tF y\rangle_{\lambda_B}=\langle \tV x,  y\rangle_{\lambda_B }^{\sigma}=\langle \tF x, y\rangle_{\lambda_B }^{\sigma}=-\langle  y, \tF x\rangle_{\lambda_B }^{\sigma}=-\{y, x\}_{t}^{\sigma}
			$$
			Thus, $(\mathscr{V}_{t},\{ \, , \,  \}_{t})$ is admissible. The lemma follows.
		\end{proof}
	\end{lem}
	\subsection{Basic correspondence}
	We define a new moduli problem which gives the normalization of the supersingular locus $S_{\r{ss}}.$
	\begin{definition}
		The moduli problem $ {N}$  associates with every 
		$R\in \mathtt{Sch'}_{/\Fpp}$ 
		the set $ {N}(R)$ of equivalence classes of sextuples
		$(B, \lambda_B , \eta_B,
		A, \lambda_A, \eta_A,
		\gamma)
		$ where
		\begin{enumerate}
			\item $(B, \lambda_B,   \eta_B) \in  { T } (R);$
			\item $(A, \lambda_A, \eta_A) \in  {S} (R);$
			\item $\gamma: A \to   B$ is an $O_F$-linear isogeny such that 
			\[
			p \lambda_{  A} 
			= \gamma^\vee \circ \lambda_{  B} \circ  \gamma;
			\]
		\end{enumerate}
	Note that condition (3) implies that $\ker(\gamma)$ is a subgroup scheme of $A[p]$ stable under $O_F$.
		Two septuplets 
		$(B, \lambda_B , \eta_B,
		A, \lambda_A, \eta_A,
		\gamma)$ and 
		$(B', \lambda_B' ,  \eta_B',
		A', \lambda_{ A' }, \eta_{ A' },
		\gamma')$ 
		are equivalent if there are $O_{F}$-linear prime-to-$p$ quasi-isogenies 
		$ \varphi : B \to B^{\prime}$ and 
		$\psi: A \to A'$ such that
		\begin{itemize}
			\item there exists $c \in \dZ_{(p)}^{\times}$such that 
			$\varphi^{\vee} \circ \lambda_{B'} \circ \varphi=c \lambda_B$
			and 
			$\psi^{\vee} \circ \lambda_{A'} \circ \psi=c \lambda_A.$
			\item the $ K^{p}$-orbit of maps 
			$v \mapsto \varphi_{*} \circ \eta_B(v)  $
			for $v \in V' \otimes_{\mathbb{Q}} \mathbb{A}^{\infty, p}$ coincides with $\eta_{B'}.$
			\item the $ K^{p}$-orbit of maps 
			$v \mapsto \varphi'_{*} \circ \eta_{A}(v)  $
			for $v \in V \otimes_{\mathbb{Q}} \mathbb{ A}^{\infty, p}$ coincides with $\eta_{A'}.$
		\end{itemize}
	\end{definition}
	
	We obtain in an obvious way a correspondence  
	\begin{equation}\label{diag:normal}
		\xymatrix{
			& N \ar[ld]_\theta \ar[rd]^\nu  &  \\
			T    &&  S_{\r{ss} }   
		}
	\end{equation}
	\begin{theorem}\label{thm:TN}
		In  diagram \eqref{diag:normal}, take a point
		\[
		t=( B, \lambda_B, \eta_B  ) \in T( \kappa)
		\]
		where $\kappa$ is a field containing $\mathbb{F}_{p^2}$.  
		Put $N_t:= \theta^{ -1} (t  )$, 
		and denote by 
		$( B, \lambda_B, \eta_B ,  \mathcal{A}, \lambda_\cA, \eta_\cA, \gamma )$
		the universal object over the fiber $ {N}_{t}$.
		\begin{enumerate}
			\item
			The fiber $ {N}_{t}$ is a smooth scheme over $\kappa$, with a canonical isomorphism for its tangent bundle
			\[
			\mathcal{T}_{ {N}_{t} / \Fpp} \simeq 
			(\omega_{\mathcal{A}^{\vee}, 1}, \operatorname{ker} \alpha_{*, 1} / \omega_{\mathcal{A}^{\vee}, 1})
			\]
			\item\label{Nfiber} 
			 The assignment sending $ ( B, \lambda_B, \eta_B, A, \lambda_A, \eta_A ; \gamma ) \in   N_t(R)$  for every $R\in \mathtt{Sch'}_{/\kappa}$ to the subbundle
			\[
			U:=  \delta_{*,0}^{-1}(\omega_{A^\vee / R,0})   \subseteq  \r{H}_1^{\r{dR}}(B/R)_0 \cong \r{H}_1^{\r{dR}}(B/\kappa)_0 \otimes_\kappa \cO_R=\mathscr{V}_{t}\otimes_{\kappa}R.
			\]
			induces an isomorphism 
			\[
			\zeta_t :N_t \cong   \r{DL} ( \sV_t , \{ \, ,\,  \} )
			\]
			where $\delta: B \to A$  is the unique quasi-p-isogeny such that $\gamma \circ \delta=p\id_B$ and
			$
			\delta \circ\gamma=p\id_A.$
			In particular, $N_t $ is isomorphic to the Fermat curve $\sC.$
		\end{enumerate}
		\begin{proof}
			See \cite[Theorem 4.2.5]{LTX+22}.
		\end{proof}

	\end{theorem}
	We can define a moduli problem for $S_{\spe}.$ 
	\begin{definition}
	 Let $S_{\spe}(R)$
			be the set of points $(A,\lambda_A,\eta_A) \in S(R)$ for  
			$R\in \mathtt{Sch'}_{/\Fpp},$ where 
			\[
			\tV \omega_{A^\vee/R,0}=0.
			\]
		\end{definition} 
	
		\begin{remark}\label{remark:sspdefequiv}
			The definition is equivalent to $\tV \omega_{A^\vee/R,1}=0$. Indeed, by comparing the rank we have 
			$(\ker\tV)_0= \omega_{A^\vee/R,0},$ which is equivalent to 
			$(\ker\tV)_1= \omega_{A^\vee/R,1}$ by duality. 
		\end{remark}
				\begin{remark}\label{rem:sspsmooth}
			The conditions
			$\omega_{A^\vee/R,0}=(\ker\tV)_0$ and
			$\omega_{A^\vee/R,1}=(\ker\tV)_1$ 
			imply $S_{\spe}$ is smooth of dimension 0.
		\end{remark}
		\begin{definition}
			Let $ M$ be the moduli problem   associating with every 
			$R\in \mathtt{Sch'}_{/\Fpp}$ 
			the set $ M(R)$ of equivalence classes of septuplets
			$(\tilde B, \lambda_{\tilde B} , \eta_{\tilde B},
			A, \lambda_A, \eta_A,\delta'
			)
			$ where
			\begin{enumerate}
				\item $(\tilde B, \lambda_{\tilde B}  ,  \eta_{\tilde B} ) \in  { \tilde T } (R);$
				\item $( A, \lambda_A, \eta_A)\in S(R);$
				\item\label{isotBA} $\delta': \tilde B \to A$ is a $O_F$-linear quasi-$p$-isogeny such that
				\begin{enumerate}
					\item
					$\ker\delta'[p^\infty]\subseteq \tilde B[p];$ 
					\item
					$
					\lambda_{\tilde B }=\delta^{\prime\vee} \circ \lambda_{A} \circ \delta';
					$
					\item the $ {K}^{p }$-orbit of maps $v \mapsto \delta_{*}' \circ  \eta_{\tilde B}(v)$ for $v \in  {V}^{} \otimes_{\mathbb{Q}} \mathbb{A}^{\infty, p}$ coincides with $\eta_A$.
				\end{enumerate}
			\end{enumerate}
			The equivalence relations are defined in a similar way.
		\end{definition}
		There is a natural correspondence
		\[
		\xymatrix
		{
			&M\ar[ld]_{\rho'} \ar[rd]^\rho&\\
			{\tilde T}&& {S}
		}
		\]
		\begin{lem}
			The morphism $\rho$ factors through $S_{\spe}.$ Moreover,
			$M$ is smooth of dimension 0.
			\begin{proof}
				Take a point $(\tilde B, \lambda_{\tilde B} , \eta_{\tilde B},
				A, \lambda_A, \eta_A,\delta'
				)\in M(R)$
				for $R\in \mathtt{Sch'}_{/\Fpp}.$
				By Remark \ref{remark:sspdefequiv} it suffices to show $\tV_A\omega_{A^\vee/R,0}=0.$
				By condition  \eqref{isotBA} and the proof of Lemma 3.4.12(1)(4) of \cite{LTX+22}, we have
				\[
				\rank_{\cO_R}\ker \delta'_{*,0}+\rank_{\cO_R}\ker \delta'_{*,1}=1.
				\]
				We claim 
				$\delta'_{*,1}$ is an isomorphism,
				since otherwise
				$\rank_{\cO_R} \ker\delta'_{*,0}=0$ and
				$\rank_{\cO_R} \im\delta'_{*,0}=3$, 
				which  imply $\rank_{\cO_R} \omega_{A^\vee/R,0}=3$
				by $\delta'_{*,0}\omega_{\tilde B/R,0}\subset \omega_{A^\vee/R,0},$  
				contradicting 
				the signature condition on $A.$ 
				We conclude  that $\im\delta_{*,1}'=\HdR(A/R)_1.$
				Consider the commutative diagram
				\begin{equation}\label{normal1}
					\xymatrix{
						\rH_1^{\r{dR}}(\tilde B/R)_0 \ar[r]^-{ \delta_{*,0}'}  \ar[d]^-{ \tV_{\tilde B}}&
						\rH_1^{\r{dR}}(  A/R)_0   \ar[d]^-{ \tV_{A} } \\
						\rH_1^{\r{dR}}( \tilde B^{(p)}/R)_0 \ar[r]^-{ \delta_{*,1}^{\prime (p)}} &  \rH_1^{\r{dR}}(  A^{(p)}/R)_0
					}.
				\end{equation}
				Thus   we have 
				\[
				\tV_A\omega_{A^\vee/R,0} 
				=
				\tV_A\im \delta'_{*,0}
				= \delta_{*,1} ^{\prime (p)}(\im\tV_{\tilde B})_0= 
				\delta_{*,1} ^{\prime (p)} \omega_{ \tilde B^{(p)\vee} /R,0}=
				( {\delta_{*,1}'} \omega_{ \tilde B  ^\vee /R,1})^{(p)}
				=0
				\]
				where we have used 
				$\omega_{ \tilde B  ^\vee /R,1}=0.$
				We have proved $\rho$ factors through $S_{\spe}.$
				The signature condition and 
				Remark
				\ref{rem:sspsmooth} imply $\tilde B$ and $A$ have
				trivial deformation. 
				Thus $M$ is smooth of dimension 0.
			\end{proof}
			
		\end{lem}

		\begin{lem}\label{lem:tTMSp} 
				The morphism $\rho$ induces   isomorphisms of $\Fpp$-schemes
				\[
				\rho: M\cong   S_{\spe}
				,\quad
					\rho':M \cong \tilde T
					\]
				which are both equivariant under the prime-to-$p$ Hecke correspondence. That is, given $g \in K^p\backslash G(\dA^{\infty,p}) / K'^p$ such that $g^{-1}K^p g \subset K'^p,$ we have a commutative diagram
				\[
				\xymatrix{
					S_{\spe} (\Kp)  \ar[r]^g \ar[d]_{\varphi \bKp}
					&  S_{\spe} (K'^p)  \ar[d]^{\varphi(K'^p) } \\
					\tilde{T} (\Kp) \ar[r]^g
					& \tilde{T}(K'^p)
				}
				\]
			\begin{proof}
				We show that $ \rho$ is an isomorphism.
					Since $M$ and $S_{\spe}$ are smooth of dimension 0, it suffices to check that for every algebraically closed field $\kappa$ containing $\Fpp$,  
					$\rho$ induces a bijection on $\kappa$-points. 
					We will construct an inverse map $\theta$ of $\rho$.
					Given a point $s'=(A , \lambda_A, \eta_A) \in S_{\spe}(\kappa)$. 
					We list   properties of $\cD(A)$:
					\begin{enumerate}
						\item
						$\tV\cD(A)=\tF\cD(A).$ This follows from lifting the definition
						$\tV\omega_{A^\vee/\kappa}=0$ of $S_{\spe}.$
						\item\label{Aselfdual} $\cD(A)_0^{\perp_A}=\cD(A)_1,\cD(A)_1^{\perp_A}=\cD(A)_0$. This is because $\lambda$ is self-dual, or equivalently $\cD(A)^{\perp_A}=\cD(A)$.
						\item \label{chain12}
						We have a chain of $W( \kappa )$-modules
						$$
						p\cD(A)_0 \stackrel{2} \subset \tF\cD(A)_1 \stackrel{1}\subset \cD(A)_0, \quad
						p\cD(A)_1  \stackrel{1}\subset \tF\cD(A)_0 \stackrel{2}\subset \cD(A)_1.
						$$
						This follows from \cite[Lemma 1.4]{Vol10} and in particulier $A$ is of signature type (1,2).
					\end{enumerate}
					Set
					\[
					\cD_{ \tilde B ,0}=\tV\cD(A)_1,\quad  \cD_{ \tilde B ,1}= \cD(A)_1, \quad
					\cD_{ \tilde B }=\cD_{\tilde B ,0}\oplus \cD_{ \tilde B ,1}.
					\]
					We verify that $\cD_{\tilde B}$ is $\tF,\tV$-stable. Indeed,
					since $\cD(A)$ is $\tV^{-1}\tF$-invariant, it suffices to verify the condition for $\tV$: we have  
					$\tV \cD_{ \tilde B }
					=\tV^2\cD(A)_1+ \tV\cD(A)_1
					=\tV\cD(A)_1+ p \cD(A)_1 \subset  \cD_{ \tilde B }$ 
					since $\tV^2=\tF\tV=p$.
					
					The chain \eqref{chain12} implies $  \cD_{\tilde B} \subset \cD(A)$ as $W(\kappa)$-lattices in $\cD(A)[1/p].$
					
					By Dieudonn\'e theory there exists an abelian 3-fold $\tilde B$
					such that $\cD(\tilde B)=\cD_{\tilde B}$, and the injection $   \cD(\tilde B)  \to \cD(A) $ is induced by a prime-to-$p$ isogeny $\delta': \tilde B \to A.$
					Define the endormorphism structure $i_{\tilde B}$ on $\tilde B$ by  $i_{\tilde B}(a)= \delta'^{-1} \circ i(a) \circ \delta'$ for $a \in O_F.$ Then $(\tilde B, i_{\tilde B})$ is an $O_F$-abelian scheme.
					Let $\lambda_{\tilde B}$ be the unique polarization such that 
					$$\lambda_{\tilde B} = \delta'^\vee  \circ \lambda_A  \circ \delta'.$$
					
					The pairings induced by $\lambda_A$ and $\lambda_{\tilde B}$
					have the relation
					$$
					\langle x, y \rangle_{\lambda_A} = \langle x , y \rangle_{\lambda_{\tilde B}} ,\, x,y \in \cD(A).
					$$
					Define the level structure $\eta_{\tilde B}$ by
					$\eta_{\tilde B}=  {\delta'_*}^{-1} \circ \eta_A.$
					We verify
					\begin{enumerate}[resume]
						\item
						$\cD({\tilde B})$ is of signature type (0,3). Indeed, this follows from
						\[
						\Lie(\tilde B) \cong  \cD_{\tilde B}/ \tV \cD_{\tilde B} \cong \cD(A)_1/p\cD(A)_1.
						\]
						\item
						$\ker \lambda_{\tilde B}$ is a finite group scheme of rank $p^2$. Indeed, from covariant Dieudonne theory 
						it is equivalent to show 
						$ \cD({\tilde B}) \stackrel{2}\subset \cD({\tilde B})^ { \perp_{ {\tilde B} } } $. Thus it suffices to show 
						$ \cD( \tilde B)_0    \stackrel{1} \subset \cD( \tilde B )_1^{\perp_{\tilde B} }    .$
						From \eqref{Aselfdual}
						it is equivalent to show
						$   \tF \cD(A)_1 \stackrel{1}\subset\cD(A)_0 $  which comes from $\eqref{chain12}.$
						\item $\ker\delta'[p^\infty]\subset \tilde B[p].$ 
						It suffices to show $p\cD(A) \subset \cD(\tilde B),$
						which is by definition.   
					\end{enumerate}

					Finally we set 
					$\theta(s')=({\tilde B},  \lambda_{ \tilde B } , \eta_{ \tilde B },
					A , \lambda_A, \eta_A,\delta'
					).$
					To verify   $\theta$ is equivariant under prime-to-$p$ Hecke correspondence, it suffices to consider the associativity of the following  diagram
					\[
					\xymatrix{
						V \otimes_{ \dQ } \dA^{\infty,p}
						\ar[r]^g & V \otimes_{ \dQ }   \mathbb{A}^{\infty, p} 
						\ar[r]^{\eta_A}&
						\mathrm{H}_{1}(A, \mathbb{A}^{\infty,p}) \ar[r]^{\delta_{*}^{\prime -1}} & \r{H}_{1}(\tilde{B},  \dA^{\infty,p})
					}
					\]
					for $g \in K^p \backslash G(\dA^{\infty,p}) / K'^p.$
					It is easy to verify $\theta$ and $\rho$ are the inverse of each other.	
		We show that $  \rho'$ is an isomorphism.		Since $M$ and $\tilde T $ are smooth and have dimension 0, it suffices to check that for every algebraically closed field $\kappa$ containing $\Fpp$,  
					$\rho'$ induces a bijection on $\kappa$-points. 
					We will construct an inverse map $\theta'$ of $\rho'$.
					Given 
					$t=( \tilde B ,  \lambda_{\tilde B}   ,  {\eta}_{\tilde B }) \in \tilde{T}( \kappa ),$
					we list   properties of $\cD({\tilde B})$:
					\begin{enumerate}[resume]
						\item  $\tV\cD({\tilde B})_0=\tF \cD(\tilde B)_0.$ In fact, since $\cD({\tilde B})$ is of signature (0,3), \cite[Lemma 1.4]{Vol10} gives
						\[
						\cD({\tilde B})_0=\tV\cD({\tilde B})_1=\tF\cD({\tilde B})_1.
						\]
						\item \label{duallength}
						$
						\cD({\tilde B})_1 \stackrel{1} \subset \cD({\tilde B})_0^{\perp_{\tilde B}} 
						$ and $  \cD({\tilde B})_0 \stackrel{1}\subset \cD({\tilde B})_1^ {\perp_{\tilde B}}  $.
						Indeed, since $\ker\lambda_{\tilde B}[p^\infty]$ is a $\tilde B[p]$-subgroup scheme of rank $p^2$, by covariant Dieudonn\'e theory we have 
						$
						\cD({\tilde B}) 
						\stackrel{2} \subset
						\cD({\tilde B})^{\perp_{\tilde B } },
						$
						and the claim follows.
						
						\item\label{chain03} We have the chain of $W(\kappa)$-lattice
						\[
						\cD({\tilde B})_1 
						\stackrel{1}\subset 
						\tV^{-1}\cD({\tilde B})_1^{\perp_{ \tilde B} } \stackrel{2}\subset 
						\frac{1}{p}\cD({\tilde B})_1.
						\]
						Indeed, $\ker \lambda_{\tilde B} \subset {\tilde B}[p]$
						gives $\cD({\tilde B})_0^{\perp_{\tilde B}} \subset (1/p)\cD({\tilde B})_1.$
						The claim comes from \eqref{duallength} and the fact that $\cD({\tilde B})_1^{\perp_{\tilde B}}=( \tV^{-1} \cD({\tilde B})_0)^{\perp_{\tilde B}}= 
						\tF (\cD({\tilde B})_0)^{\perp_{\tilde B}}$. 
					\end{enumerate}
					
					We set
					\[
					\cD_{A,0}=\cD(\tilde B)_1^{\perp_{ \tilde B }}, 
					\cD_{A,1}=\cD({\tilde B})_1, 
					\cD_A=\cD_{A,0}\oplus \cD_{A,1}.
					\]
					That $\cD_A$ 
					is $\tF,\tV$-stable follows from \eqref{chain03}.
					By covariant Dieudonn\'e theory there exists an abelian 3-fold $A$ such that $\cD(A)=\cD_A$, and the inclusion $\cD({\tilde B}) \to \cD(A)$ is induced by a prime-to-$p$ isogeny $\delta':{\tilde B} \to A.$
					Define the endormorphism structure $i_A$ on $A$ by 
					$i_A(a)= \delta' \circ i_{\tilde B}(a) \circ \delta'^{-1}$ for $a \in O_F.$
					Then $( A, i_A)$ is an $O_F$-abelian scheme.
					Let $\lambda_A$ be the unique polarization such that 
					$$\lambda_{\tilde B} = \delta'^\vee  \circ \lambda_A \circ \delta'.$$
					The pairings induced by $\lambda_A$ and $\lambda_{\tilde B}$ have the relation
					\[
					\langle x, y \rangle_{\lambda_A} =  \langle x , y \rangle_{\lambda_{\tilde B}} , \ x,y \in \cD(A).
					\]
					Define the level structure $\eta_A$ by
					$\eta_A=  \delta'_* \circ \eta_{\tilde B}.$
					We verify
					\begin{enumerate}[resume]
						\item $\cD(A)$ is of signature (1,2):  
						calculate the Lie algebra  
						\[
						\frac{\cD(A)}{\tV\cD(A)}=
						\frac{\cD(\tilde B)_1^{\perp_{\tilde B} } + \cD({\tilde B})_1 } 
						{ \tV \cD({\tilde B})_1  + \tV\cD(\tilde B)_1^{\perp_{\tilde B} }  }.
						\]
						The claim follows from \eqref{chain03}.
						\item $\cD(A)$ is self-dual with respect to $\langle , \rangle_{\lambda_A}$. Indeed, it suffices to show $\cD(A)_0^{\perp_A}=\cD(A)_1.$ Since $\cD(A)_0^{\perp_A}= \cD(A)_0^{\perp_{\tilde B}}$, it is enough to verify
						$ \cD(A)_0^{\perp_{\tilde B}}=\cD(A)_1$, which is exactly our construction.
					\end{enumerate}  
					Finally we set
					$\theta(t')=({\tilde B},  \lambda_{ \tilde B } , \eta_{ \tilde B },
					A , \lambda_A, \eta_A,\delta'
					).$
					The equivariance under prime-to-$p$ Hecke correspondence is clear.
			\end{proof}
		\end{lem}
		
	\subsection{The geometry of $S_0(p)$}\label{subsec:geoS0p}
	We define three closed subschemes  $Y_i, i=0,1,2$ of $S_0(p)$ over $\Fpp$ as follows: for $R\in \mathtt{Sch'}_{/\Fpp}$, 
	a point  $s=(A ,  \lambda_A, \eta_A,   \tilde A 
	, \lambda_{\tilde A}, \eta_{\tilde A}, 
	\alpha) 
	\in S_0(p)(R)$  belongs to 
	\begin{itemize}
		\item\label{defY0}
		 $ Y_0(R)$ if and only if 
		$
		\omega_{\tilde A^\vee/R,0}=\r{im} \alpha_{*,0};  $
		\item\label{defY1}
		$ Y_1(R)$ if and only if 
		$
		\omega_{A^\vee/R,1}=  \ker \alpha_{*,1};  $
		\item\label{defY2}
		  $Y_2(R)$ if and only if   
		$  \omega_{\tilde A^\vee/R,1}= 
		\rH_1^{\r{dR}}(\tilde A/R)_0^{\perp_{\tilde A } }.
		$
	\end{itemize} 
	
	\begin{remark}
		In \cite{dSG18},  the authors define two strata $\ol{Y}_m, \ol{Y}_{et}$. We will see that  $Y_0$ coincides with  their $\ol{Y}_m$ and
		$Y_1$ coincides with their $\ol{Y}_{et}.$
		
	\end{remark}

	We are going to show these three strata are all smooth of dimension 2.
	
	\begin{lem}\label{lem:Y0Y1Y2induce}
		Take $s=(A, \lambda_{A}, \eta_{A}, \tilde A, \lambda_{\tilde{A}}, \eta_{\tilde{A}}, \alpha) \in S_{0}(p)(R)$ for a scheme $R \in  \mathtt{Sch}^{\prime}_{ / \mathbb{F}_{p^{2}} }.$
		\begin{enumerate}
			\item
			If $s\in Y_0(R)$ then
			\begin{enumerate}
				\item\label{Y0definduce1}
				$\omega_{\tilde A^\vee/R,1}  \subseteq \mathrm{im}\alpha_{*,1};$
				\item \label{Y0definduce2}
				$ (\ker\tV_A)_1=\ker\alpha_{*,1}.$
				\end {enumerate}
				\item
				If $s\in Y_1(R)$ then 
				\begin{enumerate}
					\item\label{Y1definduce1}
					$\ker\alpha_{*,0} \subseteq \omega_{  A^\vee/R,0} ;$ 
					\item \label{Y1definduce2}
					$
					\alpha_{*,0} ({\omega}_{  A^{\vee}/R, 0})= 
					\HdR(\tilde A/R)_1^{\perp_{\tilde A}}.
					$
				\end{enumerate}
				\item
				If $s\in Y_2(R)$ then 
				\begin{enumerate}
					\item\label{Y2definduce1}
					$\ker\alpha_{*,0} \subseteq \omega_{  A^\vee/R,0}. $ 
				\end{enumerate}
			\end{enumerate}
			\begin{proof}
				Denote by $\breve\alpha:\tilde A \to A$ the unique isogeny 
				such that $\breve\alpha\circ\alpha=p\id_A$ and 
				$\alpha\circ \breve\alpha=p\id_{\tilde A}.$
				
				\begin{enumerate}
					\item
					\begin{enumerate}
						\item
						The condition $\omega_{\tilde A^\vee/R,0}  = \mathrm{im}\alpha_{*,0}$ implies 
						$\omega_{\tilde A^\vee/R,0} ^{\perp_{\tilde A}} = (\mathrm{im}\alpha_{*,0} )^{\perp_{\tilde A}}.$
						On the other hand, we have  
						$\ang{\im \alpha_{*,0} , \im \alpha_{*,1}
						}_{\lambda_{\tilde A}}=
						\ang{\HdR(\tilde A/R)_0 ,\breve \alpha_{*,1} \im \alpha_{*,1}
						}_{\lambda_{\tilde A}}=0,
						$
						which implies 
						$ \mathrm{im}\alpha_{*,1}  =
						(\mathrm{im}\alpha_{*,0} )^{\perp_{\tilde A}}$ by comparing the rank. We also have 
						$
						\ang{\omega_{\tilde A^\vee/R,0} ,\omega_{\tilde A^\vee/R,1}  }_{\lambda_{\tilde A}}=0
						$, thus \eqref{Y0definduce1} follows.
						\item
						It suffices to show $\ker \alpha_{*,1} \subseteq (\ker\tV_A)_1.$
						The condition \eqref{Y0definduce1} implies
						$\omega_{\tilde A^\vee/R,1}  \subseteq \mathrm{im}\alpha_{*,1}=\ker\breve\alpha_{*,1}.$ 
						We also have $(\ker\tV_{  A})_1 = (\im \tF_A)_1.$
						Consider the following commutative diagram
						\begin{equation}\label{diag:commuteF}
							\xymatrix{
								\rH_1^{\r{dR}}( \tilde A/R)_1 \ar[r]^-{\breve \alpha_{*,1} } \ar[d]^-{ \tV_{\tilde A}}&
								\rH_1^{\r{dR}}(    A /R)_1   \ar[d]^-{ \tV_{  A}} \\
								\rH_1^{\r{dR}}(  \tilde A^{(p)} /R)_0  \ar[r]^-{ \breve\alpha_{*,0}^{(p)}  } & 
								\rH_1^{\r{dR}}(    A^{(p)}/R)_0
							}. 
						\end{equation}
						Thus we have 
						\[
						\tV_A\ker \alpha_{*,1} =
						\tV_A\im\breve \alpha_{*,1} =
						\breve \alpha_{*,0}^{(p)}(\im\tV_{\tilde A})_0
						=\breve \alpha_{*,0}^{(p)}  \omega_{ (\tilde A^{(p)})^\vee/R,0} 
						=(\breve \alpha_{*,1}  \omega_{ \tilde A ^\vee/R,1})^{(p)}
						=0,
						\]
						thus \eqref{Y0definduce2} follows.
					\end{enumerate}
					\item
					\begin{enumerate}
						\item
						The condition $  \omega_{A^\vee/R,1}=  \ker\alpha_{*,1}$ implies 
						$
						\omega_{  A^\vee/R,1} ^{\perp_{  A}} = 
						(\ker \alpha_{*,1} )^{\perp_{  A}}.$
						On the other hand, we have
						$\omega_{  A^\vee/R,0}=\omega_{  A^\vee/R,1}^{\perp_A}$
						and 
						$\ang{ \ker \alpha_{*,0} ,\ker \alpha_{*,1}
						}_{\lambda_{  A}}= 
						\ang{ \im\breve\alpha_{*,0} ,\ker \alpha_{*,1}
						}_{\lambda_{  A}}=
						\ang{\HdR(\tilde A/R)_0, \alpha_{*,1} \ker \alpha_{*,1}
						}_{\lambda_{  A}}=
						0.$ 
						Thus \eqref{Y1definduce1} follows.
						\item $\eqref{Y1definduce1} $ implies $\rank_{\cO_R} \alpha_{*,0} {\omega}_{  A^{\vee}/R, 0}=1.$
						On the other hand, we have
						\[
						\ang{
							\alpha_{*,0}  {\omega}_{  A^{\vee}/R, 0},
							\HdR(\tilde A/R)_1
						}_{\lambda_{  \tilde A}}=
						\ang{
							{\omega}_{  A^{\vee}/R, 0},
							\breve \alpha_{*,1} \HdR(\tilde A/R)_1
						}_{\lambda_{  \tilde A}}=
						\ang{
							{\omega}_{  A^{\vee}/R, 0},
							\ker \alpha_{*,1}  
						}_{\lambda_{    A}}=
						\ang{
							{\omega}_{  A^{\vee}/R, 0},
							{\omega}_{  A^{\vee}/R, 1}
						}_{\lambda_{    A}}=0.
						\]
						Thus 
						$
						\alpha_{*,0} {\omega}_{  A^{\vee}/R, 0}\subseteq
						\HdR(\tilde A/R)_1^{\perp_{\tilde A}}
						$. By comparing the rank  \eqref{Y1definduce2} follows.
					\end{enumerate}
					\item
					\begin{enumerate}
						\item 
						Since 
						$\omega_{A^\vee/R,0}^{\perp_A}=
						\omega_{A^\vee/R,1} ,$
						by taking the dual it suffices to show 
						$\omega_{A^\vee/R,1} \subset (\ker \alpha_{*,0} )^{\perp_A}.$
						Since $ \ker \alpha_{*,0} =\im \breve \alpha_{*,0},$
						it suffices to show
						$\ang{ \omega_{A^\vee/R,1} , \im \breve \alpha_{*,0}}_{\lambda_A} =0.$ 
						By the equality 
						$\ang{\alpha_* x, y}_{\lambda_{\tilde A} }= \ang{  x, \breve\alpha_* y}_{\lambda_A}$
						it suffices to show 
						$\ang{ \alpha_{*,1} \omega_{A^\vee/R,1} , 
							\HdR(\tilde A/R)_0}_{\lambda_{\tilde A}}=0,$ 
						which follows from the conditions
						$ \omega_{\tilde A^\vee/R,1}= 
						\rH_1^{\r{dR}}(\tilde A^\vee/R)_0^{\perp_{\tilde A } }$   and 
						$ \alpha_{*,1} \omega_{A^\vee/R,1} \subseteq \omega_{\tilde A^\vee/R,1} .$
					\end{enumerate}
				\end{enumerate}
			\end{proof}
		\end{lem}
		
		\begin{proposition}\label{prop:Y0Y1Y2smooth}
			\begin{enumerate}
				\item  \label{prop:Y0smooth}
				$Y_0$   is smooth of dimension 2 over $\Fpp.$ 
				Moreover, 
				let $(\cA, \tilde \cA, \alpha)$ denote the universal object on
				$Y_0$.
				Then the tangent bundle $\cT_{Y_0 /\Fpp}$ of $Y_0$ fits into an exact sequence
				\begin{multline}\label{eq:Y0tangent}
					0 \to \cH om (  {\omega}_{  \cA^{\vee}, 1},
					\alpha_{*,1}^{-1}  {\omega}_{\tilde \cA^{\vee},1}/ {\omega}_{  \cA^{\vee}, 1})
					\to
					\cT_{Y_0/\Fpp} 
					\to 
					\cH om  
					(\alpha_{*,1}^{-1}  {\omega}_{\tilde \cA^{\vee},1}/\ker\alpha_{*,1}, 
					\HdR(     \cA )_{ 1} /
					\alpha_{*,1}^{-1}  {\omega}_{\tilde \cA^{\vee},1})
					\to
					0
				\end{multline}
				
				\item   \label{prop:Y1smooth}
				$Y_1$  is smooth of dimension 2 over $\Fpp.$ 
				Moreover, 
				let $(\cA, \tilde \cA, \alpha)$ denote the universal object on
				$Y_1$.
				Then the tangent bundle $\cT_{Y_1 /\Fpp}$ of $Y_1$ fits into an exact sequence
				\begin{multline}
					0 \to \cH om (  {\omega}_{  \tilde \cA^{\vee}, 1},
					{\omega}_{  \tilde \cA^{\vee}, 0}^{\perp_{\tilde \cA}}
					/ {\omega}_{\tilde  \cA^{\vee}, 1})
					\to
					\cT_{Y_1/\Fpp} 
					\to 
					\cH om  
					(  {\omega}_{\tilde \cA^{\vee},0}/
					\HdR(\tilde \cA)_1^{\perp_{\tilde \cA}}, 
					\HdR(   \tilde  \cA )_{ 1} /
					{\omega}_{\tilde \cA^{\vee},0})
					\to
					0
				\end{multline}

				\item \label{prop:Y2smooth}
				$Y_2$ is smooth of dimension 2 over $\Fpp.$
				Moreover, 
				let $(\cA, \tilde \cA, \alpha)$ denote the universal object on
				$Y_2$.
				Then the tangent bundle $\cT_{Y_2 /\Fpp}$ of $Y_2$ fits into an exact sequence
				\begin{multline}\label{eq:Y2tangent}
					0 \to 
					\cH om (  {\omega}_{  \tilde  \cA^{\vee}, 0}/
					\alpha_{*,0} {\omega}_{      \cA^{\vee}, 0},
					\HdR(   \tilde   \cA )_{ 0}
					/ {\omega}_{   \tilde \cA^{\vee}, 0})
					\to
					\cT_{Y_2/\Fpp} 
					\to 
					\cH om  
					(  {\omega}_{  \cA^{\vee},0}/\ker\alpha_{*,0},
					\HdR(      \cA )_{ 0} /
					{\omega}_{  \cA^{\vee},0})
					\to
					0.
				\end{multline}
			\end{enumerate}
			\begin{proof}
				\begin{enumerate}
					\item
					We show $Y_0$ is formally smooth using deformation theory. Consider a closed immersion $R \hookrightarrow \hat{R}$ in $\mathtt{Sch}_{ / \dF_{p^2}}^{\prime}$ defined by an ideal sheaf $\mathcal{I}$ with $\mathcal{I}^{2}=0$. Take a point 
					$
					y=( A, \lambda_A, \eta_A,\tilde{A} ,  \lambda_{\tilde A},\eta_{\tilde A},  \alpha ) \in Y_0(R).$
					By Proposition \ref{thm:STGM}
					lifting $y$ to an $\hat{R}$-point is equivalent to lifting
					\begin{itemize}
						\item
						$\omega_{A^{\vee} / R, 0}$ (resp. $ \omega_{\tilde A^{\vee} / R,0} $) to a rank 2 subbundle 
						$\hat{\omega}_{A^{\vee},0}$ (resp. $\hat \omega_{\tilde A^{\vee} ,0} $)
						of $\mathrm{H}_{1}^{\text {cris }}(A / \hat{R})_{0}$ 
						( resp. $\mathrm{H}_{1}^{\text {cris }}(\tilde A / \hat{R})_{0}$), 
						\item
						$\omega_{A^{\vee} / R, 1}$ (resp. $\omega_{\tilde A^{\vee} / R,1} $)
						to a rank 1 subbundle $\hat{\omega}_{A^{\vee},1}$  (resp. $\hat \omega_{\tilde A^{\vee} ,1} $) 
						of $\mathrm{H}_{1}^{\text {cris }}(A / \hat{R})_{1}$ 
						( resp. $\mathrm{H}_{1}^{\text {cris }}(\tilde A / \hat{R})_{1}$), 
					\end{itemize}
					subject to the following requirements
					\begin{enumerate}
						\item\label{Aorth}
						$\hat{\omega}_{A^{\vee},0}$ and $\hat{\omega}_{A^{\vee},1}$ are orthogonal complement of each other under $\langle ~,~\rangle_{\lambda_A}^{\text {cris }} $ (\ref{paircris});
						\item \label{condY2cris}
						$\hat{\omega}_{\tilde A^{\vee},0}$ and $\hat{\omega}_{\tilde A^{\vee},1}$ are orthogonal under $\langle ~,~\rangle_{\lambda_{\tilde A}}^{\text {cris }} $ (\ref{paircris});
						\item \label{liftomegaAtA}
						$  \hat \omega_{  A^\vee,1} \subseteq  \alpha_{*,1} ^{-1} \hat \omega_{\tilde A^\vee,1}  ;$
						\item  $\hat \omega_{\tilde A^\vee,0}  =  \alpha_{*,0} \mathrm{H}_{1}^{\text {cris }}(   \tilde  A / \hat{R})_0;$
					\end{enumerate}
					Since $\langle~,~\rangle_{\lambda_A,0}^{\text {cris }}$ is a perfect pairing, $\hat{\omega}_{A^{\vee}, 0}$ is uniquely determined by $\hat{\omega}_{A^{\vee}, 1}$ by (\ref{Aorth}).
					Moreover,
					$\hat{\omega}_{\tilde A^{\vee},0}$  is    uniquely determined by $\mathrm{H}_{1}^{\text {cris }}(\tilde A / \hat{R})_0.$ Therefore, it suffices to give the lifts $\hat{\omega}_{A^\vee,1}$ and $\hat{\omega}_{\tilde{A}^\vee, 1}$ subject to condition \eqref{liftomegaAtA} above.
					But
					lifting $ {\omega}_{\tilde A^{\vee}/R,1}$
					is the same as lifting its preimage
					$\alpha_{*,1}^{-1}  {\omega}_{\tilde A^{\vee}/R,1}$
					to a rank 2 subbundle $\hat{\omega}'_{  A^{\vee}, 1}$ of 
					$\mathrm{H}_{1}^{\text {cris }}(\tilde A / \hat{R})_{0}$
					containing $\ker\alpha_{*,1}.$    
					Thus the tangent space  $T_{Y_0/\Fpp,y}$ at $y$ fits canonically into an exact sequence
					\begin{multline}
						0 \to \cH om (  {\omega}_{  A^{\vee}/R, 1},
						\alpha_{*,1}^{-1}  {\omega}_{\tilde A^{\vee}/R,1}/ {\omega}_{  A^{\vee}/R, 1})
						\to
						T_{Y_0/\Fpp,y} \\
						\to 
						\cH om  
						(\alpha_{*,1}^{-1}  {\omega}_{\tilde A^{\vee}/R,1}/\ker\alpha_{*,1}, 
						\HdR(     A/R )_{ 1} /
						\alpha_{*,1}^{-1}  {\omega}_{\tilde A^{\vee}/R,1})
						\to
						0
					\end{multline}
					Thus, $Y_0$ is formally smooth over $\Fpp$ of dimension 2.
					\item  
					Now we show $Y_1$ is formally smooth. Consider a closed immersion $R \hookrightarrow \hat{R}$ in $\mathtt{Sch}_{ / \dF_{p^2}}^{\prime}$ defined by an ideal sheaf $\mathcal{I}$ with $\mathcal{I}^{2}=0$. Take a point 
					$
					y=( A, \lambda_A, \eta_A,\tilde{A} ,  \lambda_{\tilde A},\eta_{\tilde A},  \alpha ) \in Y_1(R).$
					By proposition \ref{thm:STGM}
					to lift $y$ to an $\hat{R}$-point is equivalent to lift
					\begin{itemize}
						\item
						$\omega_{A^{\vee} / R, 0}$ (resp. $ \omega_{\tilde A^{\vee} / R,0} $) to a rank 2 subbundle 
						$\hat{\omega}_{A^{\vee},0}$ (resp. $\hat \omega_{\tilde A^{\vee} ,0} $)
						of $\mathrm{H}_{1}^{\text {cris }}(A / \hat{R})_{0}$ 
						( resp. $\mathrm{H}_{1}^{\text {cris }}(\tilde A / \hat{R})_{0}$), 
						\item
						$\omega_{A^{\vee} / R, 1}$ (resp. $\omega_{\tilde A^{\vee} / R,1} $)
						to a rank 1 subbundle $\hat{\omega}_{A^{\vee},1}$  (resp. $\hat \omega_{\tilde A^{\vee} ,1} $) 
						of $\mathrm{H}_{1}^{\text {cris }}(A / \hat{R})_{1}$ 
						( resp. $\mathrm{H}_{1}^{\text {cris }}(\tilde A / \hat{R})_{1}$), 
					\end{itemize}
					subject to the following requirements
					\begin{enumerate}\label{condlift}
						\item\label{Aorth}
						$\hat{\omega}_{A^{\vee},0}$ and $\hat{\omega}_{A^{\vee},1}$ are orthogonal complement of each other under $\langle ~,~\rangle_{\lambda_A,0}^{\text {cris }} $ (\ref{paircris});
						\item \label{tAorth}
						$\hat{\omega}_{\tilde A^{\vee},0}$ and $\hat{\omega}_{\tilde A^{\vee},1}$ are orthogonal under $\langle ~,~\rangle_{\lambda_{\tilde A,0}}^{\text {cris }} ;$
						\item\label{omegafonc}
						$  \alpha_{*,0}  \hat \omega_{  A^\vee,0} \subseteq  \hat \omega_{\tilde A^\vee,0}  ;$
						\item \label{condY1lift}
						$\hat \omega_{  A^\vee,1} =\ker\alpha_{*,1}.$ 
					\end{enumerate}
					Since $\langle~,~\rangle_{\lambda_A,0}^{\text {cris }}$ is a perfect pairing, $\hat{\omega}_{A^{\vee}, 0}$ is uniquely determined by $\hat{\omega}_{A^{\vee}, 1}=\ker\alpha_{*,1}
					$ by \eqref{Aorth} and \eqref{condY1lift}.
					On the other hand, we have 
					$
					\alpha_{*,0} {\omega}_{  A^{\vee}/R, 0}= 
					\HdR(\tilde A/R)_1^{\perp_{\tilde A}}
					$ by Lemma \ref{lem:Y0Y1Y2induce}\eqref{Y1definduce2}.
					To summarize, lifting $y$ to an $\hat{R}$-point is equivalent to lifting $\omega_{\tilde A^{\vee} / R,0}$ to a subbundle $\hat{\omega}_{\tilde A^{\vee}, 0}$ containing 
					$\mathrm{H}_{1}^{\text {cris }}(\tilde A / \hat{R})_{1}^{\perp_{\tilde A}},$
					and
					lifting $\omega_{  \tilde A^{\vee} / R,1}$ to a subbundle $\hat{\omega}_{  \tilde A^{\vee}, 1} $ of 
					$\hat{\omega}_{\tilde A^{\vee}, 0}^{\perp_{\tilde A}}$
					where the latter has $\cO_{\hat R}$-rank 2. 
					Thus the tangent space  $\cT_{Y_1/\Fpp,y}$ at $y$ fits canonically into an exact sequence
					\begin{multline}
						0 \to \cH om (  {\omega}_{  \tilde A^{\vee}/R, 1},
						{\omega}_{  \tilde A^{\vee}/R, 0}^{\perp_{\tilde A}}
						/ {\omega}_{\tilde  A^{\vee}/R, 1})
						\to
						\cT_{Y_1/\Fpp,y} \\
						\to 
						\cH om  
						(  {\omega}_{\tilde A^{\vee}/R,0}/
						\HdR(\tilde A/R)_1^{\perp_{\tilde A}}, 
						\HdR(   \tilde  A/R )_{ 1} /
						{\omega}_{\tilde A^{\vee}/R,0})
						\to
						0
					\end{multline}

					Thus, $Y_1$ is formally smooth over $\Fpp$ of dimension 2.
					\item We show $Y_2$ is formally smooth using deformation theory. Consider a closed immersion $R \hookrightarrow \hat{R}$ in $\mathtt{Sch}_{ / \dF_{p^2}}^{\prime}$ defined by an ideal sheaf $\mathcal{I}$ with $\mathcal{I}^{2}=0$. Take a point 
					$
					y=( A, \lambda_A, \eta_A,\tilde{A} ,  \lambda_{\tilde A},\eta_{\tilde A},  \alpha ) \in Y_2(R).$
					We return to the proof of Proposition \ref{prop:Y2smooth}.
					By proposition \ref{thm:STGM}
					to lift $y$ to an $\hat{R}$-point is equivalent to lifting
					\begin{itemize}
						\item
						$\omega_{A^{\vee} / R, 0}$ (resp. $\omega_{\tilde A^{\vee} / R,0} $) to a rank 2 subbundle 
						$\hat{\omega}_{A^{\vee},0}$ (resp. $\omega_{\tilde A^{\vee} ,0} $)
						of $\mathrm{H}_{1}^{\text {cris }}(A / \hat{R})_{0}$ 
						( resp. $\mathrm{H}_{1}^{\text {cris }}(\tilde A / \hat{R})_{0}$), 
						\item
						$\omega_{A^{\vee} / R, 1}$ (resp. $\omega_{\tilde A^{\vee} / R,1} $)
						to a rank 1 subbundle $\hat{\omega}_{A^{\vee},1}$  (resp. $\omega_{\tilde A^{\vee} ,1} $) 
						of $\mathrm{H}_{1}^{\text {cris }}(A / \hat{R})_{1}$ 
						( resp. $\mathrm{H}_{1}^{\text {cris }}(\tilde A / \hat{R})_{1}$), 
					\end{itemize}
					subject to the following requirements
					\begin{enumerate}
						\item\label{Aorth}
						$\hat{\omega}_{A^{\vee},0}$ and $\hat{\omega}_{A^{\vee},1}$ are orthogonal complement of each other under $\langle ~,~\rangle_{\lambda_A,0}^{\text {cris }} $ (\ref{paircris});
						\item \label{Y2tA0tA1}
						$\hat{\omega}_{\tilde A^{\vee},1}$ is the  orthogonal complement of  $\mathrm{H}_{1}^{\text {cris }}(\tilde A / \hat{R})_0$ under $\langle ~,~\rangle_{\lambda_{\tilde A},0}^{\text {cris }} $;
						\item  \label{Y2A0tA0}
						$  {\alpha}_{*,i}  \hat{\omega}_{A^{\vee},i } \subseteq \hat{\omega}_{\tilde A^{\vee},i}$ for $i=0,1.$
						\item \label{Y2kerA0}
						$\ker \alpha_{*,0} \subseteq \hat\omega_{A^\vee,0}
						( \text{Lemma } \ref{lem:Y0Y1Y2induce}\eqref{Y2definduce1})
						.$
					\end{enumerate}
					Since $\langle~,~\rangle_{\lambda_A,0}^{\text {cris }}$ is a perfect pairing, $\hat{\omega}_{A^{\vee}, 1}$ is uniquely determined by $\hat{\omega}_{A^{\vee}, 0}$ by (\ref{Aorth}).
					Moreover,
					$\hat{\omega}_{\tilde A^{\vee},1}$ is  uniquely determined by $\mathrm{H}_{1}^{\text {cris }}(\tilde A / \hat{R})_0$ by (\ref{condY2cris}).
					Given a lift $ {\hat \omega}_{A^{\vee}, 0} $ with condition \eqref{Y2kerA0}
					and define $ {\hat \omega}_{A^{\vee}, 1} := {\hat \omega}_{A^{\vee}, 0}^{\perp_A}. $
					We claim $\alpha_{*,1} {\hat \omega}_{A^{\vee}, 1} 
					\subset 
					\hat \omega_{\tilde A^\vee,1}.$
					Indeed, aince 
					$\omega_{\tilde A^\vee,1}=
					\mathrm{H}_{1}^{\text {cris }}(\tilde A / \hat{R})_0
					^{\perp_{\tilde A } },$
					it suffices to check 
					$
					\ang{
						\alpha_{*,1} {\hat \omega}_{  A^{\vee}, 1},
						\mathrm{H}_{1}^{\text {cris }}(\tilde A / \hat{R})_0
					}_{\lambda_{  \tilde A}}=0.
					$
					However, we have
					\begin{multline}
						\ang{
							\alpha_{*,1} {\hat \omega}_{  A^{\vee}, 1},
							\mathrm{H}_{1}^{\text {cris }}(\tilde A / \hat{R})_0
						}_{\lambda_{  \tilde A}}=
						\ang{
							{\hat \omega}_{  A^{\vee}, 1},
							\breve \alpha_{*,0} \mathrm{H}_{1}^{\text {cris }}(\tilde A / \hat{R})_0
						}_{\lambda_{  \tilde A}}=
						\ang{
							{\hat \omega}_{  A^{\vee}, 1},
							\ker\alpha_{*,0}
						}_{\lambda_{  \tilde A}}\subset
						\ang{
							{\hat \omega}_{  A^{\vee}, 1},
							{\hat \omega}_{  A^{\vee}, 0}
						}_{\lambda_{  \tilde A}}=0.
					\end{multline}
					The claim follows.
					To summarize, lifting $y$ to an $\hat{R}$-point is equivalent to
					lifting $\omega_{     A^{\vee} / R,0}$ to a subbundle $\hat{\omega}_{      A^{\vee}, 0}$  of 
					$\mathrm{H}_{1}^{\text {cris }}(    A / \hat{R})_0$
					containing $\ker\alpha_{*,0}$
					and 
					lifting $\omega_{   \tilde A^{\vee} / R,0}$ to a subbundle
					$\hat{\omega}_{ \tilde  A^{\vee}, 0}$ of
					$\mathrm{H}_{1}^{\text {cris }}(\tilde A / \hat{R})_0$
					containing $\alpha_{*,0}\hat{\omega}_{    A^{\vee}, 0}.$
					Thus the tangent space  $T_{Y_2/\Fpp,y}$ at $y$ fits canonically into an exact sequence
					\begin{multline}
						0 \to 
						\cH om (  {\omega}_{  \tilde  A^{\vee}/R, 0}/
						\alpha_{*,0} {\omega}_{      A^{\vee}/R, 0},
						\HdR(   \tilde   A/R )_{ 0}
						/ {\omega}_{   \tilde A^{\vee}/R, 0})
						\to
						T_{Y_2/\Fpp,y} \\
						\to 
						\cH om  
						(  {\omega}_{  A^{\vee}/R,0}/\ker\alpha_{*,0},
						\HdR(      A/R )_{ 0} /
						{\omega}_{  A^{\vee}/R,0})
						\to
						0.
					\end{multline}
					Thus $Y_2$ is smooth over $\Fpp$ of  dimension 2.
				\end{enumerate}
			\end{proof}
		\end{proposition}

		\begin{lem}\label{lem:unionIw}
			$S_0(p)$ is the union of three strata defined over $\Fpp$
			\[
			S_0(p)=Y_0 \cup Y_1 \cup Y_2.
			\]
			\begin{proof}
				
				By Hilbert's Nullstellensatz,  it suffices to show that 
				\[S_0(p)(\kappa)=Y_0(\kappa)\cup Y_1(\kappa)\cup Y_2(
				\kappa)\] for an algebraically closed field $\kappa$ of characteristic $p$.
				Take $s=(A ,  \lambda_A, \eta_A,   \tilde A 
				, \lambda_{\tilde A}, \eta_{\tilde A}, 
				\alpha) 
				\in S_0(p)(\kappa).$    
				Suppose $s \notin Y_0(\kappa) \cup    Y_1(\kappa),$
				that is, $\omega_{\tilde A^\vee/R,0}  \neq \mathrm{im}\alpha_{*,0}$ and  $ \omega_{A^\vee/R,1} \neq\ker\alpha_{*,1}$. It follows that 
				$ \omega_{A^\vee/R,1} \cap \ker\alpha_{*,1}=\{0\}$
				by the rank condition and therefore 
				$\alpha_{*,1} $ induces an isomorphism
				$ \omega_{\tilde A^\vee/R,1}=\alpha_{*,1} \omega_{A^\vee,1}.$ 
				Thus 
				$ \ang{ \im\alpha_{*,0},  \omega_{\tilde A^\vee/R,1}} _{\lambda_{\tilde A}}=
				\ang{ \im\alpha_{*,0},  \alpha_{*,1} \omega_{A^\vee,1} }_{\lambda_{\tilde A}}=0.$ On the other hand, 
				we have 
				$
				\ang{\omega_{\tilde A^\vee/R,0},    \omega_{\tilde A^\vee/R,1} }_{\lambda_{\tilde A}}=0.
				$
				Since  $\omega_{\tilde A^\vee/R,0}  \neq \mathrm{im}\alpha_{*,0}$ , we conclude 
				$  
				\ang{\rH_1^{\r{dR}}(\tilde A/R)_0,    \omega_{\tilde A^\vee/R,1} }_{\lambda_{\tilde A}}=0
				. $
				Thus $s \in Y_2(\kappa)$ and the lemma follows.
			\end{proof}
		\end{lem}
		\subsection{Relation between strata of $S_0(p)$ and $S$}
		
		\begin{definition}
			Let   $S^{\#}$ be the moduli scheme  that associates with every scheme $R\in\mathtt{Sch'}_{/\Fpp},$ the
			isomorphism classes of pairs 
			$(A, \lambda_A , \eta_A, \mathcal{P}_{0})$ 
			where
			\begin{enumerate}
				\item
				$ (A, \lambda_A , \eta_A) \in S(R);$
				\item
				$\mathcal{P}_{0}$  is a line subbundle of 
				$  \operatorname{ker} (\tV: \omega_{A^\vee / R,0}   \to  \omega_{{A^\vee}^{(p)} / R,0} ).$ 
			\end{enumerate}
		\end{definition}	
		Given a point $(A,\lambda_A,\eta_A) \in S(R)$ for a scheme $R\in\mathtt{Sch'}_{/\Fpp},$ 
		recall (Notation \eqref{n:Frob}) that
		we have the locally free 
		$\cO_R$-module $\HdR(A/R)$, 
		the Frobenius map 
		$\tV_A:\HdR(A/R)_i \to \HdR(A^{(p)}/R)_{i+1}$ 
		and the Verschiebung map 
		$\tF_A: \HdR(A^{(p)}/R)_{i+1} \to \HdR(A/R)_{i}$   for $i=0,1$
		satisfying $\ker\tF_A=\im\tV_A=\omega_{A^{(p)} /R}, \ker\tV_A=\im\tF_A.$   If no confusion arises we denote them by $\tF$ and $\tV.$ The $p$-principal polarization $\lambda_A$ induces a perfect pairing $\ang{~,~}$ on $\HdR(A/R).$ Denote by $H^\perp$ the orthogonal complement of a subbundle $H$ of $ \HdR(A/R)$ under the pairing $\ang{~,~}.$

		\begin{proposition}\label{prop:Sblowsmooth}
			$S^\#$   is smooth of dimension 2 over $\Fpp.$ 
			Moreover, 
			let $(\cA, \cP_0)$ denote the universal object on
			$S^\#$.
			Then the tangent bundle $\cT_{S^\# /\Fpp}$ of $S^\#$ fits into an exact sequence
			\begin{multline}\label{eq:Sblowtangent}
				0 \to \cH om (  {\omega}_{  \cA^{\vee}/S^\#, 1},
				\cP_0^\perp/ {\omega}_{  \cA^{\vee}/S^\#, 1})
				\to
				\cT_{S^\#/\Fpp}  
				\to 
				\cH om  
				(\cP_0^\perp /(\ker\tV)_1, 
				\HdR(     \cA/S^\# )_{ 1} /
				\cP_0^\perp )
				\to
				0
			\end{multline}
			\begin{proof}
				We show $S^\#$ is formally smooth using deformation theory. Consider a closed immersion $R \hookrightarrow \hat{R}$ in $\mathtt{Sch}_{ / \dF_{p^2}}^{\prime}$ defined by an ideal sheaf $\mathcal{I}$ with $\mathcal{I}^{2}=0$. Take a point 
				$s=(A, \lambda_A , \eta_A, \mathcal{P}_{0}) \in S^\#(R).$
				By proposition \ref{thm:STGM}
				lifting $s$ to an $\hat{R}$-point is equivalent to lifting
				\begin{itemize}
					\item
					$\omega_{A^{\vee} / R, 0}$ (resp. $ \omega_{  A^{\vee} / R,1} $) to a rank 2 (resp. rank 1) subbundle 
					$\hat{\omega}_{A^{\vee},0}$ (resp. $\hat \omega_{  A^{\vee} ,1} $)
					of $\mathrm{H}_{1}^{\text {cris }}(A / \hat{R})_{0}$ 
					( resp. $\mathrm{H}_{1}^{\text {cris }}(  A / \hat{R})_{1}$), 
					\item
					$\cP_0$ to a rank 1 subbundle $\hat \cP_0$ of $(\ker \tV)_0.$
				\end{itemize}
				subject to the following requirements
				\begin{enumerate}
					\item\label{SblowAorth}
					$\hat{\omega}_{A^{\vee},0}$ and $\hat{\omega}_{A^{\vee},1}$ are orthogonal complement of each other under $\langle ~,~\rangle_{\lambda_A}^{\text {cris }} $ (\ref{paircris});
					\item  $\hat \cP_0 \subseteq \hat\omega_{A^\vee,0};$
					
				\end{enumerate}
				Since $\langle~,~\rangle_{\lambda_A,0}^{\text {cris }}$ is a perfect pairing, $\hat{\omega}_{A^{\vee}, 0}$ is uniquely determined by $\hat{\omega}_{A^{\vee}, 1}$ by \eqref{SblowAorth}. In the meanwhile, 
				lifting $\cP_0$ is equivalent to lifting $\cP_0^\perp$
				to a rank 2 subbundle $\hat\cP_1$ of $\mathrm{H}_{1}^{\text {cris }}(  A / \hat{R})_{1}$ subject to the conditions
				\begin{enumerate}[resume]
					\item\label{liftconditionAP0perp}
					$  (\ker\tV)_0^\perp=(\ker\tV)_1 \subseteq \hat \cP_1$;
					\item  
					$\hat{\omega}_{A^{\vee}, 0}^\perp=\hat{\omega}_{A^{\vee}, 1} \subseteq  
					\hat \cP_1.
					$
				\end{enumerate}
				Therefore, it suffices to give the lifts $\hat{\omega}_{A^\vee,1}$ and $\hat\cP_1$ subject to the conditions
				\eqref{liftconditionAP0perp}.
				Thus the tangent space  $T_{S^\#/\Fpp,s}$ at $s$ fits canonically into an exact sequence
				\begin{multline}
					0 \to \cH om (  {\omega}_{  \cA^{\vee}/S^\#, 1},
					\cP_0^\perp/ {\omega}_{  \cA^{\vee}/S^\#, 1})
					\to
					\cT_{S^\#/\Fpp,s}  
					\to 
					\cH om  
					(\cP_0^\perp /(\ker\tV)_1, 
					\HdR(     \cA/S^\# )_{ 1} /
					\cP_0^\perp )
					\to
					0
				\end{multline}
				Thus, $S^\#$ is formally smooth over $\Fpp$ of dimension 2.
				
			\end{proof}
		\end{proposition} 
		\begin{remark}
			By \cite[2.3]{dSG18}, $S^\#$ is the moduli space 
			represented by the
			blow up of $S$ at the superspecial points.
			Indeed, for $R\in\mathtt{Sch'}_{/\Fpp}$ and
			$(A, \lambda_A , \eta_A, \mathcal{P}_{0}) \in S^\#(R),$ 
			if $A$ is not superspecial then $\cP_0=\ker (\tV\mid_{\omega_{A^\vee / R,0}} )$ is unique. At superspecial points, since
			$\tV\mid_{\omega_{A^\vee / R,0}}  $ vanishes,  the additional datum $\cP_0$ amounts to a choice of a subline bundle
			$\omega_{A^\vee / R,0}.$
		\end{remark}
		\begin{proposition}\label{prop:Y0Y1blowup}
			\cite[4.3.2]{dSG18}
			\begin{enumerate}
				\item \label{Y0Sblow}
				There is an isomorphism of $\Fpp$-schemes
				\[
				\pi_0^\# : Y_0 \xra{\sim} S^\#
				\]
				defined as follows:
				given a point 
				$y=(A, \lambda_A , \eta_A, \tilde{A}, \lambda_{\tilde{A}},  \eta_{\tilde{A} },\alpha)  \in Y_0(R)$
				for a scheme $R\in\mathtt{Sch'}_{/\Fpp},$ 
				define
				\[
				\pi_0^\#(y)=
				(A, \lambda_A , \eta_A, 
				( \alpha_{*,1}^{-1} \omega_{\tilde A^\vee/R,1} )^{\perp}) 
				\in S^\#(R).
				\] 
				\item
				There is a purely inseparable morphism of $\Fpp$-schemes
				\[
				\pi_1^\# : Y_1 \to S^\#
				\] defined as follows:
				given a point $y=(A, \lambda_A , \eta_A, \tilde{A}, \lambda_{\tilde{A}},  \eta_{\tilde{A} },\alpha)  \in Y_1(R)$ for a scheme $R\in\mathtt{Sch'}_{/\Fpp},$ define
				\[
				\pi_1^\#(y)=
				(A, \lambda_A , \eta_A, (\alpha_{*,1}^{-1} (\ker\tV_{\tilde A})_1  )^{\perp} )\in S^\#(R).
				\]   
			\end{enumerate}
			\begin{proof}
				\begin{enumerate}
					\item
					We check $\pi_0^\# $ is well-defined. Given a point $y=(A, \lambda_A , \eta_A, \tilde{A}, \lambda_{\tilde{A}},  \eta_{\tilde{A} },\alpha)  \in Y_0(R)$ for a scheme $R\in\mathtt{Sch'}_{/\Fpp},$ 
					we need to show
					$ (\alpha_{*,1}^{-1} \omega_{\tilde A^\vee/R,1} )^{\perp} \subseteq (\ker\tV_A)_0 \cap  \omega_{  A^\vee/R, 0}.$
					Firstly we show 
					$(\alpha_{*,1}^{-1} \omega_{\tilde A^\vee/R,1} )^{\perp}
					\subseteq  \omega_{  A^\vee/R, 0}.$ By duality it suffices to show
					$
					\omega_{  A^\vee/R, 1} \subset 
					\alpha_{*,1}^{-1} \omega_{\tilde A^\vee/R,1} ,
					$
					which follows from functoriality.
					Secondly we show 
					$(\alpha_{*,1}^{-1} \omega_{\tilde A^\vee/R,1} )^{\perp}
					\subseteq (\ker\tV_A)_0 .$
					By duality it suffices to show
					$
					(\ker\tV_A)_1 \subseteq
					\alpha_{*,1}^{-1} \omega_{\tilde A^\vee/R,1} .
					$   The condition 
					$ \im \alpha_{*,0}  = \omega_{\tilde A^\vee/R,0  }$
					implies
					$ \im \alpha_{*,0}^{(p)} = \omega_{(\tilde A^{(p)})^\vee/R,0  } .$
					The   commutative diagram \eqref{diag:commuteF}
					then implies 
					$
					\alpha_{*,1}  (\ker\tV_A)_1 =  \alpha_{*,1}   (\im\tF_A)_1 
					=\tF_{\tilde A}   \im \alpha_{*,0}^{(p)} 
					=\tF_{\tilde A}  \omega_{(\tilde A^{(p)} )^\vee/R,1}
					=0. 
					$
					Thus  $\pi_0^\# $ is well-defined.
					
					Since $S^\#$ is smooth over $\Fpp$, to show that $\pi_0^\#$ is an isomorphism, it suffices to check that for every algebraically closed field $\kappa$ containing $\Fpp$, we have
					\begin{enumerate}
						\item \label{isoY0Sblowpt}
						$\pi_0^\#$ induces a bijection on $\kappa$-points;  
						\item \label{isoY0Sblowtan}
						$\pi_0^\#$ induces an isomorphism on the tangent spaces at every $\kappa$-point.
					\end{enumerate}
					For \eqref{isoY0Sblowpt}, it suffices to construct a map $ \theta :S^{\#}(\kappa)\to Y_0(\kappa) $ inverse to $\pi_0^\#.$ Take a point
					$s = (A, \lambda_A , \eta_A, \mathcal{P}_{0}) \in S^\# (\kappa).
					$ 
					We will construct a point $y=
					( 
					A, \lambda_A, \eta_A,
					\tilde{A},\lambda_{\tilde A},\eta_{\tilde A},
					\alpha
					) \in 
					Y_0(\kappa).$
					Recall that  there is a perfect pairing $\ang{~,~}$ on $\cD(A)$ lifting that on $\HdR(A/\kappa).$ Given a $W(\kappa)$-submodule $M$ of $\cD(A)$ denote by $M^\vee$ the dual lattice
					\[
					M^\vee:=\{x\in\cD(A)\mid \ang{x,M} \in W(\kappa)\}.
					\]
					We list miscellaneous properties of $\cD(  A )$ and $\cP_0:$
					\begin{enumerate}[resume]
						\item\label{cond:chainA}
						We have two chains of $W(\kappa)$-modules
						\[
						p   \cD( A )_0 \stackrel{2} \subset
						\tF \cD(A )_1
						\stackrel{1}\subset
						\cD( A)_0, \,
						p   \cD( A)_1 \stackrel{1} \subset
						\tF  \cD(A )_0
						\stackrel{2}\subset 
						\cD( A )_1.
						\]
						Here, for an inclusion of $W(\kappa)$-modules $N\stackrel{i}\subset M$, the number $i$ above $\subset$ means  $\dim_{\kappa}(M/N)=i$. 
						\item\label{cond:Aself}
						$\cD(A)$ is self dual: $\cD(A)_0^\perp=\cD(A)_1, ~\cD(A)_1^\perp=\cD(A)_0.$
						\item\label{cond:Domega}
						The preimage of 
						$  (\ker\tV_A)_0\cap \omega_{A^\vee/S,0}$
						under the reduction map $\cD(A)_0 \to \cD(A)_0/p\cD(A)_0\cong \HdR(A/R)_0$  is
						$\tF\tD(A)_1\cap \tV\cD(A)_1.$
						\item\label{cond:P0sub}
						$\cP_0$ is a $\kappa$-vector subspace of 
						$\ker\tV\cap \omega_{A^\vee/R,0}$ of dimension 1.
						\item\label{chainpreP0}
						Denote by $\tilde \cP_0$ the preimage of $\cP_0$ under the reduction map $\cD(A)_0 \to   \HdR(A/\kappa)_0.$ Then we have   chains of $W(\kappa)$-modules
						\[
						p\tV\cD(A)_1 \stackrel{1}  \subset 
						p \cD(A)_0 \stackrel{1}  \subset 
						\tilde \cP_0   \subset 
						\tF \cD(A)_1 \cap \tV \cD(A)_1  , ~
						\tilde \cP_0 \stackrel{2}  \subset 
						\tV \cD(A)_0.
						\]
						\[
						p\tF\cD(A)_1 \stackrel{1}  \subset 
						p \cD(A)_0 \stackrel{1}  \subset 
						\tilde \cP_0   \subset 
						\tF \cD(A)_1 \cap \tV \cD(A)_1,    ~ 
						\tilde \cP_0 \stackrel{2}  \subset 
						\tF \cD(A)_0.
						\] 
					\end{enumerate}
					
					We set
					\[
					\cD_{\tilde A,0}= \tF   (\tilde \cP_0)^\vee, ~
					\cD_{\tilde A,1}= \tV^{-1} \cD(A )_0 ,\,
					\cD_{\tilde A}=\cD_{\tilde A,0}+\cD_{\tilde A,1}.
					\]
					We verify that $\cD_{\tilde A}$ is $\tF,\tV$-stable and satisfies the following chain conditions: 
					\begin{enumerate}[resume]\label{tAFVstable}
						\item
						$\tV   \cD_{\tilde A,0}  \stackrel{2}  \subset  \cD_{\tilde A,1}.$
						It suffices to check 
						$  (\tilde \cP_0)^\vee \stackrel{2}  \subset   p^{-1}\tV^{-1} \cD(A )_0$. 
						By taking duals, this  is equivalent to 
						$  p\tF \cD(A )_1 \stackrel{2}  \subset  \tilde \cP_0   ,$
						which follows from \eqref{chainpreP0}.
						\item 
						$\tF   \cD_{\tilde A,0}  \stackrel{2}  \subset \cD_{\tilde A,1}.$
						It suffices to check 
						$  (\tilde \cP_0)^\vee \stackrel{2}  \subset   p^{-1}\tF^{-1} \cD(A )_0$. 
						By taking duals,  this is equivalent to 
						$  p\tV \cD(A )_1 \stackrel{2}  \subset  \tilde \cP_0   ,$
						which follows from \eqref{chainpreP0}.
						\item\label{tAV10}
						$\tV \cD_{\tilde A,1} \stackrel{1} \subset \cD_{\tilde A,0}.$ 
						It suffices to check 
						$   \tF^{-1} \cD(A )_0 \stackrel{1}  \subset (\tilde \cP_0)^\vee $.
						By taking  duals,  this is equivalent to 
						$  \tilde \cP_0 \stackrel{1}  \subset  \tV\cD(A )_1   ,$
						which follows from \eqref{chainpreP0}.
						
						\item \label{tAV01}
						$\tF \cD_{\tilde A,1} \stackrel{1} \subset \cD_{\tilde A,0}.$ 
						It suffices to check 
						$   \tV^{-1} \cD(A )_0 \stackrel{1}  \subset (\tilde \cP_0)^\vee .$
						By taking  duals,  this is equivalent to 
						$  \tilde \cP_0 \stackrel{1}  \subset  \tF \cD(A )_0   ,$
						which follows from \eqref{chainpreP0}.
						\item \label{BlowRaynaud}
						$ \cD(A)_0 \stackrel{1}  \subset  \cD_{\tilde A,0},~ \cD(A)_1 \stackrel{1} \subseteq  \cD_{\tilde A,1}.$ Same as \eqref{tAV10} and \eqref{cond:chainA}.
					\end{enumerate}  
					Thus we have an inclusion $ \cD(A)\subseteq  \cD_{\tilde A}  $.
					By covariant Dieudonn\'e theory there exists an abelian 3-fold $\tilde A$ such that $\cD(\tilde A)=\cD_{\tilde A}$, and the inclusion  $ \cD(A)\subseteq  \cD_{\tilde A}  $  is induced by a prime-to-$p$ isogeny $ {\alpha}: A \to  \tilde{A} .$
					Define the endormorphism structure $i_{\tilde A}$ on $\tilde A$ by 
					$i_{\tilde A}(a)= {\alpha} \circ i_{  A}(a) \circ {\alpha}^{-1}$ for $a \in O_F.$
					Then $( \tilde A, i_{\tilde A} )$ is an $O_F$-abelian scheme.
					Let $\lambda_{\tilde A}$ be the unique polarization such that 
					\[
					p\lambda_{  A} = {\alpha}^\vee  \circ \lambda_{\tilde A} \circ {\alpha}.
					\]
					The pairings induced by $\lambda_{\tilde A}$ and $\lambda_B$ have the relations
					\[
					\langle x, y \rangle_{\lambda_{  A} } = p^{-1} \langle x , y \rangle_{\lambda_{\tilde A} } , \ x,y \in \cD(A).
					\]
					For a $W(\kappa)$-submodule $M$ of $\cD(A),$ we have
					\[
					M^{\vee_A } = p M^{\vee_{\tilde A}}.
					\]
					Define the level structure $\eta_{\tilde A}$ on $\tilde A$ by
					$\eta_{\tilde A}=  {\alpha}_* \circ \eta_A.$
					We verify
					\begin{enumerate}[resume]\label{tAwell}
						\item $\cD( \tilde A)$ is of signature (1,2). This is by definition.
						\item $\ker \alpha $ is a Raynaud subgroup of $A[p]$. It suffices to show 
						$ \cD(A)_0 \stackrel{1}  \subseteq \cD(\tilde A)_0$ and
						$ \cD(A)_1 \stackrel{1} \subseteq \cD(\tilde A)_1,$ which follows from
						\eqref {BlowRaynaud}.
						\item\label{BlowpoltA}
						$ \cD( \tilde A )_1 \stackrel{1}\subset \cD( \tilde A )_0^ { \perp_{ \tilde A } } $. 
						It suffices to show 
						$\tV^{-1}\cD(A)_0  \stackrel{1}\subset 
						p^{-1} \tF\tilde \cP_0 ,$ or equivalently
						$p\cD(A)_0  \stackrel{1}\subset \tilde \cP_0 ,$ which follows from  \eqref{chainpreP0}.
						\item\label{BlowpoltAd}
						$ \cD( \tilde A )_0 \stackrel{1}\subset \cD( \tilde A )_1^ { \perp_{ \tilde A } } $. This is the dual version of \eqref{BlowpoltA}.
						\item $\ker\lambda_{\tilde A}[p^\infty]$ is a $\tilde A[p]$-subgroup scheme of rank $p^2$.
						Indeed, from covariant Dieudonn\'e theory
						it is equivalent to show 
						$ \cD( \tilde A ) \stackrel{2}\subset \cD( \tilde A )^ { \perp_{ \tilde A } } .$ Thus it suffices to show 
						$ \cD( \tilde A)_0    \stackrel{1} \subset \cD( \tilde A )_1^{\perp_{\tilde A} }$  and
						$ \cD( \tilde A)_1    \stackrel{1} \subset \cD( \tilde A )_0^{\perp_{\tilde A} }
						$
						which follows from \eqref{BlowpoltA} and \eqref{BlowpoltAd}.
						\item\label{condY0}
						$ \omega_{\tilde A^\vee/R,0}  = \mathrm{im}\alpha_{*,0} , ~
						\omega_{\tilde A^\vee/R,1}  \subset  \mathrm{im}\alpha_{*,1}.$
						It suffices to check  
						$\tV\cD(\tilde A)_0\subseteq   \cD(  A)_1,~\tV \cD(\tilde A)_1\subseteq \cD(  A)_0,$ which follows from \eqref{chainpreP0}.
					\end{enumerate}  
					Finally we set 
					$\theta(s)=(A, \lambda_A  , \eta_A ,\tilde A, \lambda_{\tilde A}  , \eta_{\tilde A},\alpha).$ By \eqref{tAwell} we see $\theta(s)\in Y_0(\kappa).$
					It is easy to verify $\theta$ is the inverse of $\pi_0^\#.$
					
					For \eqref{isoY0Sblowtan},   the morphism $\pi_0^\#$
					induces the identification
					$
					\alpha_{*,1}^{-1} \omega_{\tilde A^\vee/\kappa,1} =\cP_0^\perp.
					$
					Combined with 
					Lemma \ref{lem:Y0Y1Y2induce}
					\eqref{Y0definduce2},
					we see   two exact sequences of tangent bundle 
					\eqref{eq:Y0tangent} and \eqref{eq:Sblowtangent}
					coincide. The proposition follows.
					\item
					We check $\pi_1^\# $ is well-defined. Given a point $y=(A, \lambda_A , \eta_A, \tilde{A}, \lambda_{\tilde{A}},  \eta_{\tilde{A} },\alpha)  \in Y_1(R)$ for a scheme $R\in\mathtt{Sch'}_{/\Fpp}.$
					We need to show
					$ (\alpha_{*,1}^{-1}(\ker\tV_{\tilde A})_1 )^{\perp} \subseteq (\ker\tV_A)_0 \cap  \omega_{  A^\vee/R, 0}.$
					
					Firstly we show 
					$(\alpha_{*,1}^{-1} ( \ker\tV_{\tilde A})_1 )^{\perp}
					\subseteq  \omega_{  A^\vee/R, 0}.$ By duality it suffices to show
					$
					\omega_{  A^\vee/R, 1} \subset 
					\alpha_{*,1}^{-1}( \ker\tV_{\tilde A})_1,
					$
					which follows from the condition
					$\omega_{  A^\vee/R, 1}=\ker\alpha_{*,1}.$
					Secondly we show 
					$(\alpha_{*,1}^{-1}  ( \ker\tV_{\tilde A})_1 )^{\perp}
					\subseteq (\ker\tV_A)_0.$
					By duality it suffices to show
					$
					(\ker\tV_A)_1 \subseteq
					\alpha_{*,1}^{-1}  ( \ker\tV_{\tilde A})_1 ,
					$ 
					which is again from
					the   commutative diagram \eqref{diag:commuteF}.
					Thus  $\pi_1^\# $ is well-defined.
					
					To show that $\pi_1^\#$ is a purely inseparable morphism, it suffices to check that for every algebraically closed field $\kappa$ containing $\Fpp$,  
					$\pi$ induces a bijection on $\kappa$-points.
					We construct an inverse map $  \theta $ of $\pi_1^\# .$ Take a point
					$s = (A, \lambda_A , \eta_A, \mathcal{P}_{0}) \in S^\# (\kappa).
					$ 
					
					We set
					\[
					\cD_{\tilde A,0}= \tV   (\tilde \cP_0)^\vee, ~
					\cD_{\tilde A,1}= \tF^{-1} \cD(A )_0 ,\,
					\cD_{\tilde A}=\cD_{\tilde A,0}+\cD_{\tilde A,1}.
					\]
					In an entirely similar manner we can construct a point 
					$\theta(s)=(A, \lambda_A  , \eta_A ,\tilde A, \lambda_{\tilde A}  , \eta_{\tilde A},\alpha) \in Y_1(\kappa)$.
					It is easy to verify that $\theta$ is the inverse of $\pi_1^\#.$
				\end{enumerate}
			\end{proof}
		\end{proposition}
		
		We now introduce a new moduli problem   to show $Y_2$ is a $\dP^1$-bundle over $N$.
		\begin{definition}
			Let $ {P}$ be the moduli problem   associating with every 
			$R\in \mathtt{Sch'}_{/\Fpp}$
			the set $ {P}(R)$ of equivalence classes of undecuples
			$(
			A, \lambda_A, \eta_A,
			\tilde{A},\lambda_{\tilde A},\eta_{\tilde A},
			B, \lambda_B , \eta_B,
			\alpha,\delta
			)
			$ where
			\begin{enumerate}
				\item $( A, \lambda_A, \eta_A, \tilde{A},\lambda_{\tilde A},\eta_{\tilde A}, \alpha ) \in  {  S_0(p)} (R);$
				\item $( A, \lambda_A, \eta_A,B, \lambda_B,   \eta_B,  \delta\circ \alpha ) \in  { N } (R);$
				\item $\delta: \tilde A \to B$ is a $O_F$-linear quasi-$p$-isogeny such that
				\begin{enumerate}
					\item $\ker\delta[p^\infty]\subseteq \tilde A[p];$
					\item
					$\lambda_{  \tilde A}=\delta^\vee\circ \lambda_{  B} \circ \delta;$
					\item the $ {K}^{p }$-orbit of maps $v \mapsto \delta_{*}\circ  \eta_{\tilde A}(v)$ for $v \in  {V}^{} \otimes_{\mathbb{Q}} \mathbb{A}^{\infty, p}$ coincides with $\eta_B$.
				\end{enumerate}
			\end{enumerate}
			Two undecuples 
			$(B, \lambda_B , \eta_B,
			A, \lambda_A, \eta_A,
			\tilde{A},\lambda_{\tilde A},\eta_{\tilde A},
			\alpha,\delta
			)
			$ and 
			$(B', \lambda_B' ,  \eta_B',
			A', \lambda_{ A' }, \eta_{ A' },
			\tilde{A}',\lambda_{\tilde A'},\eta_{\tilde A'},
			\alpha',\delta')$ 
			are equivalent if there are $O_{F}$-linear prime-to-$p$ quasi-isogenies 
			$ \varphi : B \to B^{\prime},$ 
			$\psi: A \to A'$ and 
			$\phi:\tilde A \to \tilde A'$ such that
			\begin{itemize}
				\item there exists $c \in \dZ_{(p)}^{\times}$such that 
				$\varphi^{\vee} \circ \lambda_{B'} \circ \varphi=c \lambda_B, 
				\psi^{\vee} \circ \lambda_{A'} \circ \psi=c \lambda_A$
				and 
				$
				\phi^{\vee} \circ \lambda_{\tilde A'} \circ \phi=c \lambda_{\tilde A};
				$
				\item the $ K^{p}$-orbit of maps 
				$v \mapsto \varphi_{*} \circ \eta_B(v)  $
				for $v \in W \otimes_{\mathbb{Q}} \mathbb{A}^{\infty, p}$ coincides with $\eta_{B'};$
				\item the $ K^{p}$-orbit of maps 
				$v \mapsto \psi_{*} \circ \eta_{A}(v)  $
				for $v \in V \otimes_{\mathbb{Q}} \mathbb{ A}^{\infty, p}$ coincides with $\eta_{A'};$
				\item the $ K^{p}$-orbit of maps 
				$v \mapsto \phi_{*} \circ \eta_{\tilde A}(v)  $
				for $v \in V \otimes_{\mathbb{Q}} \mathbb{ A}^{\infty, p}$ coincides with $\eta_{\tilde A'}.$
			\end{itemize}
		\end{definition}

		\begin{lem}\label{lem:PimplyY2}
			Take a point $s=  
			(B, \lambda_B , \eta_B,
			A, \lambda_A, \eta_A,
			\tilde{A},\lambda_{\tilde A},\eta_{\tilde A},
			\alpha,\delta
			) \in P(R)$ for a scheme 
			$R\in \mathtt{Sch'}_{/\Fpp}.$ 
			Then
			\begin{enumerate}
				\item \label{deltaiso}
				$\delta_{*,0} : \HdR(\tilde A/R)_0 \to  \HdR(B/R)_0$ is an isomorphism and  
				$\rank_{\cO_R} \ker \delta_{*,1}=1.$ 
				\item\label{PdefY2}
				$ \omega_{\tilde A^\vee/R,1}= 
				\rH_1^{\r{dR}}(\tilde A/R)_0^{\perp_{\tilde A } }.$
			\end{enumerate}
			\begin{proof}
				\begin{enumerate}
					\item
					Denote by $\gamma$ the quasi-$p$-isogeny
					$\gamma:=\delta\circ \alpha:A \to B.$ The relation 
					$
					p\lambda_{  A}=\alpha^\vee\circ \lambda_{\tilde A} \circ \alpha
					$
					and
					$
					\lambda_{\tilde A}=\delta^\vee\circ \lambda_B \circ \delta
					$
					implies 
					\[
					p \lambda_{  A}=\gamma^\vee\circ \lambda_B \circ \gamma.
					\]
					By \cite[Lemma 3.4.12(2),(3a),(3b),(4)]{LTX+22},
					we have 
					\[
					\operatorname{rank}_{\mathcal{O}_{R}}(\operatorname{ker}\alpha_{*,0})-\operatorname{rank}_{\mathcal{O}_{R}}(\operatorname{ker}\alpha_{*,1})=0,
					\]
					\[
					\operatorname{rank}_{\mathcal{O}_{R}}(\operatorname{ker}\alpha_{*,0})+\operatorname{rank}_{\mathcal{O}_{R}}(\operatorname{ker}\alpha_{*,1})=2,
					\]
					\[
					\operatorname{rank}_{\mathcal{O}_{R}}(\operatorname{ker} \gamma_{*,0})-\operatorname{rank}_{\mathcal{O}_{R}}(\operatorname{ker} \gamma_{*,1})=-1,
					\]
					\[
					\operatorname{rank}_{\mathcal{O}_{R}}(\operatorname{ker} \gamma_{*,0})+\operatorname{rank}_{\mathcal{O}_{R}}(\operatorname{ker} \gamma_{*,1})=3,
					\]
					\[
					\operatorname{rank}_{\mathcal{O}_{R}}(\operatorname{ker} \delta_{*,0})+\operatorname{rank}_{\mathcal{O}_{R}}(\operatorname{ker} \delta_{*,1})=1.
					\]
					The solution is
					\[
					\rank_{\cO_R} \ker \alpha_{*,0}=1,~
					\rank_{\cO_R} \ker \gamma_{*,0}=1,~
					\rank_{\cO_R} \ker \alpha_{*,1}=1,~
					\rank_{\cO_R} \ker \gamma_{*,1}=2.
					\]
					We claim $\rank_{\cO_R} \ker \delta_{*,0}=0$
					since otherwise $\delta_{*,1}$ is an isomorphism and 
					therefore  $\rank_{\cO_R} \ker \alpha_{*,1}=
					\rank_{\cO_R} \ker \gamma_{*,1}$ which is absurd.
					Then by comparing the ranks we conclude $\delta_{*,0}$ is an isomorphism. \eqref{deltaiso} follows.
					\item \label{lem:Pdef1Y2}
					By comparing the rank it suffices to show 
					$\ang{\omega_{\tilde A^\vee/R,1},  \rH_1^{\r{dR}}(\tilde A/R)_0  }_{\lambda_{\tilde A}}=0.$
					We claim $\omega_{\tilde A^\vee/R,1}= \ker\delta_{*,1}.$
					Indeed, by \eqref{deltaiso} it suffices to show 
					$\omega_{\tilde A^\vee/R,1}\subseteq \ker\delta_{*,1}.$
					The signature condition of $B$ implies 
					$ \omega_{B^\vee/R,1} =0$.  
					Thus $\omega_{\tilde A^\vee/R,1}\subseteq \ker\delta_{*,1}$
					and the claim follows. 
					On the other hand,  
					from $\lambda_{\tilde A}=\delta^\vee\circ \lambda_B \circ \delta$ we have 
					$ \ang{x,y}_{\lambda_{\tilde A}}=\ang{\delta_* x,\delta_* y}_{\lambda_{ B }}
					$ for $x,y\in \rH_1^{\r{dR}}(\tilde A/R).$
					Therefore
					$\ang{\ker\delta_{*,1}, 
						\rH_1^{\r{dR}}(\tilde A/R)_0  }_{\lambda_{\tilde A}} =0.
					$
					We can then conclude
					$\ang{\omega_{\tilde A^\vee/R,1},
						\rH_1^{\r{dR}}(\tilde A/R)_0  }_{\lambda_{\tilde A}} =0$
					and \eqref{PdefY2} follows.
				\end{enumerate}
			\end{proof}
		\end{lem}
		
		\begin{proposition}\label{prop:Psmooth}
			$P$ is smooth of dimension 2 over $\Fpp.$ Moreover, 
			let $(\cA, \tilde \cA, \cB, \alpha,\delta)$ denote the universal object over 
			$P$.
			Then the tangent bundle $\cT_{\cP /\Fpp}$ of $\cP$ fits into an exact sequence
			\begin{multline}\label{eq:Ptangent}
				0 \to 
				\cH om (  {\omega}_{  \tilde  \cA^{\vee}/P, 0}/
				\alpha_{*,0} {\omega}_{      \cA^{\vee}/P, 0},
				\HdR(   \tilde   \cA/P )_{ 0}
				/ {\omega}_{   \tilde \cA^{\vee}/P, 0})
				\to
				\cT_{P/\Fpp} \\
				\to 
				\cH om  
				(  {\omega}_{  \cA^{\vee}/P,0}/\ker\alpha_{*,0},
				\HdR(      \cA/P )_{ 0} /
				{\omega}_{  \cA^{\vee}/P,0})
				\to
				0.
			\end{multline}
			\begin{proof}
				The proof resembles that of Proposition 
				\ref{prop:Y0Y1Y2smooth}\eqref{prop:Y2smooth}.
				We show $P$ is formally smooth using deformation theory. Consider a closed immersion $R \hookrightarrow \hat{R}$ in $\mathtt{Sch}_{ / \dF_{p^2}}^{\prime}$ defined by an ideal sheaf $\mathcal{I}$ with $\mathcal{I}^{2}=0$. Take a point 
				$
				s=( A, \lambda_A, \eta_A,\tilde{A} ,  \lambda_{\tilde A},\eta_{\tilde A}, 
				B, \lambda_B, \eta_B ,
				\alpha,\delta ) \in P(R).$
				Denote by $\breve \delta:B \to \tilde A$  the unique quasi-$p$-isogeny such that
				$\breve \delta \circ\delta=\id_{\tilde A}$ and
				$ \delta\circ \breve \delta =\id_B.$
				By proposition \ref{thm:STGM}
				lifting $s$ to an $\hat{R}$-point is equivalent to lifting
				\begin{itemize}
					\item
					$\omega_{A^{\vee} / R, 0}$ (resp. $\omega_{\tilde A^{\vee} / R,0} $) to a rank 2 subbundle 
					$\hat{\omega}_{A^{\vee},0}$ (resp. $\omega_{\tilde A^{\vee} ,0} $)
					of $\mathrm{H}_{1}^{\text {cris }}(A / \hat{R})_{0}$ 
					( resp. $\mathrm{H}_{1}^{\text {cris }}(\tilde A / \hat{R})_{0}$), 
					\item
					$\omega_{A^{\vee} / R, 1}$ (resp. $\omega_{\tilde A^{\vee} / R,1} $)
					to a rank 1 subbundle $\hat{\omega}_{A^{\vee},1}$  (resp. $\omega_{\tilde A^{\vee} ,1} $) 
					of $\mathrm{H}_{1}^{\text {cris }}(A / \hat{R})_{1}$ 
					( resp. $\mathrm{H}_{1}^{\text {cris }}(\tilde A / \hat{R})_{1}$), 
					\item
					$\omega_{B^{\vee} / R, 0}$ (resp. $\omega_{  B^{\vee} / R,1} $)
					to a rank 3 (resp. rank 0) subbundle $\hat{\omega}_{B^{\vee},0}$  (resp. $\hat \omega_{  B^{\vee} ,1} $) 
					of $\mathrm{H}_{1}^{\text {cris }}(B / \hat{R})_{0}$ 
					( resp. $\mathrm{H}_{1}^{\text {cris }}(  B / \hat{R})_{1}$), 
				\end{itemize}
				subject to the   requirements in the proof of 
				Proposition
				\ref{prop:Y0Y1Y2smooth}\eqref{prop:Y2smooth} and
				\begin{enumerate}\label{Pliftcondtion} 
					\item\label{PtA1B1}
					$ 
					{\delta}_{*,1}  \hat{\omega}_{\tilde A^{\vee},1} \subseteq \hat{\omega}_{B^{\vee},1}.
					$  
				\end{enumerate}
				We verify \eqref{PtA1B1} holds.
				Indeed, since $\lambda_B$ is $p$-principal, it suffices to show
				$\ang{\delta_{*,1} \hat{\omega}_{\tilde A^{\vee},1} , \mathrm{H}_{1}^{\text {cris }}(B / \hat{R})_0}_{\lambda_B}=0.$
				However,
				the same argument as Lemma \ref{lem:PimplyY2}\eqref{deltaiso} shows 
				$\delta_{*,0}:
				\mathrm{H}_{1}^{\text {cris }}(\tilde A / \hat{R})_0 \to \mathrm{H}_{1}^{\text {cris }}(B/ \hat{R})_0 $ is an isomorphism. Thus 
				we have 
				$\ang{\delta_{*,1}  \hat{\omega}_{\tilde A^{\vee},1} ,   \mathrm{H}_{1}^{\text {cris }}(B / \hat{R})_0} _{\lambda_B}=
				\ang{\hat{\omega}_{\tilde A^{\vee},1} , \mathrm{H}_{1}^{\text {cris }}(\tilde A / \hat{R})_0}_{\lambda_{\tilde A} }=0$
				by  Proposition \ref{prop:Y0Y1Y2smooth}\eqref{Y2tA0tA1}
				and therefore \eqref{PtA1B1} holds.
				We conclude the requirements  are the same as those in Proposition \ref{prop:Y2smooth}. Thus  
				the tangent space  $\cT_{P /\Fpp,s}$ at $s$ fits into an exact sequence
				\begin{multline}
					0 \to 
					\cH om (  {\omega}_{  \tilde  A^{\vee}/R, 0}/
					\alpha_{*,0} {\omega}_{      A^{\vee}/R, 0},
					\HdR(   \tilde   A/R )_{ 0}
					/ {\omega}_{   \tilde A^{\vee}/R, 0})
					\to
					\cT_{P/\Fpp,s} \\
					\to 
					\cH om  
					(  {\omega}_{  A^{\vee}/R,0}/\ker\alpha_{*,0},
					\HdR(      A/R )_{ 0} /
					{\omega}_{  A^{\vee}/R,0})
					\to
					0.
				\end{multline}
				We have shown $P$ is smooth over $\Fpp$ of   dimension 2. 
			\end{proof}
		\end{proposition}
		
		\begin{lem}\label{lem:PY2}
			The natural forgetful map $\tilde\nu$
			induces an isomorphism of $\Fpp$-schemes
			\[
			\tilde\nu: P \cong Y_2.
			\]
			\begin{proof}
				Since $Y_2$ is smooth over $\Fpp$ by Proposition \ref{prop:Y0Y1Y2smooth}\eqref{prop:Y2smooth}, to show that $\tilde\nu$ is an isomorphism, it suffices to check that for every algebraically closed field $\kappa$ containing $\Fpp$, we have
				\begin{enumerate}
					\item \label{PY2iso1}
					$\tilde\nu$ induces a bijection on $\kappa$-points; and
					\item \label{PY2iso2}
					$\tilde \nu$ induces an isomorphism on the tangent spaces at every $\kappa$-point.
				\end{enumerate}
				For \eqref{PY2iso1}, we construct an inverse map $  \theta(\kappa)$ of $\tilde\nu.$ Take a point
				$y = ( 
				A, \lambda_A, \eta_A,
				\tilde{A},\lambda_{\tilde A},\eta_{\tilde A},
				\alpha
				) \in Y_2(\kappa).
				$ 
				We have the  following facts: 
				\begin{enumerate}[resume]
					\item\label{exactseq}
					We have two chains
					\[
					\cD(A)_0 \stackrel{1} \subset   \cD(\tilde A)_0, \quad \cD(A)_1  \stackrel{1} \subset  \cD(\tilde A)_1
					\]
					since $\ker \alpha$ is a Raynaud subgroup of $A[p].$
					\item \label{Y2tA0selfdual}
					$\cD(\tilde A)_0^{\perp_{\tilde A}}=p^{-1}\tV \cD(\tilde A)_0.$
					Indeed,  this is by taking the preimage of the condition   $  \omega_{\tilde A^\vee/R,1}= 
					\rH_1^{\r{dR}}(\tilde A/\kappa)_0^{\perp_{\tilde A } }
					$ under the reduction map
					$\cD(\tilde A)_0\to \cD(\tilde A)_0/p\cD(\tilde A)_0\cong \HdR(\tilde A/\kappa)_0.$
					\item\label{Y2tA0FequalV}
					$\tV\cD(\tilde A)_0=\tF\cD(\tilde A)_0.$ 
					Rewrite \eqref{Y2tA0selfdual} as 
					$\cD(\tilde A)_0^{\perp_{  A}}= \tV \cD(\tilde A)_0$
					by identifying $ \cD(\tilde A)$ as a lattice in $\cD(A)[1/p]$
					and taking account of the relation 
					$\cD(\tilde A)_0^{\perp_{  A } }=p \cD(\tilde A)_0^{\perp_{  \tilde A } }.$
					 By taking the $\lambda_A$-dual we get 
					$\cD(\tilde A)_0 =
					(\tV \cD(\tilde A)_0)^{\perp_{  A } }=
					\tF^{-1}  \cD(\tilde A)_0^{\perp_{  A } }=
					\tF^{-1}\tV\cD(\tilde A)_0.
					$
					Thus \eqref{Y2tA0FequalV} follows.
					\item\label{chainb}
					There is a chain of $W(\kappa)$-lattice in $\cD(A)_0[1/p]:$
					\label{Vo2l1}
					\[
					p \cD(A)_1  \stackrel{1} \subseteq
					\tV \cD(A)_0  \stackrel{1} \subseteq
					\tV \cD(\tilde A)_0    =  
					\cD(\tilde A)_0^{\perp_A}  \stackrel{1} \subseteq
					\cD(A)_1.
					\]

					Indeed, the first inclusion follows from \eqref{chain12} and  the second follows from \eqref{exactseq}.
				\end{enumerate}
				Now we define
				\[
				\cD_{B,0}=\cD(\tilde A)_0,\,
				\cD_{B,1}=p^{-1}\tV\cD_{B,0},\, 
				\cD_B=\cD_{B,0}+ \cD_{B,1}.
				\]
				We can easily verify $\cD_B$ is $\tF,\tV$-stable from the fact that $\cD_{B,0}$ is $\tV^{-1}\tF$-invariant.
				Moreover, we have an injection $\cD(\tilde A) \to \cD_B.$
				By covariant Dieudonn\'e theory there exists an abelian 3-fold $B$ such that $\cD(B)=\cD_B$, and the inclusion $\cD(\tilde A) \to \cD(B)$ is induced by an isogeny $\delta: \tilde A \to B.$
				Let $\lambda_B$ be the unique polarization such that 
				\[
				\lambda_{\tilde A} = \delta^\vee  \circ \lambda_B \circ \delta.
				\]
				
				We have the relation
				$$
				\langle x, y \rangle_{\lambda_{\tilde A} } = \langle x , y \rangle_{\lambda_B } , \, x,y \in \cD(\tilde A).
				$$
				Define the level structure $\eta_B$ by
				$\eta_B=  \delta_* \circ \eta_{\tilde A}.$
				We verify  
				\begin{enumerate}[resume]
					\item $\cD(B)$ is of signature type (0,3): this follows from the definition.
					\item  $\cD(B)$ is self-dual with respect to $
					\ang{~ ,~ }_{\lambda_B}$. Indeed, as above it suffices to show $\cD(B)_1=\cD(B)_0^{\perp_B} $, which is equivalent to 
					$ \tV \cD_{B,0} = \cD_{B,0}^{\perp_A},$ which follows from (\ref{Vo2l1}).
					
				\end{enumerate} 
				Finally we set $ \theta(y)=
				(B, \lambda_B , \eta_B,
				A, \lambda_A, \eta_A,
				\tilde{A},\lambda_{\tilde A},\eta_{\tilde A},
				\alpha,\delta
				).$
				
				For \eqref{PY2iso2}, take  $s\in P(\kappa)$ and thus $y=\tilde \nu(s)\in Y_2(\kappa).$
				Under the morphism $\tilde \nu$
				the exact sequences 
				\eqref{eq:Ptangent}
				and \eqref{eq:Y2tangent} coincide. Thus 
				\eqref{PY2iso2} follows and $\tilde \nu$ is an isomorphism.
			\end{proof}
		\end{lem}
		
		\begin{proposition}\label{prop:PV}
			\begin{enumerate}
				\item\label{defsV}
				Define $\sV$ by
				\[
				\sV(R):= 
				\HdR( B/ R) _{ 0}/ 
				\gamma_{*,0}\omega_{  A^\vee/ R,0} 
				\]
				where $ (
				A, \lambda_A, \eta_A, 
				B, \lambda_B , \eta_B,
				\gamma 
				) \in N(R)$ 
				for every 
				$R\in \mathtt{Sch'}_{/\Fpp}.$
				Then $\sV$ is a locally free sheaf of rank 2 over $N.$
				\item\label{isoPPV}
				The assignment 
				sending a point $(
				A, \lambda_A, \eta_A,
				\tilde{A},\lambda_{\tilde A},\eta_{\tilde A},
				B, \lambda_B , \eta_B,
				\alpha,\delta 
				) \in P(R)$  for every  $R\in \mathtt{Sch'}_{/\Fpp}$ to the subbundle
				\[
				I:= \delta_{*,0} \omega_{  \tilde A^\vee/R,0} / \delta_{*,0}   \alpha_{*,0}\omega_{  A^\vee/ R,0}   \subseteq \sV(R)
				\]
				induces an isomorphism of $\Fpp$-schemes
				\[
				\mu: P \simeq   \mathbb{P}( \sV ).
				\] 
			\end{enumerate}
			The relations of morphisms are summarized in the following diagram
			\begin{equation}\label{diag:normal1}
				\xymatrix{
				  	\dP(\sV ) \ar[d]  & Y_2  \ar[l]_-{\cong }     \ar@{^{(}->}[r]  \ar[d]^{\pi_2 } & S_0(p) 
					\ar[d]^{\pi }
					\\
					 N   \ar[r]^{\nu  } & S_{\r{ss} }  \ar@{^{(}->}[r]    & S  
				}.
			\end{equation}
			\begin{proof}
				\begin{enumerate}
					\item
					It suffices to show 
					$ \delta_{*,0}\alpha_{*,0} \omega_{  A^\vee/ R,0} $
					is locally free $\cO_R$-module of rank 1.
					Since $\delta_{*,0} $ is an isomorphism by 
					Lemma \ref{lem:PimplyY2}\eqref{deltaiso}, it suffices to show $\alpha_{*,0}\omega_{  A^\vee/ R,0}$ is locally free of rank 1, which follows from 
					Lemma \ref{lem:Y0Y1Y2induce}\eqref{Y2definduce1}.
					\item
					To show $I$ is locally free of rank 1, the argument is the same as \eqref{defsV}.
					Now we show $\mu$ is an isomorphism.
					Since $Y_2$ is smooth,  it suffices to check that for every algebraically closed field $\kappa$ containing $\Fpp$, we have
					\begin{enumerate}
						\item \label{isoPPVpt}
						$\mu$ induces a bijection on $\kappa$-points;  
						\item \label{isoPPVtan}
						$\mu$ induces an isomorphism on the tangent spaces at every $\kappa$-point.
					\end{enumerate}
					To show \eqref{isoPPVpt}, it suffices to construct an inverse   map $\theta$. 
					Take $ 
					p'=(
					A, \lambda_A, \eta_A, 
					B, \lambda_B , \eta_B,
					\gamma,
					I
					) \in \dP(\sV)(\kappa)$
					where $I$ is a locally free rank 1 $\cO_\kappa$-submodule of $\sV(\kappa).$   
					We list miscellaneous properties of $\cD(  A )$ and $\cD(  B)$:
					\begin{enumerate}[resume]
						\item\label{cond:inv}
						$\tV\cD(B )=\tF\cD(B).$ In fact, since $\cD(B)$ is of signature (0,3), \cite[Lemma 1.4]{Vol10} gives
						\[
						\cD(B)_0=\tV\cD(B)_1=\tF\cD(B)_1,
						\]
						which implies $\cD(B)_0$ and $\cD(B)_1$ are both $\tV^{-1}\tF$-invariant. 
						\item \label{duallength}
						$
						\cD(B)_0^{\perp_B } = \cD(B)_1
						$ 
						and 
						$  \cD(B)_1^{\perp_B} = \cD(B)_0.$
						This follows from the self-dual condition of $\lambda_B$.

						\item\label{cond:AB}
						We have chains of $W(\kappa)$-module
						\[
						p   \cD( B )_0 \stackrel{1} \subseteq
						\cD(   A )_0
						\stackrel{2}\subseteq
						\cD( B )_0, \,
						p   \cD( B )_1 \stackrel{1} \subseteq
						\cD( A )_1
						\stackrel{2}\subseteq
						\cD( B )_1.
						\]
						\item\label{chaintI}
						Denote by $\tilde I$ the preimage of $I$ under the composition of the reduction map $\cD(B)_0 \to \cD(B)_0/p\cD(B)_0\cong \HdR(B/R)_0$
						and the quotient map $\HdR(B/\kappa)_0 \to \sV(\kappa)$.
						Then we have a chain of $W(\kappa)$-module
						\[
						p\cD(B)_0 \stackrel{2}  \subset 
						\tilde I \stackrel{1}  \subset 
						\cD(B)_0.
						\]
					\end{enumerate}
					Now define 
					\[
					\cD_{\tilde A,0}= D(B)_0,
					\cD_{\tilde A,1 }=  \tV^{-1} \tilde I,
					\cD_{\tilde A} =   \cD_{\tilde A,0}+  \cD_{\tilde A,1}
					\] 
					
					We verify that $\cD_{\tilde A}$ is $\tF,\tV$-stable and has the following chain conditions: 
					\begin{enumerate}[resume]\label{FVtA}
						\item\label{FVtA1}
						$\tV   \cD_{\tilde A,0}  \cong 
						\tF   \cD_{\tilde A,0}.$
						This follows from  \eqref{cond:inv}.
						\item\label{FVtA2}
						$\tV   \cD_{\tilde A,0}  \stackrel{2}  \subset \cD_{\tilde A,1}.$
						The rank condition of $I$ gives $p\cD(B)_0\stackrel{2} \subset   \tilde I, $ thus we have 
						$\tF\cD(B)_0\stackrel{2} \subset  \tV^{-1} \tilde I. $
						\item\label{FVtA3}
						$\tV \cD_{\tilde A,1} \stackrel{1} \subset \cD_{\tilde A,0}.$ 
						This is by definition.
						\item $\tF \cD_{\tilde A,1} \stackrel{1} \subset \cD_{\tilde A,0}.$ This is equivalent to 
						$  \cD_{\tilde A,1} \stackrel{1} \subset \tF^{-1}\cD_{\tilde A,0}.$ The claim follows from the fact that  
						$
						\tF^{-1}\cD_{\tilde A,0}=\tV^{-1}\cD_{\tilde A,0}.
						$
					\end{enumerate}
					We also have an inclusion ${\delta}:D_{\tilde A} \subset \cD(B)$ by definition.
					By covariant Dieudonn\'e theory there exists an abelian 3-fold $A$ such that $\cD(\tilde A)=\cD_{\tilde A}$, and the inclusion $ \cD(\tilde A) \to \cD(B)$  is induced by a prime-to-$p$ isogeny $ {\delta}: \tilde{A} \to B.$
					Define the endormorphism structure $i_{\tilde A}$ on $\tilde A$ by 
					$i_{\tilde A}(a)= {\delta}^{-1} \circ i_{  B}(a) \circ {\delta}$ for $a \in O_F.$
					Then $( \tilde A, i_{\tilde A} )$ is an $O_F$-abelian scheme.
					Let $\lambda_{\tilde A}$ be the unique polarization such that 
					\[
					\lambda_{\tilde A} = {\delta}^\vee  \circ \lambda_B \circ {\delta}.
					\]
					The pairings induced by $\lambda_{\tilde A}$ and $\lambda_B$ have the relation
					\[
					\langle x, y \rangle_{\lambda_{\tilde A} } =  \langle x , y \rangle_{\lambda_B} , \ x,y \in \cD(A).
					\]
					Define the level structure $\eta_{\tilde A}$ on $\tilde A$ by
					$\eta_{\tilde A}=  {\delta}_*^{-1} \circ \eta_B.$
					We verify
					\begin{enumerate}[resume]
						\item $\cD( \tilde A)$ is of signature (1,2). This follows from \eqref{FVtA2} and \eqref{FVtA3}. 
						\item\label{poltA}
						$ \cD( \tilde A )_1 \stackrel{1}\subset \cD( \tilde A )_0^ { \perp_{ \tilde A } } $. 
						Consider  
						$\cD( \tilde A )_0^ { \perp_{ \tilde A } } =
						\cD( B)_0^ { \perp_{ B } }=p^{-1}\tV\cD(B)_0.
						$
						The claim follows from the definition and \eqref{chaintI}.
						\item\label{poltAd}
						$ \cD( \tilde A )_0 \stackrel{1}\subset \cD( \tilde A )_1^ { \perp_{ \tilde A } } $. This is the dual version of \eqref{poltA}.
						\item $\ker\lambda_{\tilde A}[p^\infty]$ is a $\tilde A[p]$-subgroup scheme of rank $p^2$.
						Indeed, from covariant Dieudonn\'e theory
						it is equivalent to show 
						$ \cD( \tilde A ) \stackrel{2}\subset \cD( \tilde A )^ { \perp_{ \tilde A } } .$ Thus it suffices to show 
						$ \cD( \tilde A)_0    \stackrel{1} \subset \cD( \tilde A )_1^{\perp_{\tilde A} }$  and
						$ \cD( \tilde A)_1    \stackrel{1} \subset \cD( \tilde A )_0^{\perp_{\tilde A} }
						$
						which follows from \eqref{poltA} and \eqref{poltAd}.
					\end{enumerate}  
					Now we prove \eqref{isoPPVtan}. 
					Indeed, a deformation argument shows  that
					the tangent space  $\cT_{\dP(\sV) /\Fpp,p'}$ at $p'$ fits into an exact sequence
					\begin{multline}
						0 \to 
						\cH om ( I,
						\sV(R)/I)
						\to
						\cT_{\mathbb{P}(\mathscr{V})  /\Fpp,p} 
						\to 
						\cH om  
						(  {\omega}_{  A^{\vee}/R,0}/\ker\gamma_{*,0},
						\HdR(      A/R )_{ 0} /
						{\omega}_{  A^{\vee}/R,0})
						\to
						0
					\end{multline}
					which coincides with \eqref{eq:Ptangent} under $\mu.$
					Thus  \eqref{isoPPVtan} follows.
				\end{enumerate}
			\end{proof}
			
		\end{proposition}

		\subsection{Intersection of irreducible components of $S_0(p)$}
		Define $Y_{i,j}:=Y_i \times_{S_0(p)}  Y_j$ and $Y_{i,j,k}:=Y_i \times_{S_0(p)} Y_j \times_{S_0(p)} Y_k.$ The intersection of irreducible components are parametrized by some discrete Shimura varieties:

		\begin{proposition}\label{prop:Y01Y02Y12}
			\begin{enumerate}
				\item\label{Y01}
				Denote by $\pi_{0,1}$ the restriction of the morphism $\pi$
				on $Y_{0,1}.$ Then  $\pi_{0,1}$ factors through $S_{\spe}.$
				Moreover,
				denote by $(\cA, \lambda_\cA, \eta_\cA)$ the universal object on $S_{\spe}.$ Let 
				$\dP:=\dP (\omega_{\cA^\vee,0} )$ be  the projective bundle associated with $\omega_{\cA^\vee,0}.$
				Then the assignment 
				sending a point $(
				A, \lambda_A, \eta_A,
				\tilde{A},\lambda_{\tilde A},\eta_{\tilde A},
				\alpha 
				) \in Y_{0,1}(R)$  for every  $R\in \mathtt{Sch'}_{/\Fpp}$ to the subbundle
				\[
				I:= (\alpha_{*,1}^{-1} \omega_{  \tilde A^\vee/R,1}  )^\perp  \subseteq \omega_{A^\vee/R,0}
				\]
				induces an isomorphism of $\Fpp$-schemes
				\[
				\varphi_{0,1}:  Y_{0,1} \cong \dP.
				\]
				The morphism $\varphi_{0,1}$
				is equivariant under the prime-to-$p$ Hecke correspondence. That is, given $g \in K^p\backslash G(\dA^{\infty,p}) / K'^p$ such that $g^{-1}K^p g \subset K'^p,$ we have a commutative diagram
				\[
				\xymatrix{
				 Y_{0,1} \bKp \ar[rr]^-{\varphi_{0,1}\bKp }\ar[d]_g
					&&  Y_{0,1}(K'^p)   \ar[d]^g \\
					\dP  \bKp   \ar[rr]^-{\varphi_{0,1}(K'^p) }
					&&  \dP  (K'^p) 
				}.
				\]
				To summarize, we have the commutative diagram
				\begin{equation}\label{diag:normal1}
					\xymatrix{
				    \dP \ar[dr]    & Y_{0,1} \ar[l]_{ \cong }  
				       \ar@{^{(}->}[r]        \ar[d]^{\pi_{0,1} } & S_0(p)  \ar[d]^{\pi }
						\\
					  &  S_{\spe} \ar@{^{(}->}[r]    & S   }.
				\end{equation}
				\item \label{Y02}
				The restriction of  the morphism $\tilde \pi:=\pi\circ\tilde\nu^{-1}$ on $Y_{0,2}$ 
				in the diagram \eqref{diag:normal1}
				is
				an isomorphism of $\Fpp$-schemes which is equivariant under the prime-to-$p$ Hecke correspondence. 
				\[
				\tilde \pi_{0,2}:=\tilde\pi\mid_{Y_{0,2}} : Y_{0,2} \cong N.
				\]
				\[\label{normal1}
				\xymatrix{
					Y_{0,2} \ar@{^{(}->}[r]  &  Y_2  \ar[r]^{\tilde \nu^{-1} } 
					& P \ar[r]^\pi  & N
				}.
				\]
				
				\item\label{Y12}
				The morphism $\tilde \pi$ induces
				a finite flat purely inseparable map
				\[
				\tilde \pi_{1,2}: Y_{1,2} \to N.
				\]
				which is equivariant under the prime-to-$p$ Hecke correspondence.
			\end{enumerate}
		\end{proposition}
		\begin{proof}
			\begin{enumerate}
				\item
				We show $\pi_{0,1}$ factors through $S_{\spe}.$
				Take    $y=(A, \lambda_A , \eta_A, \tilde{A}, \lambda_{\tilde{A}},  \eta_{\tilde{A} },\alpha)  \in Y_{0,1}(R)$
				for a scheme $R\in\mathtt{Sch'}_{/\Fpp},$ 
				we need to show $\tV\omega_{A^\vee/R,1}=0.$
				By definition we have 
				$   \omega_{A^\vee/R,1}=  \ker\alpha_{*,1};$
				By Lemma \ref{lem:Y0Y1Y2induce}\eqref{Y0definduce2}
				we have  $\tV \ker\alpha_{*,1}=0.$
				Thus $(A, \lambda_A , \eta_A)\in S_{\spe}(R).$
				
				It is easy to see $\varphi_{0,1}$ is well-defined. 
				Now we show it is an isomorphism.
				A deformation argument shows $Y_{0,1}$ is smooth with tangent bundle
				\begin{equation}\label{eq:Y01tan}
					\cT_{Y_{0,1} }\cong \cH om  
					(\alpha_{*,1}^{-1}  {\omega}_{\tilde \cA^{\vee},1}/\ker\alpha_{*,1}, 
					\HdR(     \cA )_{ 1} /
					\alpha_{*,1}^{-1}  {\omega}_{\tilde \cA^{\vee},1}).
				\end{equation}
				Thus it suffices to check
				that for every algebraically closed field $\kappa $ containing $\Fpp,$
				\begin{enumerate}
					\item\label{isoPY01} $  \varphi_{0,1}$ induces a bijection
					on $\kappa$-points;
					\item\label{isoPY01t}
					$  \varphi_{0,1}$ induces an isomorphism on the tangent spaces at every $\kappa$-point.
				\end{enumerate}
				To show \eqref{isoPY01t}, it suffices to construct an inverse map $\theta.$ 
				Take $p=
				( 
				A, \lambda_A, \eta_A, 
				I
				) \in \dP(\kappa)$
				where $I$ is a locally free rank 1 sub $\kappa$-module of  
				$\omega_{A^\vee/R,0}.$
				We list miscellaneous properties of $\cD(  A ):$
				\begin{enumerate}[resume]
					\item \label{Y01invtau}
					$\tF\cD(A)_0=\tV\cD(A)_0.$ This is by
					$\tV\omega_{A^\vee/R,1}=0.$
					\item\label{Y01chaintI}
					Denote by $\widetilde {I^\perp}$ the preimage of $I^\perp$ under the reduction map $\cD(A)_1 \to \cD(A)_1/p\cD(A)_1\cong \HdR(A/R)_1.$ Then 
					the condition $\omega_{A^\vee/\kappa,1} \subset I^\perp$ lifts as a chain of $W(\kappa)$-module
					\[
					\tV\cD(A)_0 \stackrel{1}  \subset 
					\widetilde {I^\perp} \stackrel{1}  \subset 
					\cD(A)_1 \stackrel{1}  \subset \tF^{-1}\cD(A)_0 .
					\]
				\end{enumerate}
				Now define 
				\[
				\cD_{\tilde A,0}= \tV^{-1} \widetilde {I^\perp},
				\cD_{\tilde A,1 }= \tV^{-1}\cD(A)_0 ,
				\cD_{\tilde A} =   \cD_{\tilde A,0}+  \cD_{\tilde A,1}
				\] 
				
				We verify that $\cD_{\tilde A}$ is $\tF,\tV$-stable and has the following chain conditions: 
				\begin{enumerate}[resume]\label{Y01chain}
					\item 
					$\tV   \cD_{\tilde A,0}  \stackrel{2}  \subset \cD_{\tilde A,1}$
					and
					$\tF  \cD_{\tilde A,0}  \stackrel{2}  \subset \cD_{\tilde A,1}.$
					By \eqref{Y01invtau}
					it suffices to show that 
					$  \widetilde {I^\perp}\stackrel{2}  \subset \tF^{-1}  \cD(A)_0,$
					which is by   \eqref{Y01chaintI}.
					\item $\tV \cD_{\tilde A,1} \stackrel{1} \subset \cD_{\tilde A,0}$ and
					$\tF \cD_{\tilde A,1} \stackrel{1} \subset \cD_{\tilde A,0}.$ 
					By \eqref{Y01invtau}
					it suffices to show 
					$\tV\cD(A)_0  \stackrel{1}  \subset  \widetilde {I^\perp}   ,$
					which is by   \eqref{Y01chaintI}.
				\end{enumerate}
				We also have an inclusion ${\alpha_*}:\cD({  A}) \subset \cD_{\tilde A}$ by definition.
				By covariant Dieudonn\'e theory there exists an abelian 3-fold $\tilde A$ such that $\cD(\tilde A)=\cD_{\tilde A}$, and $\alpha_*$  is induced by a prime-to-$p$ isogeny $ {\alpha}: A \to \tilde A.$
				Define the endormorphism structure $i_{\tilde A}$ on $\tilde A$ by 
				$i_{  A}(a)= {\alpha}^{-1} \circ i_{  \tilde A}(a) \circ {\alpha}$ for $a \in O_F.$
				Then $( \tilde A, i_{\tilde A} )$ is an $O_F$-abelian scheme.
				Let $\lambda_{\tilde A}$ be the unique polarization such that 
				\[
				p\lambda_{  A} = {\alpha}^\vee  \circ \lambda_{\tilde A} \circ {\alpha}.
				\]
				The pairings induced by $\lambda_{  A}$ and $\lambda_{\tilde A}$ are related by 
				\[
				\langle x, y \rangle_{\lambda_{  A} } =p^{-1}  \langle \alpha_* x , \alpha_* y \rangle_{\lambda_{\tilde A}} , \ x,y \in \cD(A).
				\]
				Define the level structure $\eta_{\tilde A}$ on $\tilde A$ by
				$\eta_{\tilde A}=  {\alpha}_*  \circ \eta_A.$
				We verify
				\begin{enumerate}[resume]
					\item $\cD( \tilde A)$ is of signature (1,2). This follows from \eqref{Y01chain}.
					\item\label{Y01poltA}
					$ \cD( \tilde A )_1 \stackrel{1}\subset \cD( \tilde A )_0^ { \vee_{ \tilde A } } $. 
					Consider  
					$\cD( \tilde A )_0^ { \vee_{ \tilde A } } =
					( \tV^{-1} \widetilde{I^ { \perp  } })^{\vee_{\tilde A}} =
					p^{-1}\tF (\widetilde{ I^ { \perp  }})^{\vee_{  A}}.
					$
					The claim follows from the definition and \eqref{Y01chaintI}.
					\item\label{Y01poltAd}
					$ \cD( \tilde A )_0 \stackrel{1}\subset \cD( \tilde A )_1^ { \vee_{ \tilde A } } $. This is the dual version of \eqref{Y01poltA}.
					\item $\ker\lambda_{\tilde A}[p^\infty]$ is a $\tilde A[p]$-subgroup scheme of rank $p^2$.
					Indeed, from covariant Dieudonn\'e theory
					it is equivalent to show 
					$ \cD( \tilde A ) \stackrel{2}\subset \cD( \tilde A )^ { \perp_{ \tilde A } } .$ Thus it suffices to show 
					$ \cD( \tilde A)_0    \stackrel{1} \subset \cD( \tilde A )_1^{\perp_{\tilde A} }$  and
					$ \cD( \tilde A)_1    \stackrel{1} \subset \cD( \tilde A )_0^{\perp_{\tilde A} }
					$
					which follows from \eqref{Y01poltA} and \eqref{Y01poltAd}.
					\item $\omega_{\tilde A^\vee/\kappa,0}=\im\alpha_{*,0}$
					and $\omega_{  A^\vee/\kappa,1}=\ker\alpha_{*,1}.$
					These are from the definition of $\cD(\tilde A)_1$ and
					\eqref{Y01invtau}.
				\end{enumerate}  
				Finally we set
				$\theta(p)=(
				A , \lambda_A, \eta_A, {\tilde A},  \lambda_{ \tilde A } , \eta_{ \tilde A } , \alpha
				).$
				The equivariance under prime-to-$p$ Hecke correspondence is clear.
				
				To show \eqref{isoPY01t}, denote by $ \mathcal{I} \subseteq \omega_{\cA^\vee,0}
				$ the universal subbundle (of rank 1). Then we have an isomorphism
				\begin{equation}\label{eq:dPtan}
					\mathcal{T}_{\mathbb{P} / S_{\spe} } 
					\simeq {\cH om}_{\mathcal{O}_{\mathbb{P} } }
					(\mathcal{I}^\perp/\omega_{\cA^\vee,1}, 
					\HdR(\cA)_1
					/ \mathcal{I}^\perp) .
				\end{equation}
				Under the morphism $\varphi_{0,1} $
				we have 
				\[
				I^\perp=\alpha_{*,1}^{-1}\omega_{A/\kappa,1},~\ker\alpha_{*,1}=\omega_{A/\kappa,1}.
				\]
				Thus the expression of tangent space 
				\eqref{eq:Y01tan} and \eqref{eq:dPtan} coincide.
				Thus $\varphi_{0,1}$ is an isomorphism.
				
				\item\label{mintss}
				Since $N$ is smooth over $\Fpp$ by Proposition \ref{prop:Y0Y1Y2smooth}\eqref{prop:Y2smooth}, to show that $ \varphi_{0,2}$ is an isomorphism, it suffices to check that for every algebraically closed field $\kappa$ containing $\Fpp$, we have
				\begin{enumerate}
					\item \label{iso1}
					$ \varphi_{0,2}$ induces a bijection on $\kappa$-points; and
					\item \label{iso2}
					$ \varphi_{0,2}$ induces an isomorphism on the tangent spaces at every $\kappa$-point.
				\end{enumerate}
				For \eqref{iso1}, we construct an inverse map $  \theta$ of $\varphi_{0,2}.$ Take a point
				$n = ( 
				A, \lambda_A, \eta_A, B, \lambda_B , \eta_B,
				\gamma
				) \in N(\kappa).
				$ 
				We define
				\[
				\cD_{\tilde A,0}=\cD(B)_0, \,
				\cD_{\tilde A,1}= \tV^{-1} \cD(A )_0 ,\,
				\cD_{\tilde A}=\cD_{\tilde A,0} \oplus \cD_{\tilde A,1}.
				\]
				We can easily verify $\cD_B$ is $\tF,\tV$-stable from the fact that $\cD_{B,0}$ is $\tV^{-1}\tF$-invariant. 
				
				We also have an inclusion ${\alpha_*}:\cD({  A}) \subset \cD_{\tilde A}$ by definition.
				By covariant Dieudonn\'e theory there exists an abelian 3-fold $\tilde A$ such that $\cD(\tilde A)=\cD_{\tilde A}$, and $\alpha_*$  is induced by a prime-to-$p$ isogeny $ {\alpha}: A \to \tilde A.$
				Define the endormorphism structure $i_{\tilde A},$ 
				polarization $\lambda_{\tilde A}$
				and prime-to-$p$ level structure $\eta_{\tilde A}$ in a similar way.
				We verify
				\begin{enumerate}
					\item $\cD( \tilde A)$ is of signature (1,2). This is by definition.
					\item\label{Y01poltA}
					$ \cD( \tilde A )_1 \stackrel{1}\subset \cD( \tilde A )_0^ { \perp_{ \tilde A } } $. 
					Consider  
					$\cD( \tilde A )_0^ { \perp_{ \tilde A } } =
					p \tilde I^ { \perp_{ \tilde A } }= \tilde I^ { \perp_{   A } }.
					$
					The claim follows from the definition and \eqref{Y01chaintI}.
					\item\label{Y01poltAd}
					$ \cD( \tilde A )_0 \stackrel{1}\subset \cD( \tilde A )_1^ { \perp_{ \tilde A } } $. This is the dual version of \eqref{Y01poltA}.
					\item $\ker\lambda_{\tilde A}[p^\infty]$ is a $\tilde A[p]$-subgroup scheme of rank $p^2$.
					Indeed, from covariant Dieudonn\'e theory
					it is equivalent to show 
					$ \cD( \tilde A ) \stackrel{2}\subset \cD( \tilde A )^ { \perp_{ \tilde A } } .$ Thus it suffices to show 
					$ \cD( \tilde A)_0    \stackrel{1} \subset \cD( \tilde A )_1^{\perp_{\tilde A} }$  and
					$ \cD( \tilde A)_1    \stackrel{1} \subset \cD( \tilde A )_0^{\perp_{\tilde A} }
					$
					which follows from \eqref{Y01poltA} and \eqref{Y01poltAd}.
					\item $\omega_{\tilde A^\vee/\kappa,0}=\im\alpha_{*,0}$
					and $\omega_{  A^\vee/\kappa,1}=\ker\alpha_{*,1}.$
					These are from the definition of $\cD(\tilde A)_1$ and
					\eqref{Y01invtau}.
				\end{enumerate}  
				Finally we set
				$\theta(n)=(
				A , \lambda_A, \eta_A, {\tilde A},  \lambda_{ \tilde A } , \eta_{ \tilde A } , \alpha
				).$
				The equivariance under prime-to-$p$ Hecke correspondence is clear.
				
				For \eqref{iso2}, take  $p\in P(\kappa)$ and thus $y=\tilde \nu(p)\in Y_2(\kappa).$
				By the proof of Proposition \ref{prop:Y2smooth} and 
				Proposition \ref{prop:Psmooth}, 
				the canonical morphism of tangent space  
				\[
				\cT_{Y_2,y} \to \tilde\nu_*  \cT_{P,p}
				\]
				is an isomorphism.
				\item 
				To show that $ \varphi_{1,2}$ is a purely inseparable morphism, we need check that for every algebraically closed field $\kappa$ containing $\Fpp$, 
				$ \varphi_{1,2}$ induces a bijection on $\kappa$-points.
				We construct an inverse map $  \theta$ of $\varphi_{0,2}.$ Take a point
				$n = ( 
				A, \lambda_A, \eta_A, B, \lambda_B , \eta_B,
				\gamma
				) \in N(\kappa).
				$ 
				We define
				\[
				\cD_{\tilde A,0}=\cD(B)_0, \,
				\cD_{\tilde A,1}= p^{-1}\tV \cD(A )_0 ,\,
				\cD_{\tilde A}=\cD_{\tilde A,0}+\cD_{\tilde A,1}.
				\]
				We can easily verify $\cD_B$ is $\tF,\tV$-stable from the fact that $\cD_{B,0}$ is $\tV^{-1}\tF$-invariant. 
				In a entirely similar way we can construct a point 
				$\theta(n)=(
				A , \lambda_A, \eta_A, {\tilde A},  \lambda_{ \tilde A } , \eta_{ \tilde A } , \alpha
				).$
			\end{enumerate}
		\end{proof}
		
		\begin{definition}
			Let $\tilde M$ be the moduli problem   associating with every $R\in \mathtt{Sch'}_{/\Fpp}$
			the set $ \tilde M(R)$ of equivalence classes of tuples
			$(\tilde B, \lambda_{\tilde B}, \eta_{ \tilde  B},
			A, \lambda_A, \eta_A, 
			\tilde A, \lambda_{\tilde A}, \eta_{\tilde A} ,
			B, \lambda_{  B} , \eta_{  B},
			\delta',\alpha,\delta
			)
			$ where
			\begin{enumerate}
				\item $(\tilde B, \lambda_{\tilde B}  ,  \eta_{\tilde B},A, \lambda_A, \eta_A, \delta' ) \in  { M } (R);$
				\item $( A, \lambda_A, \eta_A, \tilde A, \lambda_{\tilde A}, \eta_{\tilde A}, \alpha)\in Y_{0,2}(R);$
				\item
				$( A, \lambda_A, \eta_A, \tilde A, \lambda_{\tilde A}, \eta_{\tilde A}, 
				B, \lambda_{  B} , \eta_{  B},
				\alpha,\delta) \in P(R);$
				\item
				$( B, \lambda_B, \eta_B,\tilde  B, \lambda_{ \tilde B} , \eta_{ \tilde B},
				\delta'\circ \alpha\circ \delta) \in T_0(p)(R).$
			\end{enumerate}
			The equivalence relations are defined in a similar way.
		\end{definition}
		There is a natural correspondence
		\[
		\xymatrix
		{
			&\tilde M\ar[ld]_{\tilde \rho'} \ar[rd]^{\tilde \rho}&\\
			{ T_0(p) }&& Y_{0,2}
		}
		\]
		\begin{lem}
			The morphism $\tilde \rho$ factors through $Y_{0,1,2}.$ Moreover,
			$\tilde M$ is smooth of dimension 0.
			\begin{proof}
				Take a point $(\tilde B, \lambda_{\tilde B}, \eta_{ \tilde  B},
				A, \lambda_A, \eta_A, 
				\tilde A, \lambda_{\tilde A}, \eta_{\tilde A} ,
				B, \lambda_{  B} , \eta_{  B},
				\delta',\alpha,\delta
				) \in \tilde M(R)$
				for $R\in \mathtt{Sch'}_{/\Fpp}.$
				By Lemma \ref{lem:Y0Y1Y2induce}\eqref{Y0definduce2}
				we have  $(\ker\tV)_1= \ker\alpha_{*,1}.$
				By Remark \ref{rem:sspsmooth}
				we have $(\ker\tV)_1= \omega_{A^\vee/R,1}.$
				Thus $\omega_{A^\vee/R,1}= \ker\alpha_{*,1}$
				and  $\tilde \rho$ factors through $Y_{0,1,2}.$
				It is easy to see  $\tilde B,A,\tilde A, B$ have
				trivial deformation. 
				Thus $\tilde M$ is smooth of dimension 0.
			\end{proof}
			
		\end{lem}

		\begin{lem}\label{lem:Y012}
			\begin{enumerate}
				\item\label{tMY012}
				The morphism $\tilde \rho$ induces an isomorphism of $\Fpp$-schemes
				\[
				\tilde \rho: \tilde M\cong   Y_{0,1,2}
				\]
				which is equivariant under the prime-to-$p$ Hecke correspondence.  
				\item\label{tMT0p}
				The morphism $\tilde \rho'$ is an isomorphism of $\Fpp$-schemes
				\[
				\tilde \rho':\tilde M\cong T_0(p)
				\]
				which is equivariant under the prime-to-$p$ Hecke correspondence.
			\end{enumerate}
			\begin{proof}
				\begin{enumerate}
					\item 
					Since $\tilde M$ and $Y_{0,1,2}$ are smooth of dimension 0,  to show that $  \tilde \rho$ is an isomorphism, it suffices to check that for every algebraically closed field $\kappa$ containing $\Fpp$,  
					$\tilde \rho$ induces a bijection on $\kappa$-points. 
					Take a point
					$y = ( 
					A, \lambda_A, \eta_A,
					\tilde{A},\lambda_{\tilde A},\eta_{\tilde A},
					\alpha
					) \in Y_{0,1,2}(\kappa).
					$ 
					We set 
					$\tilde \theta(y)= (\tilde B, \lambda_{\tilde B}, \eta_{ \tilde  B},
					A, \lambda_A, \eta_A, 
					\tilde A, \lambda_{\tilde A}, \eta_{\tilde A} ,
					B, \lambda_{  B} , \eta_{  B},
					\delta',\alpha,\delta
					),$
					where $\tilde B$ with $\delta'$ are constructed in
					Lemma \ref{lem:tTMSp} 
					and 
					$  B$ with $\delta$ are constructed in
					Lemma \ref{lem:PY2}.
					It is easy to verify $\tilde \theta(y) \in \tilde M(R)$
					and $\tilde \theta$  is the inverse of   $\tilde \rho $. The equivariance under prime-to-$p$ Hecke correspondence is clear.
					\item
					Since $\tilde M$ and $  T_0(p) $ are smooth of dimension 0,  to show that $ \tilde  \rho'$ is an isomorphism, it suffices to check that for every algebraically closed field $\kappa$ containing $\Fpp$,  
					$\tilde \rho'$ induces a bijection on $\kappa$-points. 
					Take a point  
					$t=( \tilde B ,  \lambda_{\tilde B}   ,  {\eta}_{\tilde B },
					B ,  \lambda_{  B}   ,  {\eta}_{  B }, \beta
					) \in \tilde{T}( \kappa ).$
					
					We list   properties of $\cD(  B )$ and $\cD(\tilde B)$:
					\begin{enumerate}
						\item\label{BtBF=V}  $\cD(B)$ and $\cD(\tilde B)$ is $\tV^{-1}\tF$-invariant. In fact, since $\cD(B)$ is of signature (0,3), \cite[Lemma 1.4]{Vol10} gives
						\[
						\cD(B)_0=\tV\cD(B)_1=\tF\cD(B)_1,
						\]
						which implies $\cD(B)_0$ and $\cD(B)_1$ are both $\tV^{-1}\tF$-invariant. The argument is identical for $\cD(\tilde B).$
						\item \label{duallength}
						$
						\cD(B)_0^{\perp_B } = \cD(B)_1
						$ 
						and 
						$  \cD(B)_1^{\perp_B} = \cD(B)_0.$
						This follows from the self-dual condition of $\lambda_B$.
						
						\item\label{chain03} We have a chain of $W(\kappa)$-lattice
						$$
						\cD({\tilde B})_1 \stackrel{1}\subset \tV^{-1}\cD({\tilde B})_1^{\perp_{ \tilde B} } \stackrel{2}\subset \frac{1}{p}\cD({\tilde B})_1.
						$$
						Indeed, $\ker \lambda_{\tilde B} \subset {\tilde B}[p]$
						gives $\cD({\tilde B})_0^{\perp_{\tilde B}} \subset (1/p)\cD({\tilde B})_1.$
						The claim comes from \eqref{duallength} and the fact that $\cD({\tilde B})_1^{\perp_{\tilde B}}=( \tV^{-1} \cD({\tilde B})_0)^{\perp_{\tilde B}}= 
						\tF (\cD({\tilde B})_0)^{\perp_{\tilde B}}$.
						\item\label{condition:isogeny}
						We have a relation
						\[
						p   \cD( B )_0 \stackrel{1} \subset
						\cD( \tilde B )_0
						\stackrel{2}\subset
						\cD( B )_0, \,
						p   \cD( B )_1 \stackrel{1} \subset
						\cD( \tilde B )_1
						\stackrel{2}\subset 
						\cD( B )_1.
						\]
						Indeed, we have $ p \cD( B )  \subset \cD( \tilde B )
						$ since $\ker \beta\in \tilde B[p]$ and there is an exact sequence
						\[
						0 \to \cD(\tilde B) \to \cD(B) \to \cD(\ker \beta) \to 0
						\]
						by covariant Dieudonn\'e theory.
					\end{enumerate}
					
					We set
					\[
					\cD_{\tilde A,0}=\tV\cD( B)_1, \,
					\cD_{\tilde A,1}= \tV^{-1} \cD( \tilde B )_1^{\perp_{ \tilde B} } ,\,
					\cD_{\tilde A}=\cD_{\tilde A,0}+\cD_{\tilde A,1}.
					\]
					We verify that $\cD_{\tilde A}$ is $\tF,\tV$-stable. Indeed,
					since $\cD(B)$ and $\cD(\tilde B)$ are $\tV^{-1}\tF$-invariant, it suffices to verify the condition for $\tV$: we have  
					$\tV \cD_{ \tilde A }
					=\tV^2\cD(B)_1+ \cD(\tilde B)_1^{\perp_{\tilde B} }.$ 
					Then it suffices to show 
					$p\cD(B)_0 \subset p \cD(\tilde B)_1^{\perp_ B }$ and 
					$p\cD(\tilde B)_1^{\perp_B} \subset D(B)_0$
					since $\tV^2=\tF\tV=p$. Then it suffices to show 
					$\cD(\tilde B)_1 \subset D(B)_0^{\perp_B}$ and $p\cD(B)_0^{\perp_B}\subset \cD(\tilde B)_1,$ which are from \eqref{condition:isogeny}.
					By covariant Dieudonn\'e theory there exists an abelian 3-fold $A$ such that $\cD(\tilde A)=\cD_{\tilde A}$, and the inclusion $ \cD(\tilde A) \to \cD(B)$  is induced by a prime-to-$p$ isogeny $\delta: \tilde{A} \to B.$
					Define the endormorphism structure $i_{\tilde A}$ on $\tilde A$ by 
					$i_{\tilde A}(a)= \delta^{-1} \circ i_{  B}(a) \circ \delta$ for $a \in O_F.$
					Then $( A, i_A)$ is an $O_F$-abelian scheme.
					Let $\lambda_{\tilde A}$ be the unique polarization such that 
					\[
					\lambda_{\tilde A} = \delta^\vee  \circ \lambda_B \circ \delta.
					\]
					The pairings induced by $\lambda_{\tilde A}$ and $\lambda_B$ have the relation
					\[
					\langle x, y \rangle_{\lambda_{\tilde A} } =  \langle x , y \rangle_{\lambda_B} , \ x,y \in \cD(A).
					\]
					Define the level structure $\eta_{\tilde A}$ on $\tilde A$ by
					$\eta_{\tilde A}=  \delta_*^{-1} \circ \eta_B.$
					We verify
					\begin{enumerate}[resume] 
						\item $\cD( \tilde A)$ is of signature (1,2):  
						calculate the Lie algebra  
						\[
						\frac{\cD(\tilde A)}{\tV\cD(\tilde A)}=
						\frac{ \tV\cD( B)_1+
							\tV^{-1} \cD( \tilde B )_1^{\perp_{ \tilde B} } }
						{  \cD( \tilde B)_1^{\perp_{\tilde B}}  + p\cD(  B)_1  }.
						\]
						The argument is the same as that in verifying $\cD_{\tilde A}$ is $\tF,\tV$-stable.
						\item $\ker\lambda_{\tilde A}[p^\infty]$ is a $\tilde A[p]$-subgroup scheme of rank $p^2$.
						Indeed, from covariant Dieudonn\'e theory
						it is equivalent to show 
						$ \cD( \tilde A ) \stackrel{2}\subset \cD( \tilde A )^ { \perp_{ \tilde A } } $. Thus it suffices to show 
						$ \cD( \tilde A)_0    \stackrel{1} \subset \cD( \tilde A )_1^{\perp_{\tilde A} },$
						which is equivalent to
						$ p \cD( \tilde B)_0^{\perp_ B  } \stackrel{1}\subset \cD( B )_1,$ which is equivalent to 
						$p\cD( B )_0 \stackrel{1}\subset  D(\tilde B)_0,$ which comes from $\eqref{condition:isogeny}.$
					\end{enumerate}  
					
					We have constructed $\tilde A$ and $\delta,$
					while $A$ and $\delta'$ are constructed in
					Lemma \ref{lem:tTMSp}.
					The inclusion $\cD(A) \subset \cD(\tilde A)$ is then induced by a prime-to-$p$ isogeny $\alpha:A \to\tilde A.$
					
					Finally we set 
					$\tilde \theta'(t)= (\tilde B, \lambda_{\tilde B}, \eta_{ \tilde  B},
					A, \lambda_A, \eta_A, 
					\tilde A, \lambda_{\tilde A}, \eta_{\tilde A} ,
					B, \lambda_{  B} , \eta_{  B},
					\delta',\alpha,\delta
					).$ 
					It is easy to verify $\tilde \theta'$  is the inverse of   $\tilde \rho'. $
					The equivariance under prime-to-$p$ Hecke correspondence is clear.
				\end{enumerate}
			\end{proof}
		\end{lem} 
		
		\section{Level lowering}\label{sec:proof}

		\subsection{Langlands group of $G$}\label{subsection:subgroup}
	Denote by $Z$ the center of $G$
      and $G_0$ the unitary group associated with $G$.
	By \cite[p. 378]{Kni01} we have  $Z(\dA)=\dA_F^\times$ and
	\[
          G(\dA)=Z(\dA)G_0(\dA).
	\]
	Let $P$ be the parabolic of $G$ and $M\subset P$ be the Levi factor of $G$ such that $P(\dQ)$ consists of matrices
 		under the standard basis of $(\Lambda,\psi)$
		of the form
		\[
	P(\dQ)=\left\{
		\left.
		 \left(\begin{array}{ccc}
			a & *&*  \\
			0 &   b & * \\
			0 & 0 &   c
		\end{array} \right)   \right\vert a,b,c\in F^\times, ac^\fc =bb^\fc  
		\right\}
				 \] 
		and $M\subset P$ be the subgroup of diagonal matrices.
	The Langlands dual group of $G$ and $G_0$ are
 \[
	\widehat{G_0}=\GL_3(\dC)  ,\quad  \widehat{G}=\GL_3(\dC) \times \dC^\times  ,\]
	\[  ^L G_0=\widehat{G_0} \rtimes \Gal(F/\dQ), \quad ^L{G}=\widehat{G} \rtimes \Gal(F/\dQ).
 \]
Let $c$ be the nontrivial element in $  \Gal(F/\dQ)$. The action of $c$ on $\widehat{G}$ is given by 
\[
c(g,\lambda)=(\Phi (^t{g})^{-1} \Phi,\lambda \det g ).
\]
	     	The embedding   $G_0 \hookrightarrow G$ corresponds to the natural projection 
   $\widehat{G}\to   \widehat{G_0}.$
   
	Let $p$ be a rational prime  unramified in $F.$   By Satake's classification,
 each unramified principal series  $\sigma_p$ of $G_p$
 corresponds to a $\widehat G$-conjugacy class of semisimple elements in $\widehat{G}\rtimes \Frob_p$ where $\Frob_p$ is the image of an Frobenius element at $p,$
 called the Langlands/Satake parameter of $\sigma_p.$  
	     		\subsection{Classification of unramified principal series at an inert place}   \label{subsec:principal}
	Keep the notation of Section \ref{subsection:subgroup}. 
	Suppose  $p$ is inert in $F$. 
	 Let $\r{LC}(P_p \backslash G_p)$ be the space 
			of locally constant functions on $P_p \backslash G_p,$
			equipped with the natural action by $G_p$ via right multiplication.	 
	Let $\St_p$ be the quotient space of $\r{LC}(P_p \backslash G_p)$ by the constant function. Then $\St_p$ is an irreducible admissible representation of $G_p$, called the Steinberg representation of $G_p$.
	
	Moreover, let $\nu:G\to\dG_m$ be the similitude homomorphism. For any $\beta\in \dC^\times,$ let $\mu_\beta:G_p\to \dC^\times$ be the composite
\[
\mu_\beta:G_p  \xrightarrow{g\mapsto \frac{\det g  }{\nu(g)} } \dQ_p^\times \xrightarrow
{ x \mapsto \beta^{\val_p(x)} } \dC^\times \\
\]

	 Any unramified character of $M_p$
has the form
\begin{align*}
\chi_{\alpha,\beta}: &M_p \to \dC^\times \\
 &\left(\begin{array}{ccc}
			a & 0&0  \\
			0 &   b & 0 \\
			0 & 0 &   c
		\end{array}\right) 
		  \mapsto \alpha^{\val_p(a)-\val_p(b)}\beta^{\val_p(b)}
\end{align*}
where $\alpha,\beta \in \dC^*$ and $\val_p$ is the $p$-adic valuation on $F_p.$

Denote by $I_{\alpha, \beta}:=\operatorname{Ind}_{P_p}^{G_p} (\chi_{\alpha, \beta} )$ be the normalized unitary induction of $\chi_{ \alpha, \beta}$, viewed as a character on $P_p$ trivial on its unipotent radical.
Then 
$I_{\alpha, \beta}\vert_{G_{0,p}}$ coincides with $I(\alpha)$ in the notation  of \cite[3.6.5, 3.6.6]{BG06}.
We list the properties of $I_{\alpha,\beta}:$
	\begin{enumerate}
			\item If $\alpha\neq p^{\pm2} ,  -p^{\pm1}$, then $I_{\alpha,\beta} $ is irreducible.
			\item If $\alpha=p^{\pm 2}, I_{\alpha,\beta}$ has two
Jorden-Holder foctors:  $\St_p \otimes \mu_\beta$ and $\mu_\beta.$
			\item If $\alpha=-p^{\pm 1}$, then $I_{\alpha,\beta}$ has two Jordan-H\"older factors,  
			$\pi_\beta^n$ which is  unramified and non-tempered, and $\pi_\beta^2$ which is ramified and square-integrable.  
			\item 
			The central character of $I_{\alpha,\beta}$ is
\begin{align*}
Z_p \cong F_p^\times  \xrightarrow{}& \mathbb{C}^* \\
b  \mapsto& \beta^{\val_p(b)}.
\end{align*}
			\item\label{dimIab}
			For all  $\alpha,\beta \in \mathbb{C}^{*}, \operatorname{dim} I_{\alpha,\beta}^{K_{p}}=\operatorname{dim} I_{\alpha,\beta}^{\tilde{K}_{p}}=1, \operatorname{dim} I_{\alpha,\beta}^{\Iw_{p}}=2. $ 
			\item\label{dim} $\mathrm{St}_p^{K_p}=\mathrm{St}_p^{\tilde{K}_p}=0, \operatorname{dim}(\pi_\beta^{n})^{K_p}=\operatorname{dim}(\pi_\beta^{2})^{\tilde{K}_p}=1,(\pi_\beta^{n})^{\tilde{K}_p}=(\pi_\beta^{2})^{K_p}=0.$
			\item\label{Iwahoricfix}
			Let $\pi_{p}$ be an admissible irreducible representation of $G_p$. Then $\pi_{p}^{\Iw_p} \neq 0$ if and only if it is a Jordan-H\"older factor 
			of $I_{\alpha,\beta}$   for $\alpha,\beta \in \dC^* $\cite[Theorem 3.8]{Car79}.		
 \item
  The Langlands parameter of $I_{\alpha,\beta}$ is the $\widehat G$-conjugacy class  of 
 \[
  t_{\alpha, \beta}=\left(
\left(\begin{array}{ccc}
			\alpha & 0&0  \\
			0 &   \beta/\alpha & 0 \\
			0 & 0 &   1
		\end{array}\right) 
  , 1\right) \rtimes c.
 \]
 Note that   $t_{\alpha, \beta}^2=
 \left(
\left(\begin{array}{ccc}
			\alpha & 0&0  \\
			0 &   1 & 0 \\
			0 & 0 &   \alpha^{-1}
		\end{array}\right) 
  , \beta \right) \in \widehat{G}
 $.
 		\end{enumerate}

		\subsection{Automorphic   and Galois representation}\label{subsec:autoGal} 
 Let $\pi =\otimes_{v} \pi_{v}$ be a  cuspidal automorphic  representation of $G (\mathbb{A}   )$. Let $\pi_0$ be the restriction of $\pi$ to $G_0(\dA)$
  and   $\chi_\pi$ be the central character of $\pi.$
  Recall that Rogawski defined, a base change map from automorphic representations of $G_0(\dA)$(resp. $G(\dA))$ to $G_0(\dA_F)\cong \GL_3(\dA_F)$(resp. $\GL_3(\dA_F)\times \GL_1(\dA_F)).$ 
  Denote by $\pi_{0F}$(resp. $\pi_F$) the base change of $\pi_0$(resp. $\pi$).
  By \cite[Lemma 4.1.1]{Rog92},   we have
\[
\pi_F=\pi_{0 F} \otimes \ol{\chi}_{\pi}
\]
as a representation of $\mathrm{GL}_3(\dA_F) \times \mathrm{GL}_1(\dA_F),$
 where $\ol{\chi}_\pi$ is the character $z \mapsto$ $\chi_\pi(\bar{z})$. 
			We say  $\pi$ is \emph{stable} \cite[Theorem 13.3]{Rog90} if  $\pi_{0F}$ is a cuspidal representation.	
				
	Let $\square $ be a finite set of places of $\dQ$ 
   containing the archimedean  place such that $\pi$ is unramified outside $\square,$
   	   $\ell$ be a rational prime and fix an isomorphism $ \iota_\ell: \dQ_\ell^\ac \to \dC.$
			Let $p\nmid \square$ be a finite place of $\dQ$ unramified in $F,$
$t_{\pi,p}\in {^L{G}}$ be the Satake parameter of $\pi_p,$ well defined  up to $\widehat{G}$-conjugacy, and 
$t_{\pi_0,p}\in ^L{G_0}$ be the image of  $t_{\pi,p}$ via the canonical projection 
 $ ^L G \to  ^L{G_0}.$
  \begin{enumerate}
 \item\label{Satakesplit}   If $pO_F=ww^\fc$ splits, then $t_{\pi_0,p}\in \widehat{ G}=\GL_3(\dC)$
 and 
 \[
 \{ t_{\pi_{0F},w},t_{\pi_{0F},w^\fc} \} =\{t_{\pi_0,p}, ^t t_{\pi_0,p}^{-1}\}
 \]
 \item\label{Satakeinert}
  If $p$ is inert in $F$, then $t_{\pi_0,p}\in \widehat{ G}\rtimes \Frob_p $
  and $t_{\pi_{0F},p}=t_{\pi_{0},p}^2  \in  \GL_3(\dC).$  If $\pi_p=I_{\alpha,\beta}$ for $\alpha,\beta\in\dC^\times,$
then  $t_{\pi_{0F},p}=\left(\begin{array}{ccc}
			\alpha & 0&0  \\
			0 &   1 & 0 \\
			0 & 0 &   \alpha^{-1}
		\end{array}\right).$
  \end{enumerate} 
	Assume now $\pi$ is stable and cohomological with trivial coefficient, i.e.,
\[\rH^*(\fg,K_\infty;\pi_\infty)\neq 0\]
			 where
			 $K_\infty$ is defined in Section \ref{subsec:Picard}
			and  $\mathfrak{g}= \Lie G(\mathbb{R}) \otimes \mathbb{C}.$
Blasius and Rogawski \cite[1.9]{BR92} defined a semisimple 3-dimensional $\ell$-adic representation
 \[
\rho_{\pi_0,\ell}:\Gal(F^\ac/F)\to \GL_3(\dQ_\ell^\ac)
\]
attacted to $\pi_0$(or $\pi_{0,F}$) that is characterized as follows:
 \begin{enumerate}[resume]
 \item   $\rho_{\pi_0,\ell}$ is unramified outside $\square\cup \{\ell \}.$
 \item\label{Satakepi0} Let $w$ be a non-archimedean place of $F$ with $w\nmid \square\ell$ 
 and $\Frob_w \in  \Gal(F_w^\ac/F_w) \hookrightarrow \Gal(F^\ac/F)$  be a geometric Frobenius of $w$. Then the characteristic polynomial of 
$
 \rho_{\pi_0,\ell}  (\operatorname{Frob}_w )
 $
 coincides with that of $\iota_\ell(t_{\pi_{0F}, w})q_w,$ 
 where 
 $t_{\pi_{0F}, w}\in\GL_3(\dC)$ is the Satake parameter of $\pi_{0F}$ at $w$, which
 is well defined up to conjugation. 
  \end{enumerate}
 
 Since $\pi$ is cohomological with trivial coefficient,
 $\chi_{\pi,\infty}:\dC^\times \to \dC^\times$ is trivial.
 By class field theory, $\iota_\ell \circ \chi_\pi$ can be viewed as a character of $\Gal(F^\ac/F).$ We put 
 \begin{equation}\label{eq:decompcentral}
 \rho_{\pi, \ell}:=\rho_{\pi_0,\ell} \otimes (\iota_\ell  \circ  \ol\chi_\pi).
\end{equation}
 Let $L/\dQ_\ell$ be a sufficiently large finite extension such that $\r{Im}(\rho_{\pi, \ell})\subseteq \GL_3(L).$
 Let $M^\circ$ be a $\Gal(F^\ac/F)$-stable $O_L$-lattice in the representation space of $\rho_{\pi,\ell}.$ We denote by $\ol\rho_{\pi,\ell}$ the semi-simplification of $M^\circ/\varpi_L M^\circ$ as $\Gal(F^\ac/F)$-representation. By Brauer-Nesbitt theorem, 
 $\ol\rho_{\pi,\ell}$ is independent of the choice of $M^\circ.$
 
 By the local-global compatibility, 
 If $p$ is inert in $F$ and $\pi_p\cong \St_p \otimes \mu_\beta$
  for some $\beta \in \dC^\times, $
then the multiset of eigenvalues of $\ol\rho_{\pi,\ell} (\Frob_p)$ is    
 $ \{ \iota_\ell^{-1}(\beta) p^4, \iota_\ell^{-1}(\beta)p^2 , \iota_\ell^{-1}(\beta)  \} \mod \ell.$
 
 \subsection{Spherical Hecke algebra}
		\cite[3.3.1]{BG06}
		For a finite place $p$ of $\dQ$ at which $G$ is unramified, let $K_p$ denote a hyperspecial subgroup of $G_p$. Denote by $\dT(G_p,K_p):=\dZ[K_p\backslash G_p/K_p]$   
		the  Hecke algebra   of all $\dZ$-valued locally constant, compactly supported bi-$K_p$-invariant functions on $G_p$. 
		It is known that  $\dT(G_p,K_p)$ is a \emph{commutative} algebra with unit element given by the  characteristic function of $K_p$.
		  We put   $K^\square:=\prod_{p\notin \square} K_p.$
		Denote by $\dT(G^\square,K^\square)$ the 
		prime-to-$\square$ spherical Hecke algebra 
		\[
		 \dT(G^\square,K^\square):=\bigotimes_{p\notin \square} \dT(G_p,K_p).
		\]
 Suppose $(\pi^\square)^{K^\square}\neq 0.$ Then      $\dim (\pi^\square)^{K^\square}=1$ and   there exists a homomorphism $\phi_{\pi}: \dT^\square \to O_L$  such that $T\in\dT^\square$ acts on $(\pi^\square)^{K^\square}$ via $ \iota_\ell(\phi_\pi(T)).$
 Let $\lambda$ be the place in $L$ over $\dQ_\ell.$
  Define \[\ol\phi_{\pi,\ell}:\dT^\square\xrightarrow{\phi_\pi} O_L \to O_L/\lambda, 
  		\quad		\fm:=\ker \ol\phi_{\pi,\ell}.
							\]
							The residual Galois representation $\ol\rho_{\pi,\ell}$ depends only on $\fm$ thus is also denoted by $\ol\rho_\fm.$
		
		With the above preparations we can state the main theorem:
		\begin{theorem}\label{thm:main}
			Let $\pi$ be a stable cuspidal automorphic representation 
			of $G(\dA)$      cohomological with trivial coefficient.
		Let $p$ be a prime number inert in $F$. Suppose that
			\begin{enumerate}
			\item $ \pi_{p}\cong   \operatorname{St}_{p}\otimes \mu_\beta$
			for some $\beta\in \dC^\times$ as defined in 
			Section \ref{subsec:principal};
			\item\label{middeg}
			 if $i\neq 2$ then $\rH^i(S\otimes{F^\ac},\dF_\ell)_\fm = 0;$
				\item \label{absirr}
				$\ol{\rho}_{\pi,\ell }$ is absolutely irreducible;
				\item
				$\ol\rho_{\pi,\ell}$ is unramified at $p$;
				\item $ \ell\nmid(p-1)(p^3+1).$
			\end{enumerate} 
			Then there exists a cuspidal automorphic  representation $\tilde\pi$ of $G(\dA)$ such that  ${\tilde\pi}^{K^pK_p}\neq 0$  and 
			$ \ol{\rho}_{\tilde\pi,\ell} \cong  \ol{\rho}_{\pi,\ell}.$
		\end{theorem}
		To prove the theorem, we will firstly use the Rapoport-Zink weight-monodromy spectral sequence to study the cohomology of Picard modular surface, then we argue by contradiction.		
		 We need some preliminaries on the compactification of Shimura varieties.

 		\subsection{Borel-Serre compactification of $S_0(p)$}
		Let ${S_0(p)}^\BS $ be the Borel-Serre compactification of $ S_0(p)(\dC) $
		and  $\partial {S_0(p)}^\BS $ the boundary. 
		 By \cite[Lemma 3.10]{NT16}
		we have a $G(\dA^\infty)$-equivariant isomorphism
	\begin{equation}\label{eq:boundary}
			\partial {S_0(p)}^\BS (\mathbb{C}) \cong 
			P(\mathbb{Q}) \backslash( G (\dA^\infty  )/K^p\Iw_p  \times e(P) )
			\cong \Ind_{P(\dA^\infty)}^{G(\dA^\infty)} 
				P(\mathbb{Q}) \backslash( P (\dA^\infty  )/K_P^p\Iw_p   \times e(P) )
	\end{equation}
		where $e(P)$ is the smooth manifold with corners described in \cite[\S 7.1]{BS73}
		and $K_P^p=K^p\cap P(\dA^{\infty,p}).$

		\begin{lem}\label{lem:boundaryvanish}
		Keep the notations and assumptions of Theorem \ref{thm:main}. We have
			\[
			\rH^{*}(\partial S_0(p)^\BS, \mathbb{F}_{\ell})_{\mathfrak{m}}=0.
			\]
			\begin{proof}
			Suppose on the contrary that $
			\rH^{*}(\partial S_0(p)^\BS, \mathbb{F}_{\ell})_{\mathfrak{m}}\neq 0.
			$	 
			We will show that $\ol\rho_{\pi,\ell}$ is reducible, which contradicts 
			 the condition \eqref{absirr} in Theorem \ref{thm:main}.
			  Since $\ol\rho_{\pi,\ell}\cong  \ol\rho_{\pi_0,\ell}\otimes (\iota_\ell \circ  \ol\chi_\pi)$ by \eqref{eq:decompcentral},
 it suffices to show that $\ol\rho_{\pi_0,\ell}$ is reducible.
 		Put $K_P^\square=K^\square\cap P(\dA^\square), K_M^\square=K^\square \cap M(\dA^\square),$ etc.
		We have a Satake map
		\[
		\cN: \dT(G^\square,K_{G}^\square) 		\to  \dT(M^\square,K_{M}^\square).
		\]
 	 	Following the argument of \cite[p. 36]{ACC+22} or \cite[Theorem 4.2]{NT16}, 
		 since $\fm$ is in the support of $\rH^*(\partial {S_0(p)}^\BS,\dF_\ell) ,$ 
 there exists a subgroup
  $ {K}_M^{\prime} \subset  {K}_M$ with  
  $ ( {K}_M^{\prime} )^\square={K}_M^\square$
  and a maximal ideal $\mathfrak{m}^{\prime}$ of $\mathbb{T}(M^\square, {K_M}^\square)$ in the support of 
  $\rH^0 (M(\dQ) \backslash M(\dA^\infty) /K_M' , \dF_\ell )$ such that $ {\mathfrak{m}} =\cN^{-1} (\mathfrak{m}^{\prime} ).$ In other words, there exists a homomorphism
  $\ol\theta_{\pi,\ell}: \mathbb{T}(M^\square, K_M^\square) \to \ol{L}$ for a finite extension 
  $\ol{L}$ of 
  $\dF_\ell$  
such that  $\ol \phi_{\pi,\ell}=\ol\theta_{\pi,\ell}\circ \cN.$

Put $H:=\Res_{F/\dQ}\dG_m.$ The standard Levi $M$ is a torus
			\begin{align*}
			M &\cong H \times H  \\
		\diag(a,b,c) &\mapsto (a,b ).
			\end{align*}   
			We can now assume $K_M'=K_H'\times K_H'$ which implies
	\[
	\dT(M^\square,K_{M}'^\square)\cong 
	\dT(H^\square,K_{H}'^\square)\otimes  \dT(H^\square,K_{H}'^\square).
	\]
	Since $H(\dA)=\dA_F^\times,$
  $\ol\theta_{\pi,\ell} $ 
 is equivalent to two Hecke characters
 $\ol\psi_1,\ol\psi_2:  \dA_F^\times/F^\times K_H' \to \ol{L}.$
  
 By class field theory, $\ol\psi_1$ and $\ol\psi_2$ correspond to two Galois characters
 $\ol\sigma_1,\ol\sigma_2:\Gal(F^\ac/F)\to \ol{L}^\times$ such that 
 for a place $w$ in $F$ and  a uniformizer $\varpi_w$ in $F_w \subset \dA_F^\times,$
  we have
 \[
 \ol\sigma_i(\Frob_w)=\ol\psi_i(\varpi_w).
 \]
    We claim that
    \begin{equation}\label{eq:reducible}
 \ol\rho_{\pi_0,\ell} 
 \cong
 (
 \ol\sigma_1 \oplus  \ol\sigma_2\cdot   \ol\sigma_2^{\fc,\vee} \oplus \ol\sigma_1^{\fc,\vee})
 \otimes \epsilon_\ell
\end{equation}
  where $\epsilon_\ell$ is the $\ell$-adic cyclotomic character and $\ol\sigma_i^{\fc,\vee}(g):=\ol\sigma_i((g^{\fc})^{-1})$. 
  Indeed, 
by Chebatorev density and Brauer-Nesbitt theorem,
 it suffices to verify that  for every place
					$q=ww^\fc$ split in $F$,
					the eigenvalues of $\Frob_w$ for  
					$\ol\rho_{\pi_0,\ell} $ and
 $
  (
 \ol\sigma_1 \oplus  \ol\sigma_2\cdot \ol\sigma_2^{\fc,\vee} \oplus \ol\sigma_1^{\fc,\vee})
 \otimes \epsilon_\ell
 $					coincide.
					
 	To show this,  recall that
\begin{align*}
G\left(\dQ_q\right)&=\left\{g \in \mathrm{GL}_3\left( F\otimes_\dQ \dQ_q\right) \mid {^t g^{\fc}}  \Phi g=\nu(g) \Phi 
\text { for some }
 \nu(g) \in \dQ_q^\times \right\}\\
&=\left\{ 
g=(g_w,g_{w^\fc}) \in \GL_3(F_w)\times \GL_3(F_{w^c}) \mid
 g_{w^\fc}= \nu(g) \Phi (^t g_w^{-1}) \Phi
 \right\}.
 \end{align*}
 Therefore, we have an isomorphism
		\begin{align*}
	  	   G_q&\cong \GL_3(F_w) \times  \dQ_q^\times  \\
		g &\mapsto (g_w, \nu(g))
			\end{align*} 
		under which 
		$g=\diag(a,b,c) \in M_q$ is identified with 
	$	(\diag(a_w,b_w,c_w),b_wb_{w^\fc}).$
 If $T\subset \GL_3$ denotes the diagonal torus,   we have  an isomorphism
 			\begin{align*} 
		   T_q\times \dQ_q^\times  \cong& M_q  \\
	    ( \diag(a,b,c),\nu) \mapsto& \diag((a,\nu/c) , (b,\nu/b),  (c,\nu/a)).
			\end{align*} 	
	Since $H_q\cong F_w^\times \times F_{w^\fc}^\times,$ we have 
		\begin{align*} 
		   T_q\times \dQ_q^\times  \cong& H_q \times H_q   \\
	    ( \diag(a,b,c),\nu) \mapsto& ((a,\nu/c) , (b,\nu/b)).			\end{align*}  			The local component at $q$ of $\ol\psi_1\ol\psi_2$ is given by
 		\begin{align*}
		(\ol\psi_1\ol\psi_2)_q:
		 (\diag(a,b,c),\nu)
		&\mapsto 
		\ol\psi_{1,w}(a)
		\ol\psi_{1,w^\fc}(\nu/c)
		\ol\psi_{2,w}(b)
		\ol\psi_{2,w^\fc}(\nu/b)
		\\
		&=(\ol\psi_{1,w^\fc}\ol\psi_{2,w^\fc})(\nu)
		\ol\psi_{1,w}(a)
		(\ol\psi_{2,w}\ol\psi_{2,w^\fc}^{-1})(b)
		\ol\psi_{1,w^\fc}^{-1}(c).
				\end{align*}
  Let $\widehat{M}=\widehat{T} \times \dG_m$ be the torus 
  over $\dZ_\ell$ dual to $M_{\dQ_q}.$
  By duality,
  the group of the unramified characters of $M_{\dQ_q}$ with values in $\ol L^\times$
  is isomorphic to 
  \[
  X^*(M_{\dQ_q})\otimes L^\times=X_*(\widehat{M})\otimes L^\times \cong \widehat{M}(\ol L),
  \]
  where $X^*(M_{\dQ_q})$ (resp. $X_*(\widehat{M})$)
  denotes the character group of $M_{\dQ_q}$
  (resp. the cocharacter group of $\widehat{M}$).
  With this identification $(\ol\psi_1\ol\psi_2)_q$ corresponds to the semisimple element
  \[
 \left( \left(\begin{array}{ccc}
			\ol \psi_{1,w}(q) & 0&0  \\
			0 &  (\ol \psi_{2,w}/ \ol  \psi_{2,w^\fc}) (q) & 0 \\
			0 & 0 &   \ol  \psi_{1,w^\fc}^{-1}(q)
		\end{array}\right),\nu \right) \in \widehat{M}(\ol L).
				\]
		 By Section \ref{subsec:autoGal} and Satake isomorphism
 the eigenvalues of $\ol\rho_{\pi_0,\ell}(\Frob_w)$ are given by  
\[
q\{  \ol \psi_{1,w}(q), 
   (\ol \psi_{2,w}\ol \psi_{2,w^\fc}^{-1})(q), 
    \ol\psi_{1,w^\fc}^{-1}  (q ) \}.
\]
		By Chebotarev density, the equality \eqref{eq:reducible} holds. This finishes the proof of the lemma.
			\end{proof}
		\end{lem}
		
\begin{corollary}\label{cor:boundaryvanish}	
Denote by $S_0(p)^\BB$ the Baily-Borel compactification of $S_0(p).$
Then we have canonical isomorphisms
	\begin{equation}\label{eq:boundaryvanish}
			\rH_c^{2}( {S_0(p)}\otimes F^\ac, \dF_\ell)_\fm \cong 
				\IH^{2}( {S_0(p)}^\BB \otimes F^\ac, \dF_\ell)_\fm
				\cong 
			\rH^{2}( {S_0(p)}\otimes F^\ac, \dF_\ell)_\fm.
		\end{equation}	
		\begin{proof}
			One has an exact sequence of Betti cohomology \cite[Remark 1.5]{CS19}
	\begin{equation}\label{eq:compactsuppexact}
		0 \rightarrow \rH^{1}(\partial S_0(p)^\BS, \mathbb{F}_{\ell})
		\rightarrow \rH^{2}_c( {S_0(p)} ,  \mathbb{F}_{\ell}) 
		\rightarrow \rH^{2}( {S_0(p)}, \mathbb{F}_{\ell}) \rightarrow \rH^{2}(\partial S_0(p)^\BS, \mathbb{F}_{\ell}) \rightarrow 0
	\end{equation}
	which is equivariant under $\dT(G^\square,K^\square)$-action.
				By \cite[1.8]{HLR86} the intersection cohomology group   $\IH^{2}( {S_0(p)}^\BB\otimes F^\ac, \dF_\ell)_\fm
		$ is the image of the map $\rH^{2}_c( {S_0(p)}\otimes{F^\ac}, \mathbb{F}_{\ell})_{\mathfrak{m}} \rightarrow \rH^{2}( {S_0(p)}\otimes{F^\ac}, \mathbb{F}_{\ell})_{\mathfrak{m}}.$  
	The corollary then follows from Lemma \ref{lem:boundaryvanish}.
	\end{proof}
		\end{corollary}
		\subsection{Generalities on the weight-monodromy spectral sequence} 
		\cite[Corollary 2.2.4]{Sai03}, \cite[2.1]{Liu19}\label{monodromy}.
		Let $K$ be a henselian discrete valuation field with residue field $\kappa$ and a separable closure $\bar{K}$. We fix a prime $p$ that is different from the characteristic of $\kappa$. Throughout this section, the coefficient ring $\Lambda$ will be $\dF_\ell.$
		We first recall the following definition.
		\begin{definition}[Strictly semistable scheme] Let $X$ be a scheme locally of finite presentation over Spec $O_{K}$. We say that $X$ is strictly semistable if it is Zariski locally \'etale over
			\[		
			\text { Spec } O_{K} [t_{1}, \ldots, t_{n} ]/(t_{1} \cdots t_{s}-\varpi)
			\]
			for some integers $0 \leq s \leq n$ (which may vary) and a uniformizer $\varpi$ of $K$.
		\end{definition}
		Let $X$ be a proper strictly semistable scheme over $O_{K}$. The special fiber $X_{\kappa}:=X \otimes_{O_{K}} \kappa$ is a normal crossing divisor of $X$. Suppose that $\left\{X_{1}, \ldots, X_{m}\right\}$ is the set of irreducible components of $X_{\kappa}$. For $r \geq 0$, put
		\[
		X_{\kappa}^{(r)}=\coprod_{I \subset\{1, \ldots, m\},|I|=r+1} \bigcap_{i \in I} X_{i} .
		\]
		Then $X_{\kappa}^{(r)}$ is a finite disjoint union of smooth proper $\kappa$-schemes of codimension $r$. From \cite[page 610]{Sai03}, we have the pullback map
		\[
		\delta_{r}^{*}: \mathrm{H}^{s}(X_{\bar{\kappa}}^{(r)}, \Lambda(j)) \rightarrow \mathrm{H}^{s}(X_{\bar{\kappa}}^{(r+1)}, \Lambda(j))
		\]
		and the pushforward (Gysin) map
		\[
		\delta_{r *}: \mathrm{H}^{s}(X_{\bar{\kappa}}^{(r)}, \Lambda(j)) \rightarrow \mathrm{H}^{s+2}(X_{\bar{\kappa}}^{(r-1)}, \Lambda(j+1))
		\]
		for every integer $j$. These maps satisfy the formula
		\[
		\delta_{r-1}^{*} \circ \delta_{r *}+\delta_{r+1 *} \circ \delta_{r}^{*}=0
		\]
		for $r \geq 1$. For reader's convenience, we recall the definition here. For subsets $J \subset I \subset$ $\{1, \ldots, m\}$ such that $|I|=|J|+1$, let $i_{J I}: \bigcap_{i \in I} X_{i} \rightarrow \bigcap_{i \in J} X_{i}$ denote the closed immersion. If $I=\left\{i_{0}<\cdots<i_{r}\right\}$ and $J=I \backslash\left\{i_{j}\right\}$, then we put $\epsilon(J, I)=(-1)^{j}$. We define $\delta_{r}^{*}$ to be the alternating sum $\sum_{I \subset J,|I|=|J|-1=r+1} \epsilon(I, J) i_{I J}^{*}$ of the pullback maps, and $\delta_{r *}$ to be the alternating sum $\sum_{I \supset J,|I|=|J|+1=r+1} \epsilon(J, I) i_{J I *}$ of the Gysin maps.
		
		\begin{remark}
			In general, the maps $\delta_{r}^{*}$ and $\delta_{r *}$ depend on the ordering of the irreducible components of $X_{\kappa}$. However, it is easy to see that the composite map $\delta_{1 *} \circ \delta_{0}^{*}$ does not depend on such ordering.
		\end{remark}
		Let us recall the weight spectral sequence attached to $X$. Denote by $K^{\text {ur }} \subset  {K^\ac}$ the maximal unramified extension, with the residue field $\bar{\kappa}$ which is a separable closure of $\kappa$. Then we have ${G}_{K} / {I}_{K} \simeq {G}_{\kappa}$. Denote by $t_{0}: {I}_{K} \rightarrow \Lambda_{0}(1)$ the $(p$-adic) tame quotient homomorphism, that is, the one sending $\sigma \in {I}_{K}$ to $(\sigma(\varpi^{1 / p^{n}}) / \varpi^{1 / p^{n}})_{n}$ for a uniformizer $\varpi$ of $K$. We fix an element $T \in I_{K}$ such that $t_{0}(T)$ is a topological generator of $\Lambda_{0}(1)$.
		
		We have the weight spectral sequence $\rE_{X}$ attached to the (proper strictly semistable) scheme $X$, where
		\begin{equation}\label{eq:RZspectral}
		({\rE}_{X})_{1}^{r, s}=\bigoplus_{i \geq \max (0,-r)} {\rH}^{s-2 i}(X_{\bar{\kappa}}^{(r+2 i)}, \Lambda(-i)) \Rightarrow {\rH}^{r+s}(X_{\bar{K}}, \Lambda)
		\end{equation}
		This is also known as the Rapoport-Zink spectral sequence, first studied in \cite{RZ82}; here we will follow the convention and discussion in \cite{Sai03}. For $t \in \mathbb{Z}$, put ${ }^{t} \mathrm{E}_{X}=\mathrm{E}_{X}(t)$ and we will suppress the subscript $X$ in the notation of the spectral sequence if it causes no confusion.
		By \cite[Corollary 2.8(2)]{Sai03}, we have a map $\mu: \mathrm{E}_{\bullet}^{\bullet-1, \bullet+1} \rightarrow \mathrm{E}_{\bullet}^{\bullet+1, \bullet-1}$ of spectral sequences (depending on $T$ ) and its version for ${ }^{r}\rE$. 
		The map $ \mu^{r, s}:=\mu_{1}^{r, s}:{ }^{t} \mathrm{E}_{1}^{p-1, q+1} \rightarrow{ }^{t} \mathrm{E}_{1}^{r+1, s-1}$ is the sum of its restrictions to each direct summand $\mathrm{H}^{s+1-2 i}(X_{\bar{\kappa}}^{(2 i+1)}, \Lambda(r-i))$, and such restriction is the tensor product by $t_{0}(T)$ (resp. the zero map) if $\mathrm{H}^{s+1-2 i}(X_{\bar{\kappa}}^{(2 i+1)}, \Lambda(t-i+1))$ does (resp. does not) appear in the target.  
		The map $\mu^{r, s}$   induces a map, known as the monodromy operator,
		\[
		\tilde{\mu}^{r, s}:{ }^{t} \mathrm{E}_{2}^{r-1, s+1} \rightarrow{ }^{t} \mathrm{E}_{2}^{r+1, s-1}(-1)
		\]
		of $\Lambda[G_{\kappa}]$-modules. 
		\subsection{Weight-monodromy spectral sequence for ${S_0(p)}$ }
		We will try to apply the weight-monodromy spectral sequence to the surface 
		$f:{S_0(p)} \to \Spec(O_{F}\otimes \dZ_{(p)}).$  In the derivation of weight-monodromy spectral sequence $f$ is required to be proper to get $  \rH^i({S_0(p)}\otimes\Fpac, R\Psi \dZ_\ell)\cong \rH^i({S_0(p)}\otimes{F^\ac},\dZ_\ell).$
		 		However, in our case $f$ is not proper. Fortunately, according to  \cite[Corollary 4.6] {LS18},
	$ \rH^i({S_0(p)}\otimes\Fpac, R\Psi \dZ_\ell)\cong \rH^i({S_0(p)}\otimes{F^\ac},\dZ_\ell)$ still holds.
				  Put
		\[
		Y^{(2)}=   Y_{0,1,2} \otimes \dF_p^\ac, ~
		Y^{(1)}=(Y_{0,1}\sqcup   {Y_{0,2}} \sqcup   {Y_{1,2}} ) \otimes \dF_p^\ac,~ 
		Y^{(0)}= (Y_{0} \sqcup   Y_{1} \sqcup   Y_{2})\otimes \dF_p^\ac .
		\]
		The spectral sequence \eqref{eq:RZspectral} with $\Lambda=\dF_\ell$ reads 
		\begin{equation}\label{eq:RZPicard}
			\xymatrix@R=0pt{
				\rH^{0}(Y^{(2)})_{\fm} (-2)
				\ar[r]&  \rH^{2}(Y^{(1)})_{\fm}(-1) \quad 
				\ar[r] &  \rH^{4}(Y^{(0)})_{\fm}  &&
				\\
				& \rH^{1}( Y^{(1)})_{\fm}(-1) \quad 
				\ar[r] &  \rH^{3}(Y^{(0)})_{\fm} &&
				\\
				& \rH^{0}(Y^{(1)})_{\fm}(-1) \quad 
				\ar[r] &  \rH^{2}(Y^{(0)})_{\fm}  \oplus 
				\rH^{0}(Y^{(2)})_{\fm}(-1)     \ar[r] & 
				\rH^{2}(Y^{(1)})_{\fm} 
				&   
				\\
				&&    \rH^{1}(Y^{(0)})_{\fm}    \ar[r] &
				\rH^{1}(Y^{(1)})_{\fm} 
				&
				\\
				&& \rH^{0}(Y^{(0)})_{\fm}   \ar[r]   &   \rH^{0}( Y^{(1)})_{\fm}  \ar[r]   \ar[r]&
				\rH^{0}(Y^{(2)})_{\fm}
			}
		\end{equation} 
		Here we omit the coefficient $\dF_\ell$ in the cohomology group. 
	\begin{lem}\label{lem:H1Fermat}
		Let $G_0$(resp. $G_0'$) be the unitary group attached to $G$(resp. $G'$) as in Section \ref{subsection:subgroup}. Recall the inner form $G'$ defined in 
		Section \ref{sec:inner}. 
			Put $G_{0,p}:=G_0(\dQ_p), K_{0,p}:=K_p\cap G_{0,p}, K_0^p:=K^p\cap G_0^p.$ Let $K_{0,p}^1$ be the kernel of the reduction map $G_0(O_p) \to G_0(\Fpp)$.
 			Then we have an isomorphism
			\begin{equation}\label{Nuniform}
				\iota_\ell \rH^{1}( N\otimes\Fpac, \dQ_\ell^\ac)\mid_{G_0(\dA)}
				\cong
				\Map_{K_{0,p}}  (G_0'(\dQ)\backslash G_0'(\dA^\infty)/K_0^p,\Omega_3     )
			\end{equation}
			of $\dC[K_0^pK_{p,0}^1\backslash G_0'(\dA^\infty)/K_0^p K_{p,0}^1]$-modules, where $(\rho_{\Omega_3},\Omega_3)$ is the Tate-Thompson representation of $K_{0,p}$ in \cite[C.2]{LTX+22}
			and the right hand side of the isomorphism denotes the
			locally constant maps 
			$f:G_0'(\dQ)\backslash G_0'(\dA^\infty)/K_0^p \to \Omega_3$
			such that $f(gk)=\rho_{\Omega_3}(k^{-1})f(g)$ for $k\in K_{0,p}$ and $g\in G_0'(\dA^\infty).$ 
 			Moreover, let $\pi_0^{\square}$ be an irreducible admissible representation of $G_0 (\mathbb{A}^{\square})$ such that $(\pi_0^{\square})^{ K_0^{\square} }$ is a constituent of $\iota_{\ell} \mathrm{H}^{1}(N\otimes\Fpac ,  {\mathbb{Q}}_{\ell}^\ac).$ 
		Then one can complete $\pi_0^{\square}$ to an automorphic representation
$\pi_0'=\pi_0^{\square}\otimes \prod_{q\in \square}\pi'_{0,q}$		 of $G_0'(\mathbb{A})$
		  such that 
			$\mathrm{BC} (  \pi'_{0,p} )$ is a constituent of an unramified principal series of $\mathrm{GL}_{3} (F_{p} )$ with Satake parameter $\left\{-p,1,-p^{-1}\right\},$ where $\BC$
			denotes the local base change from $G_{0,p}$ to $\GL_3(F_p)$.
			\begin{proof}
			Recall the fiber of the morphism
			$N\to T$ is geometrically a Fermat curve of degree $p+1$ 
			where
			 $T(\dC)\cong
			 G'(\dQ)\backslash G'(\dA^\infty)/K^pK_p  $
			by 
			 Theorem \ref{thm:TN}\eqref{Nfiber}.
	Take $t\in T(\Fpac)$, then $\rH^1(N\otimes\Fpac\cap \theta^{-1}(t),\dQ_\ell^\ac)\mid_{G_0(\dA)} $ is a representation of $G_0(\dF_p^\ac)=K_{0,p}/K_{0,p}^1,$
				 isomorphic to $\Omega_3.$ 
				For the remaining part, note that the right-hand side of \eqref{Nuniform} is a
				 $\dC[K_0^pK_{p,0}^1\backslash G_0(\dA^\infty)/K_0^p K_{p,0}^1]$-submodule of   $\Map (G_0'(\dQ)\backslash G_0'(\dA^\infty)/K_0^pK_{0,p}^1, \dC)
				$. In particular, we can complete $\pi_0^{\square}$ to an automorphic representation 
				$\pi_0'=\pi_0^{\square} \otimes \prod_{q\in \square}\pi'_{0,q}
				$ of $G_0'(\dA)$ such that  $ \pi'_{ 0,p}  |_{K_{0,p}}$ contains $\Omega_{3}$. 
				The same argument as \cite[Theorem 5.6.4(ii)]{LTX+22} then implies 
				$\pi'_{ 0,p} \cong c$-$\Ind_{K_{p,0}}^{G_0}(\Omega_3)\cong \pi^s(1)$ where $\pi^s(1)$ appears in \cite[Proposition 13.1.3(d)]{Rog90}.
				The base change $\BC(\pi^s(1))$ has the Satake parameter
				   $\{-p,1,-p^{-1}\}$ by \cite[Proposition 13.2.2(c)]{Rog90}.
 The lemma follows.
			\end{proof}
		\end{lem}
		\begin{lem}\label{lem:cohvanish}
		Keep the notations and assumptions of Theorem \ref{thm:main}.
			Suppose there is no level-lowering, i.e., there is no automorphic representation $\pi'$ of $G(\dA)$ such that 
			${\pi'}^{K^pK_p}\neq 0$ and 
			$\ol\rho_{\pi',\ell}\cong \ol\rho_{\pi,\ell}.$ Then  
			one has	
			\begin{enumerate}
				\item\label{Scohvanish}  $\rH^2(  S\otimes\Fpac,\dF_\ell)_\fm = 0;$
				\item\label{Tcohvanish}  $\rH^2(  T\otimes\Fpac,\dF_\ell)_\fm = 0;$
			 	\item\label{tildeTcohvanish}  $\rH^0(    \tilde T\otimes\Fpac,\dF_\ell)_\fm = 0;$
				\item \label{Sblowcohvanish} 
				$\rH^*( S^\#\otimes\Fpac,\dF_\ell)_\fm = 0;$ 
				\item\label{Ncohvanish} 
				$\rH^*(  N\otimes\Fpac  , \dF_\ell)_\fm =0;$
			\end{enumerate}
			\begin{proof}
				\begin{enumerate}
					\item 
			 Suppose $\rH^2(S\otimes\Fpac, \dF_\ell)_\fm\neq 0. $	
	   By \cite[Corollary 4.6]{LS18},
 we have 
		 $ \rH^2(S\otimes{F^\ac}, \dF_\ell)_\fm \cong \rH^2(S\otimes\Fpac, \dF_\ell)_\fm \neq 0.$  
					The universal coefficient theorem gives the exact sequence  
					\[
			0	\to 	\rH^i(S\otimes{F^\ac}, \dZ_\ell)_\fm \otimes \dF_\ell \to 	\rH^i(S\otimes{F^\ac}, \dF_\ell)_\fm \to 	\rH^{i+1}(S\otimes{F^\ac}, \dZ_\ell)_\fm[\ell] \to  0,\quad i\in \dZ		\] 
		which implies that
				$\rH^2(S\otimes{F^\ac}, \dZ_\ell)_\fm$ is torsion-free
				and non-zero.  Thus		there exists a cuspidal automorphic representation $\tilde\pi$ of
					$G(\dA)$ such that
				the $\tilde\pi$-isotypic component	$\rH^2(  S\otimes{F^\ac} ,\dZ_\ell)_\fm [\tilde\pi] \otimes \dQ_\ell^\ac \neq 0$ and
					 $\tilde\pi^{K^pK_p}\neq 0 $ since $S$ is of level $K^pK_p.$
					Moreover, 
			by Section \ref{subsec:autoGal} 	the prime-to-$\square$ Hecke equivariance implies  
		    $\ol\rho_{\tilde\pi,\ell}(\Frob_q) =  \ol\rho_{\pi,\ell}(\Frob_q)$ for $q\notin\square$.
		Finally,  Chebotarev density ensures 
                				$\ol\rho_{\tilde\pi,\ell}\cong \ol\rho_{\pi,\ell}.$
					This contradicts the no-level-lowering assumption.
					\item 
						Suppose 
	$\rH^{0}( T\otimes\Fpac, \dF_\ell)_\fm\cong \rH^{0}( T\otimes{F^\ac}, \dF_\ell)_\fm \neq 0.$
 		  Since $  \rH^{0}(   T\otimes{F^\ac}, \dZ_\ell)_\fm$ is torsion-free,  
	   there exists an irreducible automorphic representation $\pi'$ of
					$G'(\dA)$ such that $\pi'^{K^p     K_p}\neq 0.$ 
				By \cite[Theorem 2.4]{Clo00} we can   transfer $\pi'$ to an automorphic representation 
				$\tilde \pi$ of $G(\dA)$ such that 
			the finite components $\tilde \pi^\infty$ and $ \pi'^\infty$ coincide.
			In particular $\tilde\pi^{K^pK_p}\neq 0.$
		The prime-to-$\square$ Hecke equivariance and Chebatrov density 
						then imply that $\ol\rho_{\tilde\pi,\ell}\cong \ol\rho_{\pi,\ell},$
		  contradicting the no-level-lowering assumption. 
		  			\item 
	 	Suppose $\rH^{0}( \tilde T\otimes\Fpac, \dF_\ell)_\fm\neq 0.$
		By the same argument as \eqref{Tcohvanish},
	   there is an irreducible automorphic representation $\pi'$ of
					$G'(\dA)$ such that $(\pi')^{ K^p  \tilde K_p}\neq 0$
	and we can again   transfer $\pi'$ to an automorphic representation 
				$\tilde \pi$ of $G(\dA)$ such that 
			the finite components $\tilde \pi^\infty$ and $ \pi'^\infty$ coincide.
						In particular $\tilde\pi^{K^p\tilde K_p}\neq 0.$
			The prime-to-$\square$ Hecke equivariance and Chebotarev density 
						then imply that $\ol\rho_{\tilde\pi,\ell}\cong \ol\rho_{\pi,\ell}.$
						
						On the other hand,
		by
		Section \ref{subsec:principal}\eqref{Iwahoricfix},
			$\tilde \pi_p$ 
					is a Jordan-H\"older factor of $I_{\alpha,\beta}$
					for some $\alpha,\beta\in\dC^\times.$
				
				If $\alpha\neq p^{\pm2},-p^{\pm1},$ then $\tilde\pi_p\cong I_{\alpha,\beta}$ thus $\tilde\pi_p^{K_p}\neq 0$ by  Section \ref{subsec:principal}\eqref{dimIab},  contradicting the no-level-lowering assumption.
					
					If $\alpha=p^{\pm2},$
					then $\tilde \pi_p\cong \St_p\otimes \mu_\beta$ or $\mu_\beta.$
				The first case is excluded since it has no non-trivial 
			$\tilde K_p$-fixed vector by  
			Section \ref{subsec:principal}\eqref{dim}.
                The second case is excluded as $\tilde\pi$ is tempered.	
                			
	If $\alpha=-p^{\pm1},$ then $\tilde \pi_p \cong \pi_\beta^n$ or $\pi_\beta^2.$
		The former is excluded since it has no non-trivial 
				$\tilde K_p$-fixed vector by Section \ref{subsec:principal}\eqref{dim}.
				For the latter
		the multiset of eigenvalues of $\ol\rho_{\tilde \pi,\ell}(\Frob_p)$
				would be $\{-p,1,-p^{-1}\}$ up to a scalar,
				leaving two possibilities:
			if $p^2\equiv -p \bmod\ell$ then $p\equiv -1\mod \ell$ thus 
			$p^2\equiv 1 \mod \ell,$   if $p^2\equiv -p^{-1} \bmod\ell$  then 
			$p^3\equiv -1 \mod \ell$,   both contradicting our assumption.	
								\item\label{blowupcoh}
					Let $E$ be the exceptional divisor of the blowup $S^\#$ of $S $ along the superspecial locus $S_{\spe}.$ 					Consider the corresponding blow up square
					\[
					\xymatrix{
						E\ar[d]_\pi  \ar[r]^j &S^\# \ar[d]^b \\
						S_{\spe} \ar [r]^i &S
					}.
					\]
			    We have a distinguished triangle \cite[\href{https://stacks.math.columbia.edu/tag/0EW5}{Tag 0EW5}]{Sta} 
						\[
					\dF_\ell \rightarrow R i_{*}(\dF_\ell\mid_{S_{\spe}}) 
					\oplus R b_{*}(\dF_\ell \mid_{ S^\#}) \rightarrow R c_{*}(\dF_\ell\mid_{E}) \rightarrow \dF_\ell[1]
					\]
					where $c=i \circ \pi=b \circ j.$
					This induces an   exact sequence of localized \'etale cohomology
					\[
					\rH^{i}(  S\otimes\Fpac, \dF_\ell)_\fm \to  
					\rH^i(  S^\#\otimes\Fpac, \dF_\ell)_\fm \oplus   \rH^i(   S_{\spe}\otimes\Fpac ,  \dF_\ell)_\fm \to
					\rH^i(  E\otimes\Fpac, \dF_\ell) _\fm  \to
					\rH^{i+1}(   S\otimes\Fpac, \dF_\ell)_\fm
					\]
					  compatible with the $\dT(G^\square,K^\square)_\fm$-action.
					Since $ \rH^*( S\otimes\Fpac, \dF_\ell)_\fm=0$
				by \eqref{Scohvanish}
					and $ \rH^0( S_{\spe}\otimes\Fpac, \dF_\ell)_\fm
					\cong \rH^0( \tilde T\otimes\Fpac, \dF_\ell)_\fm=0
					$ 	 				 by Lemma \ref{lem:tTMSp}
					and \eqref{tildeTcohvanish},
					we have an isomorphism of $\dT(G^\square,K^\square)_\fm$-modules
					\[
					\rH^i(S^\#\otimes\Fpac, \dF_\ell)_\fm \cong 
					\rH^i( E\otimes\Fpac, \dF_\ell) _\fm.
					\]
			Therefore, it suffices to show  
				  $\rH^*( E\otimes\Fpac, \dF_\ell)_\fm=0.$ 
					Since 
			 $E$ is a $\dP^1$-bundle over $S_{\spe}$
				 by the proof of Proposition \ref{prop:Y0Y1blowup}\eqref{Y0Sblow},
					we have  $ \rH^{*}( E\otimes\Fpac, \dF_\ell)_\fm= 
				\rH^{*}( S_{\spe}\otimes\Fpac, \dF_\ell)_\fm[X]/X^2 =0$
			and finish the proof.
					
					\item\label{Ncoh}
				Firstly, we have
					$\rH^{i}(  N\otimes\Fpac  , \dF_\ell)_\fm 
					\cong  \rH^{i}(  T\otimes\Fpac  , \dF_\ell)_\fm=0$
					for $i=0,2$ by \eqref{Tcohvanish}.
					If $\rH^{1}(  N\otimes\Fpac  , \dF_\ell)_\fm \neq 0$,
					then $\pi^{\square}$ appears in 
					$ \rH^1(  N\otimes\Fpac  , \dZ_\ell)_\fm\otimes \dQ_\ell^\ac$
					since $\rH^{1}(  N\otimes\Fpac  , \dZ_\ell)_\fm$ is torsion-free.
					By Lemma \ref{lem:H1Fermat} 
					we can complete $\pi^{\square}$ to 
					an automorphic representation $ \pi'=\pi^{\square}\otimes \prod_{q\in\square}\pi'_q$ of $G'(\dA)$
								 such that the Satake parameter of $\BC(\pi'_{p,0})$ is 
					$\{p,1,p^{-1}\}.$
				We can again transfer $\pi'$ to an automorphic representation 
				$\tilde \pi$ of $G(\dA)$  such that the finite components $\tilde \pi^\infty$ and $ \pi'^\infty$ coincide.
		Then		the multiset of eigenvalues of $\ol\rho_{\tilde\pi,\ell}(\Frob_p)$
				would be $\{-p,1,-p^{-1}\}$ up to a scalar.
				Comparing the eigenvalues of $\ol\rho_{\tilde\pi,\ell}$
				and $\ol\rho_{\pi,\ell}$
			as in	Lemma \ref{lem:cohvanish}\eqref{tildeTcohvanish}
				 leads to a contradiction.
				\end{enumerate}
			\end{proof}
		\end{lem}
\begin{corollary}
\begin{enumerate}
				\item  $\rH^*(  Y^{(0)} ,\dF_\ell)_\fm = 0;$
				\item\label{Ybra1}
				$
				\rH^*(   Y^{(1)} , \dF_\ell)_\fm   =0.
				$
												\end{enumerate}
		\begin{proof}
		\begin{enumerate}		
					\item
					By Proposition \ref{prop:Y0Y1blowup} we have
					isomorphisms of $\dT(G^\square,K^\square)_\fm$-module 
					$\rH^i(  {Y_0}\otimes\Fpac,\dF_\ell)_\fm\cong \rH^i( {Y_1}\otimes\Fpac,\dF_\ell)_\fm\cong \rH^i(  S^\#\otimes\Fpac,\dF_\ell)_\fm=0$ for $i=0,1,2.$
					Now we show $\rH^*(  {Y_2}\otimes\Fpac,\dF_\ell)_\fm=0.$
					By Lemma \ref{lem:PY2}, Proposition \ref{prop:PV}
					and Lemma \ref{lem:cohvanish}\eqref{Ncohvanish},
					 $Y_2$ is a $\dP^1$-bundle over $N$ and thus 
					\[
					\rH^*(  {Y_2}\otimes\Fpac,\dF_\ell)_\fm\cong \rH^*( N\otimes\Fpac,\dF_\ell)_\fm[X]/X^2=0.
					\] 
					\item 
					By Proposition \ref{prop:Y01Y02Y12}\eqref{Y01},  $ { Y_{0,1}} $ is a $\dP^1$-bundle over $\tilde T$. Thus we have   an isomorphism of $\dT(G^\square,K^\square)_\fm$-modules 
					\[
					\rH^*(  {Y_{0,1}}\otimes\Fpac,\dF_\ell)_\fm\cong \rH^*( {\widetilde T}\otimes\Fpac ,\dF_\ell)_\fm[X]/X^2=0
					\]
		by \cite[Proposition 10.1]{Mil80} and Lemma \ref{lem:cohvanish}\eqref{tildeTcohvanish}. 
			By Proposition \ref{prop:Y01Y02Y12}\eqref{Y02}\eqref{Y12},
			$Y_{0,2}$ is isomorphic to $N$, $Y_{1,2}\to N$ is a purely inseparable map, thus
				we have 
					isomorphisms of $\dT(G^\square,K^\square)_\fm$-modules
					\[
					\rH^i(  {Y_{0,2}}\otimes\Fpac,\dF_\ell)_\fm\cong
					\rH^i(  {Y_{1,2}}\otimes\Fpac,\dF_\ell)_\fm\cong \rH^i(  N\otimes\Fpac ,\dF_\ell)_\fm.
					\]
					By Lemma \ref{lem:cohvanish}\eqref{Ncohvanish}
					they all vanish.

				\end{enumerate}
 			\end{proof}
\end{corollary}	
		
		\begin{corollary} \label{lem:degenerate}
			The spectral sequence 
			\eqref{eq:RZPicard}
			localized at $\fm$ degenerates at $E_1.$
			\begin{proof}
				By Poincar\'e duality it suffices to show $E^{-1,4}_\fm= \rH^{2}(Y^{(1)}, \mathbb{F}_{\ell})_{\fm}=0$
				and 
				$E^{0,3}_\fm=\rH^{3}(Y^{(0)}, \mathbb{F}_{\ell})_{\fm} =0,$ 
				which follow from Lemma \ref{lem:cohvanish}.
			\end{proof}
		\end{corollary}

		We study the $\Gal(\dF_p^\ac/\dF_{p^2})$-action on  
		$ \rH^0(Y^{(2)}, \dF_\ell)_\fm \cong \rH^0(  { T_0(p)}\otimes\Fpac, \dF_\ell)_\fm.$ 
		Consider  the   Iwahoric Hecke algebra
		$
		\dT(G_p,\Iw_p):= \dZ [\Iw_p \backslash G_p/ \Iw_p].
		$
		The
		$ \Gal( \dF_p^\ac/\dF_{p^2})$-action  and the $\dT(G_p,\Iw_p)$-action   
		on $\rH^0({T_0(p)}\otimes\Fpac, \dF_\ell)_\fm$  commute.
		Let $\phi_{\Iw_p}$ denote
		the action of $\dT(G_p,\Iw_p)$
		on $\rH^0(Y^{(2)},\dF_\ell)_\fm.$
		 For $a\in Z(\dQ_p)=F_p^\times,$ denote by $  \ang{a} \in \dT(G_p, \Iw_p)$ the characteristic function of $a\Iw_p$.
		\begin{lem}\label{lem:Frobp}
			The action of $\Frob_{p^2}$ and $\ang{p^{-1}}$ on
			$ \rH^0(Y^{(2)}, \dF_\ell)_\fm$ coincide.
			\begin{proof}
				Take 
				$s=  (A, \lambda_{A}, \eta_{A}, \tilde A, \lambda_{\tilde{A}}, \eta_{\tilde{A}}, \alpha ) \in Y^{(2)}(\dF_p^\ac).$ 
		 $\tilde A$ is superspecial by Lemma \ref{lem:PY2}\eqref{Y2tA0FequalV}
		 and Lemma \ref{lem:Y012}\eqref{BtBF=V}.
 
 				Since $A$ and $\tilde A$ are superspecial, there are  
				supersingular elliptic curves  $E$ and $\tilde E$
				defined over $\dF_{p^2}$ such that
				$A=(E^{\oplus 3})\otimes \dF_p^\ac$ and
				$\tilde A=(\tilde E^{\oplus 3})\otimes \dF_p^\ac.$
				It is well known that the relative Frobenius 
				$\Fr_E: E \to E^{(p^2)}\cong E$ coincides with the isogeny
				$-p:E\to E,$ 
				and 
				$\Fr_{\tilde E}: \tilde E \to \tilde  E^{(p^2)}\cong E$ coincides with the isogeny
				$-p:\tilde E\to \tilde E.$ 
				It turns out that the action of $\Frob_{p^2}$ and $\ang{-p^{-1}}$
				on $ \rH^0(Y^{(2)}, \dF_\ell)_\fm$
				coincide. We conclude by remarking  that $\ang{-p^{-1}}=\ang{p^{-1}}.$ 
			\end{proof}
			
		\end{lem}
		
		\begin{lem}\label{lem:diamond}
			$ \phi_{\Iw_p} (\ang{p^{-1}})$ lies in the image of 
			$Z( \dA^\square)/K^\square \cap Z(\dA^\square)$
			in $\End_{\dF_\ell}( \rH^0(Y^{(2)}, \dF_\ell)_\fm).$ 
			\begin{proof} 
				Let  $\underline p \in Z(\dA^\infty)\cong (\dA_F^{\infty})^\times$ be the element whose $p$-component  is $p$ and other components are 1.
				By definition the action of $\underline p$ and $\ang{p}$ coincide.
				Since the action of $ \underline{p}^{-1}$ on $\rH^0(Y^{(2)}, \dF_\ell)_\fm$ factors through 
				$Z(\dA^\infty)/Z(\dQ)(K^p\Iw_p\cap Z(\dA^\infty)),$
				it suffices to show that
				there exist 
				$g^\square \in Z(\dA^\square)  $ and $f \in Z(\dQ),$ such that 
				$g^\square f \underline p^{-1}\in K^\square  \cap Z(\dA^\square)  ,$
				which follows from the weak approximation.
			\end{proof}
		\end{lem}
		
		\subsection{Proof of the main theorem}
		\begin{proof}[Proof of Theorem \ref{thm:main}]
			Suppose there is no level-lowering,
			i.e., there is no automorphic representation $\tilde \pi$ of $G(\dA)$ such that 
			${\tilde\pi}^{K^pK_p}\neq 0$ and
			$\ol\rho_{\tilde\pi,\ell}\cong \ol\rho_{\pi,\ell}.$
			 	By Zucker's conjecture and the Matsushima formula we have the decomposition
			\cite[1.9]{BR92}
			\begin{equation}\label{eq:IH2decom}
				\IH^{2}( {S_0(p)}\otimes{F^\ac},  \dZ_\ell )_\fm  \otimes  \dQ_\ell^\ac
				=\bigoplus_{\tilde \pi } 
			 \iota_\ell^{-1}  	\tilde \pi^{K^p\Iw_p} \otimes  \rho_{\tilde \pi,\ell}
			\end{equation}
			where
			$\tilde \pi$ runs over irreducible automorphic representations of $G(\dA)$ such that $\tilde\pi_\infty$ is cohomological with trivial coefficient
			and
		 $\ol\rho_{\tilde\pi,\ell}\cong \ol\rho_{\pi,\ell}.$
			By Corollary \ref{cor:boundaryvanish} and the absolute irreducibility of $\ol\rho_{\pi,\ell}$,
			every irreducible    Jordan-H\"older factor  			of $\rH^{2}( {S_0(p)}\otimes{F^\ac},  \dF_\ell )_\fm $
			 is isomorphic to $\ol\rho_{\pi,\ell}.$ 
			 The weight-monodromy spectral sequence,
			 which
			 degenerates at $E_1$ by 
			 Lemma \ref{lem:degenerate},
			 gives a filtration $\Fil^*
			\mathrm{H}^{2}( {S_0(p)}\otimes F^\ac,   \dF_\ell)_\fm$ on
			$ \mathrm{H}^{2}( {S_0(p)}\otimes F^\ac  ,   \dF_\ell)_\fm$
			of $\dT(G^\square,K^\square)_\fm$-modules. 
			Put
			$\Gr_p:=\Fil^{p}/\Fil^{p+1}.$
			Then by Lemma \ref{lem:cohvanish} the non-zero terms are
			\begin{align*}
				\Gr_{-2} &=\rH^{0}(Y^{(2)} , \mathbb{F}_{\ell}(-2))_{\fm}   ,\\
 				\Gr_{0} &= 
				\rH^{0}(Y^{(2)}  , \mathbb{F}_{\ell}(-1))_{\fm},\\
					\Gr_2&=\rH^{0}(Y^{(2)}   , \mathbb{F}_{\ell})_{\fm}.
			\end{align*}
		  The monodromy operator $\tilde \mu$ in Section \ref{monodromy} boils down to   identity maps $\Gr_{-2}\to \Gr_0(-1)$ and $\Gr_{0}\to \Gr_{2}(-1)$.
		  In particular, $\ker\tilde\mu=\Fil^2\cong\rH^{0}(Y^{(2)} , \mathbb{F}_{\ell})_{\fm}. $
			The unramifiedness of $\ol\rho_{\pi,\ell}$ at $p$ 
			 implies that
			$  \rH^{0}(Y^{(2)} , \mathbb{F}_{\ell}) [\fm]
			\subset  \rH^{0}(Y^{(2)} , \mathbb{F}_{\ell})_{\fm}$
			contains a copy of $\ol\rho_{\pi,\ell}\mid_{\Gal(\dF_p^\ac/\dF_{p^2})}.$ 
			 However, by Lemma \ref{lem:Frobp} and Lemma \ref{lem:diamond},
		 $\Frob_{p^2}$ acts  as the scalar 
		 $\chi_{\pi,\ell}(\underline p)^{-1}$ on
			$\rH^0(Y^{(2)}  ,\dF_\ell)_\fm[\fm]	$
			where $\chi_{\pi,\ell}:=\iota_\ell^{-1} \circ \chi_\pi$
			and $\chi_{\pi}$ is the central character.
			On the other hand, 
			the multiset of eigenvalues of $\ol\rho_{\pi,\ell}(\Frob_p)$
			is $\{p^2,1,p^{-2}\}\mod \ell$ up to multiplication by a common scalar.
			We then deduce that $p^2 \equiv 1 \mod l,$
			contradicting the assumption.
		\end{proof} 


\begin{bibdiv}
\begin{biblist}

\bib{ACC+22}{article}{
      author={Allen, Patrick~B.},
      author={Calegari, Frank},
      author={Caraiani, Ana},
      author={Gee, Toby},
      author={Helm, David},
      author={Le~Hung, Bao~V.},
      author={Newton, James},
      author={Scholze, Peter},
      author={Taylor, Richard},
      author={Thorne, Jack~A.},
       title={Potential automorphy over {CM} fields},
        date={2022},
     journal={Annals of Mathematics},
}

\bib{Bel02}{thesis}{
      author={Bella\"iche, Jo\"el},
       title={Congruences endoscopiques et repr\'esentations galoisiennes},
        type={Ph.D. Thesis},
        date={2002},
}

\bib{BG06}{article}{
      author={Bella\"{\i}che, Jo\"{e}l},
      author={Graftieaux, Phillipe},
       title={Augmentation du niveau pour {${\rm U}(3)$}},
        date={2006},
        ISSN={0002-9327},
     journal={Amer. J. Math.},
      volume={128},
      number={2},
       pages={271\ndash 309},
         url={https://mathscinet.ams.org/mathscinet-getitem?mr=2214894},
      review={\MR{2214894}},
}

\bib{Boy19}{article}{
      author={Boyer, Pascal},
       title={Principe de {M}azur en dimension sup\'{e}rieure},
        date={2019},
        ISSN={2429-7100},
     journal={J. \'{E}c. polytech. Math.},
      volume={6},
       pages={203\ndash 230},
         url={https://mathscinet.ams.org/mathscinet-getitem?mr=3959073},
      review={\MR{3959073}},
}

\bib{BR92}{incollection}{
      author={Blasius, Don},
      author={Rogawski, Jonathan~D.},
       title={Tate classes and arithmetic quotients of the two-ball},
        date={1992},
   booktitle={The zeta functions of {P}icard modular surfaces},
   publisher={Univ. Montr\'{e}al, Montreal, QC},
       pages={421\ndash 444},
         url={https://mathscinet.ams.org/mathscinet-getitem?mr=1155236},
      review={\MR{1155236}},
}

\bib{BS73}{article}{
      author={Borel, A.},
      author={Serre, J.-P.},
       title={Corners and arithmetic groups},
        date={1973},
        ISSN={0010-2571},
     journal={Comment. Math. Helv.},
      volume={48},
       pages={436\ndash 491},
         url={https://mathscinet.ams.org/mathscinet-getitem?mr=387495},
      review={\MR{387495}},
}

\bib{BT72}{article}{
      author={Bruhat, F.},
      author={Tits, J.},
       title={Groupes r\'{e}ductifs sur un corps local},
        date={1972},
        ISSN={0073-8301},
     journal={Inst. Hautes \'{E}tudes Sci. Publ. Math.},
      number={41},
       pages={5\ndash 251},
         url={https://mathscinet.ams.org/mathscinet-getitem?mr=327923},
      review={\MR{327923}},
}

\bib{BW06}{article}{
      author={B\"{u}ltel, Oliver},
      author={Wedhorn, Torsten},
       title={Congruence relations for {S}himura varieties associated to some
  unitary groups},
        date={2006},
        ISSN={1474-7480},
     journal={J. Inst. Math. Jussieu},
      volume={5},
      number={2},
       pages={229\ndash 261},
         url={https://mathscinet.ams.org/mathscinet-getitem?mr=2225042},
      review={\MR{2225042}},
}

\bib{Car79}{inproceedings}{
      author={Cartier, P.},
       title={Representations of {$p$}-adic groups: a survey},
        date={1979},
   booktitle={Automorphic forms, representations and {$L$}-functions ({P}roc.
  {S}ympos. {P}ure {M}ath., {O}regon {S}tate {U}niv., {C}orvallis, {O}re.,
  1977), {P}art 1},
      series={Proc. Sympos. Pure Math., XXXIII},
   publisher={Amer. Math. Soc., Providence, R.I.},
       pages={111\ndash 155},
         url={https://mathscinet.ams.org/mathscinet-getitem?mr=546593},
      review={\MR{546593}},
}

\bib{Car86}{article}{
      author={Carayol, Henri},
       title={Sur les repr\'{e}sentations {$l$}-adiques associ\'{e}es aux
  formes modulaires de {H}ilbert},
        date={1986},
        ISSN={0012-9593},
     journal={Ann. Sci. \'{E}cole Norm. Sup. (4)},
      volume={19},
      number={3},
       pages={409\ndash 468},
         url={https://mathscinet.ams.org/mathscinet-getitem?mr=870690},
      review={\MR{870690}},
}

\bib{Cho94}{article}{
      author={Choucroun, Francis~M.},
       title={Analyse harmonique des groupes d'automorphismes d'arbres de
  {B}ruhat-{T}its},
        date={1994},
        ISSN={0037-9484},
     journal={M\'{e}m. Soc. Math. France (N.S.)},
      number={58},
       pages={170},
         url={https://mathscinet.ams.org/mathscinet-getitem?mr=1294542},
      review={\MR{1294542}},
}

\bib{Clo00}{article}{
      author={Clozel, L.},
       title={On {R}ibet's level-raising theorem for {$\rm U(3)$}},
        date={2000},
        ISSN={0002-9327},
     journal={Amer. J. Math.},
      volume={122},
      number={6},
       pages={1265\ndash 1287},
         url={https://mathscinet.ams.org/mathscinet-getitem?mr=1797662},
      review={\MR{1797662}},
}

\bib{CS19}{unpublished}{
      author={Caraiani, Ana},
      author={Scholze, Peter},
       title={On the generic part of the cohomology of non-compact unitary
  {S}himura varieties},
        date={2019},
        note={arXiv:1909.01898},
}

\bib{dSG18}{incollection}{
      author={de~Shalit, Ehud},
      author={Goren, Eyal~Z.},
       title={On the bad reduction of certain {$U(2,1)$} {S}himura varieties},
        date={2018},
   booktitle={Geometry, algebra, number theory, and their information
  technology applications},
      series={Springer Proc. Math. Stat.},
      volume={251},
   publisher={Springer, Cham},
       pages={81\ndash 152},
         url={https://mathscinet.ams.org/mathscinet-getitem?mr=3880385},
      review={\MR{3880385}},
}

\bib{Hel06}{unpublished}{
      author={Helm, David},
       title={Mazur's principle for {U}(2,1) {S}himura varieties},
        date={2006},
        note={arXiv:1709.03731},
}

\bib{HLR86}{article}{
      author={Harder, G.},
      author={Langlands, R.~P.},
      author={Rapoport, M.},
       title={Algebraische {Z}yklen auf {H}ilbert-{B}lumenthal-{F}l\"{a}chen},
        date={1986},
        ISSN={0075-4102},
     journal={J. Reine Angew. Math.},
      volume={366},
       pages={53\ndash 120},
         url={https://mathscinet.ams.org/mathscinet-getitem?mr=833013},
      review={\MR{833013}},
}

\bib{Jar99}{article}{
      author={Jarvis, Frazer},
       title={Mazur's principle for totally real fields of odd degree},
        date={1999},
        ISSN={0010-437X},
     journal={Compositio Math.},
      volume={116},
      number={1},
       pages={39\ndash 79},
         url={https://mathscinet.ams.org/mathscinet-getitem?mr=1669444},
      review={\MR{1669444}},
}

\bib{Kat81}{incollection}{
      author={Katz, N.},
       title={Serre-{T}ate local moduli},
        date={1981},
   booktitle={Algebraic surfaces ({O}rsay, 1976--78)},
      series={Lecture Notes in Math.},
      volume={868},
   publisher={Springer, Berlin-New York},
       pages={138\ndash 202},
         url={https://mathscinet.ams.org/mathscinet-getitem?mr=638600},
      review={\MR{638600}},
}

\bib{Kni01}{article}{
      author={Knightly, Andrew~H.},
       title={Galois representations attached to representations of {${\rm
  GU}(3)$}},
        date={2001},
        ISSN={0025-5831},
     journal={Math. Ann.},
      volume={321},
      number={2},
       pages={375\ndash 398},
         url={https://mathscinet.ams.org/mathscinet-getitem?mr=1866493},
      review={\MR{1866493}},
}

\bib{Liu19}{article}{
      author={Liu, Yifeng},
       title={Bounding cubic-triple product {S}elmer groups of elliptic
  curves},
        date={2019},
        ISSN={1435-9855},
     journal={J. Eur. Math. Soc. (JEMS)},
      volume={21},
      number={5},
       pages={1411\ndash 1508},
         url={https://mathscinet.ams.org/mathscinet-getitem?mr=3941496},
      review={\MR{3941496}},
}

\bib{LS18}{incollection}{
      author={Lan, Kai-Wen},
      author={Stroh, Beno\^{\i}t},
       title={Nearby cycles of automorphic \'{e}tale sheaves, {II}},
        date={2018},
   booktitle={Cohomology of arithmetic groups},
      series={Springer Proc. Math. Stat.},
      volume={245},
   publisher={Springer, Cham},
       pages={83\ndash 106},
         url={https://mathscinet.ams.org/mathscinet-getitem?mr=3848816},
      review={\MR{3848816}},
}

\bib{LTX+22}{article}{
      author={Liu, Yifeng},
      author={Tian, Yichao},
      author={Xiao, Liang},
      author={Zhang, Wei},
      author={Zhu, Xinwen},
       title={On the {B}eilinson-{B}loch-{K}ato conjecture for
  {R}ankin-{S}elberg motives},
        date={2022},
        ISSN={0020-9910},
     journal={Invent. Math.},
      volume={228},
      number={1},
       pages={107\ndash 375},
         url={https://mathscinet.ams.org/mathscinet-getitem?mr=4392458},
      review={\MR{4392458}},
}

\bib{Mes72}{book}{
      author={Messing, William},
       title={The crystals associated to {B}arsotti-{T}ate groups: with
  applications to abelian schemes},
      series={Lecture Notes in Mathematics, Vol. 264},
   publisher={Springer-Verlag, Berlin-New York},
        date={1972},
         url={https://mathscinet.ams.org/mathscinet-getitem?mr=0347836},
      review={\MR{0347836}},
}

\bib{Mil80}{book}{
      author={Milne, James~S.},
       title={\'{E}tale cohomology},
      series={Princeton Mathematical Series, No. 33},
   publisher={Princeton University Press, Princeton, N.J.},
        date={1980},
        ISBN={0-691-08238-3},
         url={https://mathscinet.ams.org/mathscinet-getitem?mr=559531},
      review={\MR{559531}},
}

\bib{NT16}{article}{
      author={Newton, James},
      author={Thorne, Jack~A.},
       title={Torsion {G}alois representations over {CM} fields and {H}ecke
  algebras in the derived category},
        date={2016},
     journal={Forum Math. Sigma},
      volume={4},
       pages={Paper No. e21, 88},
         url={https://mathscinet.ams.org/mathscinet-getitem?mr=3528275},
      review={\MR{3528275}},
}

\bib{Raj01}{article}{
      author={Rajaei, Ali},
       title={On the levels of mod {$l$} {H}ilbert modular forms},
        date={2001},
        ISSN={0075-4102},
     journal={J. Reine Angew. Math.},
      volume={537},
       pages={33\ndash 65},
         url={https://mathscinet.ams.org/mathscinet-getitem?mr=1856257},
      review={\MR{1856257}},
}

\bib{Rib90}{article}{
      author={Ribet, K.~A.},
       title={On modular representations of {${\rm Gal}(\overline{\bf Q}/{\bf
  Q})$} arising from modular forms},
        date={1990},
        ISSN={0020-9910},
     journal={Invent. Math.},
      volume={100},
      number={2},
       pages={431\ndash 476},
         url={https://mathscinet.ams.org/mathscinet-getitem?mr=1047143},
      review={\MR{1047143}},
}

\bib{Rib91}{article}{
      author={Ribet, Kenneth~A.},
       title={Lowering the levels of modular representations without
  multiplicity one},
        date={1991},
        ISSN={1073-7928},
     journal={Internat. Math. Res. Notices},
      number={2},
       pages={15\ndash 19},
         url={https://mathscinet.ams.org/mathscinet-getitem?mr=1104839},
      review={\MR{1104839}},
}

\bib{Rog90}{book}{
      author={Rogawski, Jonathan~D.},
       title={Automorphic representations of unitary groups in three
  variables},
      series={Annals of Mathematics Studies},
   publisher={Princeton University Press, Princeton, NJ},
        date={1990},
      volume={123},
        ISBN={0-691-08586-2; 0-691-08587-0},
         url={https://mathscinet.ams.org/mathscinet-getitem?mr=1081540},
      review={\MR{1081540}},
}

\bib{Rog92}{incollection}{
      author={Rogawski, Jonathan~D.},
       title={Analytic expression for the number of points mod {$p$}},
        date={1992},
   booktitle={The zeta functions of {P}icard modular surfaces},
   publisher={Univ. Montr\'{e}al, Montreal, QC},
       pages={65\ndash 109},
         url={https://mathscinet.ams.org/mathscinet-getitem?mr=1155227},
      review={\MR{1155227}},
}

\bib{RZ82}{article}{
      author={Rapoport, M.},
      author={Zink, Th.},
       title={\"{U}ber die lokale {Z}etafunktion von {S}himuravariet\"{a}ten.
  {M}onodromiefiltration und verschwindende {Z}yklen in ungleicher
  {C}harakteristik},
        date={1982},
        ISSN={0020-9910},
     journal={Invent. Math.},
      volume={68},
      number={1},
       pages={21\ndash 101},
         url={https://mathscinet.ams.org/mathscinet-getitem?mr=666636},
      review={\MR{666636}},
}

\bib{Sai03}{article}{
      author={Saito, Takeshi},
       title={Weight spectral sequences and independence of {$l$}},
        date={2003},
        ISSN={1474-7480},
     journal={J. Inst. Math. Jussieu},
      volume={2},
      number={4},
       pages={583\ndash 634},
         url={https://mathscinet.ams.org/mathscinet-getitem?mr=2006800},
      review={\MR{2006800}},
}

\bib{Ser87b}{incollection}{
      author={Serre, J.-P.},
       title={Lettre {\`a} {J}.-{F}. {M}estre},
        date={1987},
   booktitle={Current trends in arithmetical algebraic geometry ({A}rcata,
  {C}alif., 1985)},
      series={Contemp. Math.},
      volume={67},
   publisher={Amer. Math. Soc., Providence, RI},
       pages={263\ndash 268},
         url={https://mathscinet.ams.org/mathscinet-getitem?mr=902597},
      review={\MR{902597}},
}

\bib{Ser87a}{article}{
      author={Serre, Jean-Pierre},
       title={Sur les repr\'{e}sentations modulaires de degr\'{e} {$2$} de
  {${\rm Gal}(\overline{\bf Q}/{\bf Q})$}},
        date={1987},
        ISSN={0012-7094},
     journal={Duke Math. J.},
      volume={54},
      number={1},
       pages={179\ndash 230},
         url={https://mathscinet.ams.org/mathscinet-getitem?mr=885783},
      review={\MR{885783}},
}

\bib{Sta}{misc}{
      author={{The Stacks Project Authors}},
       title={\textit{Stacks Project}},
        date={2018},
}

\bib{Tit79}{inproceedings}{
      author={Tits, J.},
       title={Reductive groups over local fields},
        date={1979},
   booktitle={Automorphic forms, representations and {$L$}-functions ({P}roc.
  {S}ympos. {P}ure {M}ath., {O}regon {S}tate {U}niv., {C}orvallis, {O}re.,
  1977), {P}art 1},
      series={Proc. Sympos. Pure Math., XXXIII},
   publisher={Amer. Math. Soc., Providence, R.I.},
       pages={29\ndash 69},
         url={https://mathscinet.ams.org/mathscinet-getitem?mr=546588},
      review={\MR{546588}},
}

\bib{vH21}{article}{
      author={van Hoften, Pol},
       title={A geometric {J}acquet-{L}anglands correspondence for paramodular
  {S}iegel threefolds},
        date={2021},
        ISSN={0025-5874},
     journal={Math. Z.},
      volume={299},
      number={3-4},
       pages={2029\ndash 2061},
         url={https://mathscinet.ams.org/mathscinet-getitem?mr=4329279},
      review={\MR{4329279}},
}

\bib{Vol10}{article}{
      author={Vollaard, Inken},
       title={The supersingular locus of the {S}himura variety for {${\rm
  GU}(1,s)$}},
        date={2010},
        ISSN={0008-414X},
     journal={Canad. J. Math.},
      volume={62},
      number={3},
       pages={668\ndash 720},
         url={https://mathscinet.ams.org/mathscinet-getitem?mr=2666394},
      review={\MR{2666394}},
}

\bib{VW11}{article}{
      author={Vollaard, Inken},
      author={Wedhorn, Torsten},
       title={The supersingular locus of the {S}himura variety of {${\rm
  GU}(1,n-1)$} {II}},
        date={2011},
        ISSN={0020-9910},
     journal={Invent. Math.},
      volume={184},
      number={3},
       pages={591\ndash 627},
         url={https://mathscinet.ams.org/mathscinet-getitem?mr=2800696},
      review={\MR{2800696}},
}

\bib{Wan22}{article}{
      author={Wang, Haining},
       title={Level lowering on siegel modular threefold of paramodular level},
        date={2022},
     journal={arXiv:1910.07569},
}

\bib{Wed01}{incollection}{
      author={Wedhorn, Torsten},
       title={The dimension of {O}ort strata of {S}himura varieties of
  {PEL}-type},
        date={2001},
   booktitle={Moduli of abelian varieties ({T}exel {I}sland, 1999)},
      series={Progr. Math.},
      volume={195},
   publisher={Birkh\"{a}user, Basel},
       pages={441\ndash 471},
         url={https://mathscinet.ams.org/mathscinet-getitem?mr=1827029},
      review={\MR{1827029}},
}

\end{biblist}
\end{bibdiv}
	\end{document}